\immediate\write18{makeindex final.nlo -s nomencl.ist -o final.nls}

\documentclass[12pt]{amsart}
\usepackage{amsfonts,amssymb,cite,amstext,amsbsy,bm}
\usepackage{amsthm}
\usepackage{amsmath}
 
\usepackage{paralist}
\usepackage{datetime}
\usepackage{color}
\usepackage[dvipsnames]{xcolor}
\usepackage[colorlinks,linkcolor=Bittersweet,citecolor=Cerulean]{hyperref}
\usepackage{mathrsfs}
\usepackage{mathtools}
\usepackage{euscript}
\usepackage[all]{xy}
\usepackage{dsfont}
\usepackage{graphicx}

\usepackage[text={170mm,255mm},centering]{geometry}

 \usepackage{nomencl}
 \makenomenclature

\usepackage{cite} 

\usepackage{thmtools}
\usepackage{enumitem}
\usepackage{letltxmacro}
\usepackage{comment}

\declaretheorem[
name=Theorem,
Refname={Theorem,Theorems},
numberwithin=section]{thm}
\declaretheorem[
name=Proposition,
Refname={Proposition,Propositions},
sibling=thm]{proposition}
\declaretheorem[
name=Definition,
Refname={Definition,Definitions},
sibling=thm]{dfn}
\declaretheorem[
name=Corollary,
Refname={Corollary,Corollaries},
sibling=thm]{cor}

\declaretheorem[
name=Lemma,
Refname={Lemma,Lemmas},
sibling=thm]{lem}

\declaretheorem[
name=Conjecture,
Refname={Conjecture,Conjectures},
sibling=thm]{conjecture}

\renewcommand{\thethmx}{\Alph{thmx}}

\newtheorem*{main}{Main Theorem}

\usepackage{nameref}
\usepackage[capitalize]{cleveref}

\makeatletter     
\newcommand{\crefnames}[3]{
	\@for\next:=#1\do{
		\expandafter\crefname\expandafter{\next}{#2}{#3}
	}
}
\makeatother

\crefnames{part,chapter,section}{\S}{\S\S}

\theoremstyle{plain}
\newlist{thmlist}{enumerate}{1}
\setlist[thmlist]{leftmargin=0.6cm, label= {\rm (\roman{thmlisti})}, ref=\thethm.(\roman{thmlisti}),noitemsep}

\addtotheorempostheadhook[thm]{\crefalias{thmlisti}{thm}}

\addtotheorempostheadhook[proposition]{\crefalias{thmlisti}{proposition}}

\addtotheorempostheadhook[dfn]{\crefalias{thmlisti}{dfn}}

\addtotheorempostheadhook[lem]{\crefalias{thmlisti}{lem}}

\newlist{thmenum}{enumerate}{1} 
\setlist[thmenum]{label=(\roman{thmenumi}),ref=\thethmx.(\roman{thmenumi}),noitemsep}
\crefalias{thmenumi}{thmx}

\newlist{mainenum}{enumerate}{1} 
\setlist[mainenum]{labelindent=0pt,labelwidth=1.25em,leftmargin=!, label={\rm (\roman{mainenumi})}, ref=\themain (\roman{mainenumi}) }
\crefalias{mainenumi}{main} 

\addtotheorempostheadhook[main]{\crefalias{mainenum}{main}}

 \newtheorem{rem}[thm]{Remark}

 \def\oc{\mathscr{O}}

 \def\jr{\mathscr{J}}
 \def\mr{\mathscr{M}}

 \def\af{\mathbf{a}}
 
 \def\pr{{\rm pr}}
 \def\nd{\nabla_{\!\! D}}
 \def\onu{ {\nabla}_{\!\! \mathfrak{U}}}
 
 \def\cb{\mathbb{C}}
 \def\ib{\mathbb{I}}
 \def\ab{\mathbb{A}}
 \def\lbb{\mathbb{L}}

 \def\Im{\operatorname{Im}}

 \def\tl{\widetilde}  
 
 \def\cb{\mathbb{C}} \def\nb{\mathbb{N}}
 \def\gb{\mathbb{G}} 
 \def\pb{\mathbb{P}}  \def\rb{\mathbb{R}}
 
 \def\ys{\mathscr{Y}}
 \def\ls{\mathscr{L}}
 \def\ds{\mathscr{D}}
 \def\hs{\mathscr{H}}

 \def\uk{\mathfrak{U}}
 \def\wk{\mathfrak{w}}

  \def\d{\partial}

 \def\pt{{\rm P}}

\crefname{lem}{Lemma}{Lemmas}
\crefname{thm}{Theorem}{Theorems}
\crefname{proposition}{Proposition}{Propositions}
\crefname{dfn}{Definition}{Definitions}
\crefname{rem}{Remark}{Remarks}
\crefname{cor}{Corollary}{Corollaries}
\crefname{corx}{Corollary}{Corollaries}
\crefname{problem}{Problem}{Problems}
\crefname{thmx}{Theorem}{Theorems}
\crefname{conjecture}{Conjecture}{Conjectures}

\crefname{main}{Main Theorem}{Main Theorems}

\makeatletter
\newcommand*{\rom}[1]{\expandafter\@slowromancap\romannumeral #1@}
\makeatother

\makeatletter
\newsavebox{\@brx}
\newcommand{\llangle}[1][]{\savebox{\@brx}{\(\m@th{#1\langle}\)}
  \mathopen{\copy\@brx\kern-0.5\wd\@brx\usebox{\@brx}}}
\newcommand{\rrangle}[1][]{\savebox{\@brx}{\(\m@th{#1\rangle}\)}
  \mathclose{\copy\@brx\kern-0.5\wd\@brx\usebox{\@brx}}}
\makeatother

\numberwithin{equation}{section}

\definecolor{plum}{rgb}{0.8,0.2,0.8}

\let\oldsection\section
\renewcommand{\section}{
	\renewcommand{\theequation}{\thesection.\arabic{equation}}
	\oldsection}
\let\oldsubsection\subsection
\renewcommand{\subsection}{
	\renewcommand{\theequation}{\thesubsection.\arabic{equation}}
	\oldsubsection}

\title[On the logarithmic-orbifold  Kobayashi Conjecture]{Kobayashi Hyperbolicity of the complements  
\\	of general hypersurfaces of  high degree}
\author{Damian Brotbek}
\address{Universit\'e de Strasbourg, Institut de Recherche Math\'ematique Avanc\'ee, 	7 Rue Ren\'e-Descartes,
	67084 Strasbourg, France} 
\curraddr{Centre de math\'ematiques Laurent Schwartz, \'Ecole polytechnique,	91128 Palaiseau CEDEX, France
}
\email{\quad damian.brotbek@polytechnique.edu}
\urladdr{\quad https://sites.google.com/site/damianbrotbek/home}

\author{Ya Deng} 
\address{Universit\'e de Strasbourg, Institut de Recherche Math\'ematique Avanc\'ee, 	7 Rue Ren\'e-Descartes,
	67084 Strasbourg, France} 

\curraddr{
	Department of Mathematical Sciences, Chalmers University of Technology and the University of Gothenburg, Sweden}
\email{\quad yade@chalmers.se\quad dengya.math@gmail.com}
\urladdr{\quad https://sites.google.com/site/dengyamath}
\begin{document}
	\begin{abstract}
		In this paper,  we prove that in any projective manifold, the complements of general hypersurfaces of sufficiently large degree are Kobayashi hyperbolic. We also provide an effective lower bound on the degree. This confirms a conjecture by S. Kobayashi in 1970. Our proof, based on the theory of jet differentials, is obtained by reducing the problem to the construction of a particular example with strong hyperbolicity properties. This approach relies  the construction of higher order logarithmic  connections  allowing us to construct logarithmic Wronskians. These logarithmic Wronskians are the building blocks of the more general logarithmic jet differentials we are able to construct.
		
		As a byproduct of our proof, we prove a more general result on the orbifold hyperbolicity for generic geometric orbifolds  in the sense of Campana, with only one component and large multiplicities. We   also establish a Second Main Theorem type result for holomorphic entire curves  intersecting general hypersurfaces,  and  we   prove the Kobayashi hyperbolicity of the cyclic cover of a general hypersurface, again with an explicit  lower bound on the degree of all these hypersurfaces. 
	\end{abstract}
\keywords{Kobayashi hyperbolicity, orbifold hyperbolicity, logarithmic-orbifold Kobayashi conjecture, Second Main Theorem, jet differentials, logarithmic Demailly tower, higher order log connections, logarithmic Wronskians}
	\date{\today}
\maketitle

\section{Introduction}\label{intro}
 A  complex space $X$  is said to be \emph{Kobayashi hyperbolic}  if  the (intrinsically defined) Kobayashi pseudo distance $d_X$ is a distance, meaning that $d_X(p,q)>0$ for $p\neq q$ in $X$.  One can easily see that  a Kobayashi hyperbolic complex space $X$  does not contain any non-constant entire  holomorphic curve
 $f:\cb\to X$  (this last property is called   \emph{Brody hyperbolicity}). When $X$ is compact, by a well-known theorem of  Brody \cite{Bro78}, these two definitions of hyperbolicity  are equivalent. However, in general, we have many examples of complex manifolds  which  are Brody hyperbolic but not hyperbolic in the sense of Kobayashi, see for instance \cite{Kob98}.

In 1970, Kobayashi  made the  following conjecture \cite{Kob70}, which is often called the  \emph{logarithmic Kobayashi conjecture} in the literatures.
\begin{conjecture}[Kobayashi]\label{conj:Kobayashi}
The complement $\pb^n\setminus D$ of a general hypersurface $D\subset \pb^n$ 
of   sufficiently large degree   $d\geqslant d_n$  is Kobayashi  hyperbolic.
\end{conjecture}

As is well known, \cref{conj:Kobayashi}  is simpler  to approach  when  $D$ is replaced by a simple normal crossing divisor  with several components. When  $D=\sum_{i=1}^{2n+1}H_i$ with  $\{H_i\}_{i=1,\ldots,2n+1}$  hyperplanes of  $\pb^n$  in general position, it was proved by Fujimoto \cite{Fuj72} and Green \cite{Gre77} that $\pb^n\setminus D$ is Kobayashi hyperbolic.  More generally, Noguchi-Winkelmann-Yamanoi  \cite{NWY07,NWY08,NWY13}  and Lu-Winkelmann \cite{L-W12}  even proved a stronger result towards the logarithmic Green-Griffiths conjecture: if \((Y,D)\) is a pair of log general type with logarithmic irregularity \(h^0(Y,\Omega_Y(\log D))\geqslant \dim Y\), then  \((Y,D)\) is \emph{weakly hyperbolic}.  Here we say a log pair $(Y,D)$ is {weakly hyperbolic} if all entire curves in $Y\setminus D$ lie in a proper subvariety $Z\subsetneq Y$.  When the logarithmic irregularity is strictly smaller than the dimension
of the manifold, or  equivalently the number
of irreducible components of $D$  are less or equal than the  dimension
of the manifold,  much less is known for the general logarithmic Green-Griffiths conjecture.  In \cite{Rou03,Rou09}   Rousseau dealt with the Kobayashi hyperbolicity of  $\pb^2\setminus D$ where $D$ consists of two irreducible curves of certain degrees. More recently, in \cite{BD17} we proved a more general result  concerning the hyperbolicity of the complement of  a sufficiently ample divisor with several components.
\begin{thm}[\!\!\cite{BD17}]\label{log Debarre}
	Let  $Y$ be a smooth projective variety of dimension $n$ and let \(c\geqslant n\).  Let \(L\) be a very ample line bundle on \(Y\). For any \(m\geqslant (4n)^{n+2}\) and  for general  hypersurfaces \(H_1,\dots,H_c\in |L^m|\), writing \(D=\sum_{i=1}^cH_i\), the  logarithmic cotangent bundle $\Omega_Y(\log D)$ is \emph{almost ample}. In particular, $Y\setminus D$ is Kobayashi hyperbolic and hyperbolically embedded into $Y$.
\end{thm}
This result can be seen as a logarithmic analogue of a conjecture of Debarre, which was established by the first author and Darondeau in \cite{BD15} and independently by Xie \cite{Xie15}.\\

Let us now focus on the case of one component as in \cref{conj:Kobayashi}. In the case $n=2$, the first proof to \cref{conj:Kobayashi}  was provided by Siu-Yeung \cite{SY96}, with a very high degree bound, which  was later improved  to $d\geqslant 15$ by El Goul  \cite{EG03} and $d\geqslant 14$ by Rousseau  \cite{Rou09}.   Building on ideas of  Voisin \cite{Voi96,Voi98},  Siu \cite{Siu04},  Diverio-Merker-Rousseau \cite{DMR10}, the first step towards the general case  in \cref{conj:Kobayashi} was made by  Darondeau  in \cite{Dar16c}, in which he proved the \emph{weak hyperbolicity} of $\pb^n\setminus D$ for general hypersurfaces $D$ of degree $d\geqslant (5n)^2n^n$.   Very recently, based on  his strategy outlined in \cite{Siu04},  in \cite{Siu15} Siu made an important progress towards \cref{conj:Kobayashi}, in which he showed that $\pb^n\setminus D$ is \emph{Brody hyperbolic} for $D$ a general hypersurface of degree $d\geqslant d_n^*$, where $d_n^*$ is some (non-explicit)  function  depending on $n$.

The goal of the present  paper  is to prove \cref{conj:Kobayashi} with an effective estimate on the lower degree bound $d_n$. We also prove    a  \emph{Second Main Theorem type result}, \emph{orbifold hyperbolicity} of a general orbifold with \emph{one component}  and high multiplicity, and Kobayashi hyperbolicity of the cyclic cover of a general hypersurface of large degree.
\begin{main}[=\cref{effective estimate,effective estimate2,orbifold Kobayashi}]\label{main}
	Let $Y$ be a smooth projective variety of dimension $n\geqslant 2$.  Fix any very ample line bundle $A$ on $Y$. Then for a \emph{general} smooth hypersurface \(D\in |A^d| \) with   	\[d\geqslant (n+2)^{n+3}(n+1)^{n+3}\sim_{n\to \infty}e^3n^{2n+6},\]
	\begin{mainenum} 
		\item\label{main log Kobayashi}   	The complement   $Y\setminus D$ is  hyperbolically embedded into $Y$. In particular, $Y\setminus D$ is Kobayashi hyperbolic.
		\item \label{main SMT} For   any holomorphic entire curve  (possibly algebraically degenerate) \(f:\cb\rightarrow Y\) which is not contained in \(D\), one has 
		\begin{eqnarray*}
			T_f(r,A)\leqslant N^{(1)}_f(r,D)+C\big(\log T_f(r,A)+\log r\big) \quad \lVert.
		\end{eqnarray*}	
	Here $T_f(r,A)$ is the Nevanlinna order function, $N^{(1)}_f(r,D)$ is the truncated counting function, and the symbol \(\lVert\) means that	the  inequality holds outside a Borel subset of \((1,+\infty )\) of finite Lebesgue measure.
	\item\label{main orbifold} The (Campana) orbifold $\big(Y,(1-\frac{1}{d})D\big)$ is orbifold hyperbolic, \emph{i.e.} there exists no entire curve  $f:\cb \to Y$ so that $$ f(\cb)\not\subset D \quad {\rm with} \quad {\rm mult}_{t}(f^*D)\geqslant d \quad \forall t\in f^{-1}(D). $$  		\item   \label{main cyclic} Let $\pi:X\to Y$ be the cyclic cover of $Y$ obtained by taking the $d$-th root along $D$. Then $X$ is Kobayashi hyperbolic.
	\end{mainenum} 
\end{main} 
To the best of our knowledge, \cref{main orbifold} is the first general result on the \emph{orbifold Kobayashi conjecture}   \cite[Conjecture 5.5]{Rou10} dealing with general orbifolds with only one component.  We note that \cref{main log Kobayashi} immediately follows  from \cref{main orbifold}  in view of the definition of orbifold hyperbolicity and our previous results \cite{Bro17,Den17}.   We also observe  that \cref{main orbifold} implies  \cref{main cyclic} since $\pi:(X,\varnothing)\to \big(Y,(1-\frac{1}{d})D\big) $ is a $d$-folded \emph{unramified  cover}  in the category of orbifold.  The only result of the type of  \cref{main cyclic} we are aware of is due to Roulleau-Rousseau \cite{RR13}, who proved that for a   very general hypersurface $D$ in $\pb^n$ of  degree $d\geqslant 2n+2$,    the cyclic cover $X$ of $Y$ obtained by taking the $d$-th root along $D$ is \emph{algebraically hyperbolic}.   \\

Let us mention that in a recent preprint \cite{RY18}, which appeared after the first version of the present paper was made publicly available,   Riedl-Yang provide a short proof of \cref{conj:Kobayashi} with an effective bound on \(d_n\) (which is slightly worse than the bound we give here). However, their proof relies heavily on a series of work by Darondeau \cite{Dar16b,Dar16c,Dar16a} whereas our proof is essentially self-contained.  \\

Our approach is inspired by our previous works \cite{Bro17,Den17,BD17}. Those works were motivated by the compact counterpart of  \emph{Kobayashi conjecture}, also conjectured  in \cite{Kob70} by Kobayashi: \emph{a  general hypersurface $X\subset \pb^n$  of   sufficiently high degree   $d\geqslant d'_n$  is Kobayashi  hyperbolic}. There are now several proofs of this result, \cite{Siu15,Bro17} and more recently \cite{Dem18}. Here we will provide a logarithmic counterpart to the approach of \cite{Bro17} as well as the work \cite{Den17}.\\

Let us now outline the main points of the proof of our main result. First we observe that the first statement of our main result will follow from the Brody hyperbolicity of \(Y\setminus D\) in view of a theorem of Green \cite{Gre77} and the results established in \cite{Bro17,Den17}. In order to control the entire curves in \(Y\setminus D\) we rely on the theory of logarithmic jet differentials. Logarithmic jet differentials on the pair \((Y,D)\) are higher order generalizations of symmetric differential forms with logarithmic poles along \(D\) and provide obstructions to the existence of entire curves. Roughly speaking, in order to prove that \(Y\setminus D\) is Brody hyperbolic it suffices to construct many logarithmic jet differential forms on  \((Y,D)\) vanishing along some ample divisor and control their geometry. Let us also observe that in general it is critical to use higher order jet differentials and not merely  logarithmic symmetric differential forms. In general one has to go at least to order \(k=\dim Y\)  (see \emph{e.g.} \cite[Theorem 8]{Div09}).\\

 Let us now explain the approach we use to construct logarithmic jet differential forms. For simplicity, we suppose until the end of this section that \(Y=\pb^n\) and that \(A=\oc_{\pb^n}(1)\). The first step is to introduce higher order logarithmic connections. More precisely for any integer \(d\geqslant 1 \), any smooth \(D\in |\oc_{\pb^n}(d)|\) and any \(k\geqslant 0\) we define the $k$-th order logarithmic connection associated to the pair $(\pb^n,D)$
   \begin{align*}
     \nabla^k_{\!\! D}:\oc_{\pb^n}(d)\to E_{k,k}^{\rm GG}\Omega_{\pb^n}(\log D)\otimes \oc_{\pb^n}(d)  
   \end{align*}
     by  setting
   \begin{align}\label{intro:higher order}
   \nabla^k_{\!\! D}s=\sigma d^k\! \! \left(\frac{s}{\sigma}\right)
   \end{align}
    where $D=(\sigma=0)$ and $s\in \Gamma(U,\oc_{\pb^n}(d))$ for some open set $U\subset \pb^n$.  Here \(E_{k,k}^{\rm GG}\Omega_{\pb^n}(\log D)\) denotes the vector bundle of logarithmic jet differentials of order \(k\) and weighted degree \(k\). See  \cref{sec:LogJetSpace} and \cref{sec:connection} for more details.  The crucial point is the following \emph{tautological equality} for any \(k\geqslant 1\) 
    \begin{equation}\label{eq:tauto}
    \nabla^k_{\!\! D}\sigma=0.
    \end{equation}

    \medskip
    
    Next, we follow the general strategy in \cite{BD15,Bro17,BD17} which consists in reducing the general case to the construction of  a particular example \((\pb^n,D)\) satisfying a certain ampleness property, which implies  Brody hyperbolicity and which is a Zariski open property.      
    Such examples are   indeed  suitable deformations of Fermat type hypersurfaces.     For some suitably chosen parameters \(\delta,r,\varepsilon,k\in \mathbb{N}^*\), consider  the hypersurface \(D_\af  \subset \pb^n\) defined by a polynomial  of degree $d:=\varepsilon+(r+k)\delta$ in the form 
   \begin{align}\label{intro:Fermat}
   F(\af)= \sum_{\substack{I=(i_0,\ldots,i_n)\\ i_0+\cdots+i_n=\delta}}a_{ I}z^{(r+k)I},
   \end{align}
    where  we use  the multi-index notation \(z^{(r+1)I}=(z_0^{i_0}\cdots z_n^{i_n})^{r+k}\) for \(I=(i_0,\dots,i _n)\) and homogeneous coordinates \([z_0,\dots, z_n]\) on \(\pb^n\),  and the \(a_{I}\)'s are   homogeneous polynomials of degree \(\varepsilon\geqslant 1\) in \(\cb[z_0,\dots,z_n]\).  Write \(D_{\af}=(F(\af)=0)\subset \pb^n\). By considering the tautological relation \eqref{eq:tauto} for any \(j=1,\dots, k\) we obtain 
   the following equalities
    \begin{equation}\label{intro:equations}
   \left \{
    \begin{array}{cccccc} 
     &0  &= \nabla^1_{\!\!D_\af}(F_\af)  &=\sum_{|I|=\delta }\tilde{\alpha}_{I,1}z^{(r+k-1)I}  &=\sum_{|I|=\delta}\alpha_{I,1}z^{rI} &=\sum_{|I|=\delta}\alpha_{I,1}T^{I}\\
    &0  &= \nabla^2_{\!\!D_\af}(F_\af)  &=\sum_{|I|=\delta }\tilde{\alpha}_{I,2}z^{(r+k-2)I}  &=\sum_{|I|=\delta}\alpha_{I,2}z^{rI} &=\sum_{|I|=\delta}\alpha_{I,2}T^{I}\\
 &\vdots    &\vdots     &\vdots  &\vdots &\vdots  \\
   & 0  &= \nabla^k_{\!\!D_\af}(F_\af)  &=\sum_{|I|=\delta }\tilde{\alpha}_{I,k}z^{(r+k-k)I}  &=\sum_{|I|=\delta}\alpha_{I,k}z^{rI} &=\sum_{|I|=\delta}\alpha_{I,k}T^{I} 
    \end{array} \right.
    \end{equation}
    Here \((T_0,\dots, T_n):=(z_0^r,\dots,z_n^r)\), and for each \(I\) and $i$, one has
    $$\alpha_{I,i}\in H^0\big(\pb^n,  E_{i,i}^{\rm GG}\Omega_{\pb^n}(\log D_\af)\otimes \oc_{\pb^n}(\varepsilon+k\delta) \big).$$
    One should think of these elements as some holomorphic functions on some suitable logarithmic jet space $\pb_{\!\!k}^n(D_\af)$ (the logarithmic version of the Demailly-Semple jet tower constructed in \cite{DL01}).
Once suitably interpreted,  \eqref{intro:equations} allows us to construct a rational map 
    \begin{align}\label{intro:rational}
    \Phi_\af:\pb_{\!\!k}^n(D_\af)&\dashrightarrow \ys\\\nonumber
    w&\stackrel{{\rm loc}}{\mapsto }\Big( {\rm Span}\big(\alpha_{\bullet,1}(w),\alpha_{\bullet,2}(w),\ldots,\alpha_{\bullet,k}(w)\big) ; [z_0^r,z_1^r,\ldots,z_n^r]\Big)
    \end{align}
    where  $\alpha_{\bullet,i}(w):=\big(\alpha_{I,i}(w)\big)_{|I|=\delta}\in H^0\big(\pb^n,\oc_{\pb^n}(\delta)\big)$, and   \(\ys\) is the universal complete intersections of codimension \(k\) and multidegree \((\delta,\dots, \delta)\):
    \[\ys:=\{(\Delta,[T])\in {\rm Gr}_{k}\Big(H^0\big(\pb^n,\oc_{\pb^n}(\delta)\big)\Big) \times \pb^n\ | \ \forall P\in \Delta, P([T])=0 \}.\]
    If one denotes by \(\ls \) the Pl\"ucker line bundle on the Grassmannian ${\rm Gr}_{k}\Big(H^0\big(\pb^n,\oc_{\pb^n}(\delta)\big)\Big)$, by \eqref{intro:rational},  for any \(m\in \nb^*\) and for \(r\)  large enough, the pull-back of  every section  in 
    \[H^0\big(\ys,\ls^m\boxtimes \oc_{\pb^n}(-1)|_{\ys}\big)\]
    induces a logarithmic jet differential equation on the pair \((\pb^n,D_\af)\) vanishing along some ample divisor. Observe that when $k\geqslant n$, the projection map \(\ys\to {\rm Gr}_{k}\Big(H^0\big(\pb^n,\oc_{\pb^n}(\delta)\big)\Big)\) is generically finite, and   thus the   pull back of \(\ls\) to \(\ys\) is a big and nef line bundle. Therefore,  when   $m$ is large enough, there are many  global sections of \(\ls^m\boxtimes \oc_{\pb^n}(-1)|_{\ys}\).   Moreover, in view of a result of Nakamaye \cite{Nak00} the base locus    \({\rm Bs}\big(\ls^m\boxtimes \oc_{\pb^n}(-1)|_{\ys}\big)\) can be understood geometrically.
    Altogether this will allow us to control the geometry of the logarithmic jet differential forms we construct this way and eventually prove that for a general \(\af\) and suitable restrictions on the different parameters, the pair \((\pb^n,D_{\af})\) satisfies a property which is Zariski open and implies, among other things, Brody hyperbolicity.

    \medskip
Let us however emphasize that there are many technical difficulties along the way. First of all, we are not able to work directly with the pair \((\pb^n,D_{\af})\), but     with a pair \((H_{\af},D_{\af})\) which is biholomorphic to \((\pb^n,D_{\af})\) such that the family of all such pairs is easier to study. Secondly, the above rational map \(\Phi_{\af}\) is not a regular morphism in general. Therefore the above strategy doesn't provide any information on what happens along the indeterminacy locus of this map but in order to obtain the strong hyperbolicity property we seek, we need some information on the entire logarithmic jet tower and not merely an open subset of it. Therefore as in \cite{Bro17}, we introduce a suitable modification of the logarithmic jet tower obtained by blowing up a suitable ideal sheaf induced by the logarithmic Wronskian construction we introduce here.
The main difficulty lies in the description of elements \(\alpha_{I,i}\) constructed above as holomorphic functions on  the logarithmic Demailly  jet tower. This forces us to introduce another version, more technically involved but more precise, of the logarithmic connections and the logarithmic Wronskians mentioned previously.

    \medskip
    
    The paper is organized as follows.   In \cref{sec:jet space}, we recall the  technical tools in studying the hyperbolicity of algebraic varieties, especially the logarithmic Demailly jet tower and the invariant logarithmic jet differentials. 
    \cref{sec:connection and log Wronskian} is the main technical part of our paper. In this section,  we develop our main tools in this paper: the  higher order logarithmic connections   and   logarithmic Wronskians associated to families of global sections of a line bundle. We show that logarithmic Wronskians can be seen as a morphism from the jet bundle of a line bundle to the logarithmic invariant jet bundle.  Based on this interpretation, we prove that for the   ideal sheaf induced by the base ideal of logarithmic Wronskians, its cosupport lies on the set of singular jets in the log Demailly tower, and its blow-up is functorial under restrictions and families. This gives rise to a good compactification of the set of regular jets in the Demailly-Semple jet tower for the interior of the log pair. Using this construction we build a Zariski open property  for the Brody hyperbolicity of the family of log pairs, and reduce our proof of the main theorem to find some particular examples. \cref{sec:main construction} is devoted to the construction of the family of these particular hypersurfaces in \eqref{intro:Fermat}.  In \cref{sec:proof of main}, we provide detailed proofs of the main theorem. We first prove \eqref{intro:equations} and \eqref{intro:rational}, and show the existence of $D_\af\subset \pb^n$ satisfying the above Zariski open property when we adjust the parameters. To prove \cref{main SMT,main orbifold,main cyclic}, we   reduce the problems to the existence of  \emph{sufficiently many} logarithmic jet differentials with a \emph{sufficiently negative} twist.

\section{Jet spaces, jet differentials and jets of sections} \label{sec:jet space}

\subsection{Jet spaces and jet differentials}
\subsubsection{Jet spaces}
Let \(X\) be a complex manifold of dimension \(n\). For any \(k\in \mathbb{N}^*\), one defines $J_kX\rightarrow X$ to be the bundle of $k$-jets of germs of parametrized curves in $X$, that is, the set of equivalence
classes of holomorphic maps $f:(\cb, 0)\rightarrow X$, with the equivalence relation $f\sim_k g$ if and only if all
derivatives $f^{(j)}(0)= g^{(j)}(0)$ coincide for $0\leqslant j\leqslant k$, when computed in some (equivalently, any) local coordinate system
of $X$ near $x$. Given any \(f:(\cb,0)\to X\), we denote by \(j_kf\in J_kX\) the class of \(f\) in \(J_kX\).
There is a projection map  $p_k:J_kX\rightarrow X$ defined by $p_k(j_kf)= f(0)$.  
Under this map, \(J_kX\) is a \(\cb^{nk}\)-fiber bundle over \(X\). This can be seen as follows.

Let \(U\subset X\) be an open subset. For any holomorphic 1-form $\omega\in \Gamma(U,\Omega_U)$ and $f:(\cb,0)\rightarrow U$, we set  $f^*\omega:=A(t)dt$ and define the following functional:
\begin{align}\label{derivative}
	d^{k-1}\omega: p_k^{-1}(U)&\rightarrow  \cb\\\nonumber
	{j_kf}&\mapsto  A^{(k-1)}(0)
\end{align}
One immediately checks that this is well defined.  In the particular case \(\omega=d\varphi\) for some \(\varphi\in \mathscr{O}(U)\), one writes $d^{k}\varphi:=d^{k-1}\omega$. 
This construction allows us to see \(J_kX\) as a \(\cb^{nk}\)-fiber bundle over \(X\). Indeed, given $\omega_1,\ldots,\omega_n\in \Gamma(U,\Omega_U)$  generating $\Omega_X$ at any point \(x\in U\), then $\{d^{\ell}\omega_i\}_{0\leqslant \ell\leqslant k-1, 1\leqslant i\leqslant n}$ gives rise to the local trivialization of $p_k^{-1}(U)$:
\begin{align}\label{trivialization}
	p_k^{-1}(U)&\rightarrow  U \times \cb^{nk}\\\nonumber
	{j_kf}&\mapsto  \big(f(0); d^{\ell}\omega_i(j_{\ell}f)\big)_{0\leqslant \ell\leqslant k-1, 1\leqslant i\leqslant n}
\end{align}
In this case the projection to the second factor $\cb^{nk}$ is called the \emph{jet projection}, and
the natural coordinates of $\cb^{nk}$ are called \emph{jet coordinates}.

In particular, if  $(z_1,\ldots,z_n)$ are local holomorphic coordinates on $U$ centered at a point \(x\in U\), then $dz_1,\ldots,dz_k$ generates $\Omega_U$ at each point of \(U\).  Any germ of curve $f:(\cb,0)\rightarrow (X,x)$    can be written as
$$
f=(f_1,\ldots,f_n):(\cb, 0)\rightarrow  (\cb^n,0).
$$
It follows from the trivialization \eqref{trivialization} given by $\{d^{\ell}z_j\}_{1\leqslant \ell\leqslant k, 1\leqslant j\leqslant n}$ that the fiber $p_k^{-1}(x)$ can  be identified with the set of $k$-tuples of vectors
$$(\xi_1,\ldots,\xi_k) = \big(f'(0),f''(0), \ldots,f^{(k)}(0)\big)\in \cb^{nk}.$$
Observe also that there is a natural $\cb^*$-action on fibers of $J_kX$ defined by
$$
\lambda\cdot j_kf:=j_k(t\mapsto f(\lambda t)), \ \ \forall \ \lambda\in \cb^*, \ j_kf\in J_kX. 
$$ 
With respect to the above trivialization, this action  is described in jet coordinates by
$$
\lambda\cdot \big(f'(0),f''(0),\ldots,f^{(k)}(0)\big)=\big(\lambda f'(0),\lambda^2f''(0),\ldots,\lambda^kf^{(k)}(0)\big).
$$
\subsubsection{Jet differentials} Let us now recall the fundamental concept of jet differentials. For \(X\) as above, any open subset \(U\subset X\) and any integer \(k\geqslant 1\), a \emph{jet differential of order \(k\) on \(U\)} is an element \(P\in \mathscr{O}(p_k^{-1}(U))\). The (non-coherent) sheaf of jet differentials is defined to be \(\mathscr{E}^{\rm GG}_{k,\bullet}\Omega_X:=(p_{k})_*\mathscr{O}_{J_kX}\).

The $\cb^*$-action can be used to define the notion of 
of weight for jet differentials: a $k$-jet differential $P\in \mathscr{E}_{k,\bullet}^{\rm GG}\Omega_X(U)=\mathscr{O}(p_k^{-1}(U))$  is said to be of \emph{weight} $m$ if for any $j_kf\in p_k^{-1}(U)$, one has
 $$P(\lambda\cdot j_kf)=\lambda^mP(j_kf).$$  
 We thus define the Green-Griffiths sheaf $\mathscr{E}_{k,m}^{\rm GG}\Omega_X$ of jet differentials of order $k$ and weighted degree $m$ to be the subsheaf of $\mathscr{E}_{k,\bullet}^{\rm GG}\Omega_X$, of jet differentials of weight  $m$ with respect to the $\cb^*$-action.   With the above local coordinates, any element $P\in \mathscr{E}_{k,m}^{\rm GG}\Omega_X(U)$ can be written as
\begin{equation}\label{eq:ExpressionP}
 P(z,dz,\dots, d^kz)=\sum_{|\alpha|=m} c_\alpha(z) (d^{1}z)^{\alpha_1}(d^{2}z)^{\alpha_2}\cdots (d^{k}z)^{\alpha_k},
 \end{equation}
 where $c_\alpha(z)\in \oc(U)$ for any $\alpha:=(\alpha_1,\ldots,\alpha_k)\in (\mathbb{N}^n)^{k}$ and where we used the usual multi-index notation with the weighted degree  $|\alpha|:=|\alpha_1|+2|\alpha_2|+\cdots+k|\alpha_k|$. From this it follows at once that \(\mathscr{E}_{k,m}^{\rm GG}\Omega_X\) is locally free, and we shall denote the associated vector bundle by \(E_{k,m}^{\rm GG}\Omega_X\). One also defines
 $
 E_{k,\bullet}^{\rm GG}\Omega_X=\bigoplus_{m\geqslant 0} E_{k,m}^{\rm GG}\Omega_X,
 $ 
 which is in a natural way a bundle of graded algebras (the product is obtained simply by taking
 the product of polynomials).

Besides the multiplication, one can define for every \(k,m\geqslant 0\), a \(\cb\)-linear operator  \(d:\mathscr{E}_{k,m}^{\rm GG}\Omega_X\to \mathscr{E}^{\rm GG}_{k+1,m+1}\Omega_X\)  by
\[(dP)(j_{k+1}f):=\frac{d}{dt}\big(P (j_{k}f(t))\big)(0).\]
The fact that \(dP\) is well defined and holomorphic follows from a  local computation. This operator is coherent with the definition of \(d^k\) above in the sense that for any holomorphic one form \(\omega\in \Gamma(U,\Omega_U)\) on some open subset \(U\subset X\), and any \(k\in \mathbb{N}^*\) one has \(d^k\omega=d(d^{k-1}\omega)\). For instance, this implies that \(d^{k}\omega\in \mathscr{E}_{k+1,k+1}^{\rm GG}\Omega_X(U)\). In coordinates the operator \(d\) can be computed as follows. Take an open subset \(U\subset X\) with a coordinate chart \((z_1,\dots, z_n)\), and let \(P\in  \mathscr{E}_{k,m}^{\rm GG}\Omega_X(U)\) represented in coordinates by the expression \eqref{eq:ExpressionP}, then \(dP\) is given by
\begin{align*}
\sum_{|\alpha|=m} \left(\sum_{i=1}^n\frac{\partial c_\alpha(z)}{\partial z_i} d^1z_i(d^{1}z)^{\alpha_1}\cdots (d^{k}z)^{\alpha_k} +
\sum_{j=1}^k\sum_{i=1}^nc_\alpha(z) \alpha_j^id^{j+1}z_i(d^{1}z)^{\alpha_1}\cdots (d^jz)^{\alpha'_{j,i}}\cdots (d^{k}z)^{\alpha_k}\right),
\end{align*}
where \(\alpha'_{j,i}=(\alpha_j^1,\dots, \alpha_j^{i-1},\alpha_j^i-1,\alpha_j^{i+1},\dots, \alpha_j^n)\).

\subsection{Logarithmic jet spaces and bundles}\label{sec:LogJetSpace}
Let $X$ be a complex manifold (not necessarily compact), and let $D=\sum_{i=1}^{c}D_i$ be a simple normal crossing divisor on $X$,
that is, all the components $D_i$ are smooth irreducible divisors that meet
transversally. Such a pair $(X,D)$ is called a (smooth) \emph{log manifold}.  One denotes by $T_X(-\log D)$
the logarithmic
tangent bundle of $X$ along $D$. By definition, it is the subsheaf of the holomorphic tangent bundle \(T_X\) consisting of vector fields tangent to \(D\). One can then show that under our assumptions on \(D\), \(T_X(-\log D)\) is a locally free sheaf. Let  $U\subset X$ be an open subset of with local coordinates $(z_1,\ldots,z_n)$ such that for some \(0\leqslant c'\leqslant c\), \(D\cap U=(z_1\cdots z_{c'}=0)\) and  (up to reordering the components) one has $D_i\cap U=(z_i=0)$ for all \(i=1\dots c'\). Then $T_X(-\log D)$ is generated by
$$
z_1\frac{\d}{\d z_1},\ldots,z_{c'}\frac{\d}{\d z_{c'}}, \frac{\d}{\d z_{c'+1}}, \ldots, \frac{\d}{\d z_n}.
$$ Consider the dual of $T_X(-\log D)$, which is the locally free sheaf generated by
$$
\frac{d z_1}{z_1},\ldots,\frac{d z_{c'}}{z_{c'}},  d z_{c'+1}, \ldots,  d z_n,
$$
and denoted by $\Omega_X(\log D)$. The vector bundle $\Omega_X(\log D)$ is called the \emph{logarithmic cotangent bundle} of $(X,D)$.  We denote by $\jr_k(X)$ the set of local holomorphic sections
$\alpha:U\rightarrow J_kX$ of the $k$-jet bundle $J_kX\rightarrow X$, and   $\jr_k(X,\log D)$   the sheaf of
germs of local holomorphic sections $\alpha$ of $J_kX$ such that for any $\omega\in \Omega_X(\log D)_x$,
$(d^{j-1}\omega)(\alpha)$ are all holomorphic for any $j=1,\ldots,k$. $\jr_k(X,\log D)$  is called the logarithmic
$k$-jet sheaf and $\alpha$ is called a logarithmic $k$-jet field. Here we observe that for any meromorphic 1-form $\omega\in \mr(U,\Omega_X)$, one can also define $d^{i} \omega $ for any $i=1,\ldots,k$ as \eqref{derivative}, which can be seen as meromorphic sections of the fiber bundle $J_kX\rightarrow X$. It follows from \cite{Nog86} (see also \cite[\S 4.6.3]{NW14}) that there exists also a natural holomorphic fiber bundle $J_k(X,\log D)$ such that 
\begin{thmlist}
	\item there is a fiber mapping $\lambda:J_k(X,\log D)\rightarrow J_kX$, locally defined by
\begin{align*}\label{local trivialization}
\lambda:J_k(X,\log D)_{\upharpoonright U}&\rightarrow  J_kX_{\upharpoonright U}\\\nonumber
\big(z;z^{(j)}_\ell,z^{(j)}_i\big)_{1\leqslant j\leqslant k, 1\leqslant \ell\leqslant c',c'<i\leqslant n} &\mapsto  (z;z_\ell\cdot z^{(j)}_\ell,z^{(j)}_i)_{1\leqslant j\leqslant k, 1\leqslant \ell\leqslant c',c'<i\leqslant n};
\end{align*}
	\item the induced mapping between sections of holomorphic fiber bundles 
	$$\lambda_*:\Gamma\big(U,J_k(X,\log D)\big)\rightarrow  \jr_k(X,\log D)(U)$$
	is an isomorphism.
\end{thmlist}
Let us denote by $w_1= \log z_1,\ldots,w_{c'}= \log z_{c'},w_{c'+1}= z_{c'+1},\ldots,w_n= z_n$. The notation \(\log z_i\) should be understood formally and is used to simplify the notation \(d w_i=d\log z_i=\frac{dz_i}{z_i}\). One then has another  trivialization of $J_k(X,\log D)_{\upharpoonright U}$ is given as follows:
\begin{align*}
(d^{j}w_i)_{1\leqslant j\leqslant k,1\leqslant i\leqslant n}: J_k(X,\log D)_{\upharpoonright U} \rightarrow U\times \cb^{nk}.
\end{align*}
A local meromorphic $k$-jet differential $\alpha$ on $U$ is called a logarithmic
$k$-jet differential, if $\alpha(\beta)$ is holomorphic for any logarithmic $k$-jet field $\beta\in \jr_k(X,\log D)(U)$. The sheaf of logarithmic
$k$-jet   differential is denoted by $\mathcal{E}_{k,\bullet}^{\rm GG}\Omega_X(\log D)$, which is  also a locally free sheaf. The associated vector bundle is denoted by $E_{k,\bullet}^{\rm GG}\Omega_X(\log D)$, and is called \emph{$k$-jet logarithmic Green-Griffiths bundle}.  One also the following natural splitting 
$$
E_{k,\bullet}^{\rm GG}\Omega_X(\log D)=\bigoplus_{m\geqslant 0} E_{k,m}^{\rm GG}\Omega_X(\log D),
$$
where $E_{k,m}^{\rm GG}\Omega_X(\log D)$ is the  logarithmic $k$-jet  differentials of weighted degree $m$. Any local section $P\in \mathscr{E}_{k,m}^{\rm GG}\Omega_X(\log D)(U)$ can be written as
\begin{eqnarray}\label{local expression}
\sum_{|\alpha|=m} c_\alpha(z) (d^{1}w)^{\alpha_1}(d^{2}w)^{\alpha_2}\cdots (d^{k}w)^{\alpha_k},
\end{eqnarray}
where $c_\alpha(z)\in \oc(U)$ for any $\alpha:=(\alpha_1,\ldots,\alpha_k)\in (\mathbb{N}^n)^{k}$.  We will use another trivialization of $E_{k,m}^{\rm GG}\Omega_X(\log D)$. First, let us begin with a lemma.
\begin{lem}\label{lem:new basis}
Assume that locally on an open subset of $U\subset X$ with local coordinates $(z_1,\ldots,z_n)$ such that $D\cap U=(z_1=0)$. Then for any $j\in \mathbb{N}$,  $\frac{d^{j}z_1}{z_1}$ is  a logarithmic jet differential and moreover,  any local section $P\in  \mathscr{E}_{k,m}^{\rm GG}\Omega_X^*(\log D)(U)$ can be written as
\begin{eqnarray}\label{another expression}
\sum_{|\alpha|=m} \frac{c_\alpha(z)}{z_1^{\alpha_{1}^1+\cdots+\alpha_{k}^1}} (d^{1}z)^{\alpha_1}(d^{2}z)^{\alpha_2}\cdots (d^{k}z)^{\alpha_k},
\end{eqnarray}
where $c_\alpha(z)\in \oc(U)$ for any $\alpha:=(\alpha_1,\ldots,\alpha_k)\in (\mathbb{N}^n)^{k}$ with $\alpha_j=(\alpha_{j}^1,\ldots,\alpha_{j}^n)\in \mathbb{N}^n$.
\end{lem}
\begin{proof} Let us prove by induction that for any \(j\geqslant 1\)  and any  $\beta=(\beta_1,\ldots,\beta_j)\in \mathbb{N}^j$ there exists $b_{j\beta}\in \mathbb{Z}$ such that 
	\begin{eqnarray}\label{another log}
	\frac{d^{j}z_1}{z_1}=\sum_{\beta_1+2\beta_2+\cdots+j\beta_j=j} b_{j\beta}\cdot (d^{1}\log z_1)^{\beta_1}(d^{2}\log z_1)^{\beta_2}\cdots (d^{j}\log z_1)^{\beta_j}.
	\end{eqnarray}
 By definition, it holds for \(j=1\). Assume now that \eqref{another log} holds for $j$. Then 
\begin{align*} 
d^{j+1}z_1=&d^{1}z_1\cdot\sum_{\beta_1+2\beta_2+\cdots+j\beta_j=j} b_{j\beta}\cdot (d^{1}\log z_1)^{\beta_1}(d^{2}\log z_1)^{\beta_2}\cdots (d^{j}\log z_1)^{\beta_j}+\\
&z_1\cdot \sum_{\beta_1+2\beta_2+\cdots+j\beta_j=j}\sum_{i=1}^{j} \beta_ib_{j\beta}\cdot (d^{1}\log z_1)^{\beta_1}\cdots(d^{i}\log z_1)^{\beta_{i}-1}(d^{i+1}\log z_1)^{\beta_{i+1}+1}\cdots (d^{j}\log z_1)^{\beta_j}\\
=&z_1\cdot\sum_{\beta_1+2\beta_2+\cdots+j\beta_j=j} b_{j\beta}\cdot (d^{1}\log z_1)^{\beta_1+1}(d^{2}\log z_1)^{\beta_2}\cdots (d^{j}\log z_1)^{\beta_j}+\\
&z_1\cdot \sum_{\beta_1+2\beta_2+\cdots+j\beta_j=j}\sum_{i=1}^{j} \beta_ib_{j\beta}\cdot (d^{1}\log z_1)^{\beta_1}\cdots(d^{i}\log z_1)^{\beta_{i}-1}(d^{i+1}\log z_1)^{\beta_{i+1}+1}\cdots (d^{j}\log z_1)^{\beta_j},
\end{align*}
and thus \eqref{another log} holds also for $j+1$.

	 On the other hand, one will prove by induction on $j$ that
	\begin{eqnarray}\label{inverse trivialization}
	d^{j}\log z_1=\sum_{\beta_1+2\beta_2+\cdots+j\beta_j=j} b_{j\beta}\cdot \Big(\frac{d^{1}z_1}{z_1}\Big)^{\beta_1}\Big(\frac{d^{2}z_1}{z_1}\Big)^{\beta_2}\cdots \Big(\frac{d^{j}z_1}{z_1}\Big)^{\beta_j}
	\end{eqnarray}
		where $b_{j\beta}\in \mathbb{Z}$ and $\beta=(\beta_1,\ldots,\beta_j)\in \mathbb{N}^j$. Assume that \eqref{inverse trivialization} holds for $j$. Then
		\begin{eqnarray*}
		d^{j+1}\log z_1=\sum_{\beta_1+2\beta_2+\cdots+j\beta_j=j}\sum_{i=1}^{j} \beta_ib_{j\beta}\cdot \Big(\frac{d^{1}z_1}{z_1}\Big)^{\beta_1}\cdots\Big(\frac{d^{i}z_1}{z_1}\Big)^{\beta_i-1}\cdots \Big(\frac{d^{j}z_1}{z_1}\Big)^{\beta_j}\Big(\frac{d^{i+1}z_1}{z_1}-\frac{d^{i}z_1d^{1}z_1}{z_1^2}\Big).
		\end{eqnarray*}
Hence \eqref{inverse trivialization}  also holds for $j+1$.  It thus just remains to use \eqref{inverse trivialization} and \eqref{local expression} to show \eqref{another expression}.
\end{proof}

\subsection{Demailly-Semple tower}\label{subsec:DS}
In this section we recall the formalism of directed pairs as introduced by Demailly \cite{Dem95}. A  \emph{directed manifold} $(X,V)$ is  a complex manifold $X$   equipped with a subbundle $V\subset T_X$ of rank $r$. A morphism of directed manifolds \(f:(Y,V_Y)\to(X,V_X)\) is by definition a morphism \(f:Y\to X\) such that \(f_*V_Y\subset f^*V_X\subset f^*T_X\). In \cite{Dem95} Demailly introduced the \emph{$1$-jet functor} which to any directed manifold $(X,V)$ associates the directed manifold defined by \(\pt_1V=\pt(V)\) and $ V_1:= (\pi_{0,1})_*^{-1}\oc_{\pt(V)}(-1)\subset T_{X_1}$, where $\oc_{\pt_1V}(-1)$ denotes the tautological line bundle  $\oc_{\pt(V)}(-1)$. This induces a morphism between directed manifolds
$
 (\pt_1V,V_1)\xrightarrow{\pi_{0,1}} (X,V).
 $
By iterating this $1$-jet functor, Demailly then  constructed
the so-called \emph{Demailly-Semple $k$-jet tower}  
\[
(\pt_kV,V_{k})\xrightarrow{\pi_{k-1,k}}  (\pt_{k-1}V,V_{k-1})\xrightarrow{\pi_{k-2,k-1}}  \cdots\rightarrow  (\pt_1V,V_1)\xrightarrow{{\pi_{0,1}}}  (X,V)
\]  
such that $\pt_kV := \pt (V_{k-1})$ and
 $V_k := (\pi_{k-1,k})_*^{-1}
\oc_{\pt_kV}(-1)\subset T_{\pt_kV}.$  
Here we denote by $\oc_{\pt_kV}(-1)$  the tautological line bundle  $\oc_{\pt(V_{k-1})}(-1)$, $\pi_{k-1,k}: \pt_kV  \rightarrow \pt_{k-1}V$ the
natural projection and $(\pi_{k-1,k})_* = d\pi_{k-1,k} : T_{\pt_kV} \rightarrow \pi^*_{k-1,k}T_{\pt_{k-1}V}$
the differential. By composing the projections we get for all pairs of indices $0 \leqslant j \leqslant  k$ natural morphisms
$$\pi_{j,k} : \pt_kV \rightarrow \pt_jV,\quad (\pi_{j,k})_* = (d\pi_{j,k})_{\upharpoonright V_k}: V_k \rightarrow (\pi_{j,k})^*V_j.
$$
For every $k$-tuple $  (a_1, \ldots , a_k)\in\mathbb{Z}^k$ we write
 $\oc_{\pt_kV}(a_1, \ldots, a_k)= \bigotimes_{1\leqslant j\leqslant k}\pi^*_{j,k}\oc_{\pt_jV}(a_j).
$ 
One can   inductively
define  $k$-th lift $f_{[k]}:(\cb,0)\rightarrow \pt_kV$  for germs of non-constant holomorphic curves $f:(\cb,0)\to X$  by $f_{[k]}(t)=\big(f_{[k-1]}(t),[f_{[k-1]}'(t)]\big)$ (although this is not well defined when \(f_{k-1}'(t)=0\) one can easily extend this definition to every \(t\) in the domain of definition of \(f\)).  

On the other hand, let $\gb_k$ be the group of germs of $k$-jets of biholomorphisms of $(\cb, 0)$, that is, the group of germs of biholomorphic maps
$$\varphi:t \mapsto  a_1t + a_2t^2+ \cdots + a_kt^k,\quad a_1\in \cb^*, a_j \in \cb, j>2,$$
in which the composition law is taken modulo terms $t^j$
of degree $j > k$. Then $\gb_k$ is a $k$-dimensional
nilpotent complex Lie group, which admits a natural fiberwise right action on $J_kV$. The action consists of reparameterizing $k$-jets of maps $f : (\cb, 0) \rightarrow (X,V)$ by a biholomorphic change of parameter
$\varphi: (\cb, 0) \rightarrow(\cb, 0)$ defined by $(f,\varphi) \mapsto f\circ\varphi$.  Moreover, if one denotes by $$J_k^{\rm reg}V:=\{j_kf\in J_kV \mid f'(0)\neq 0\}$$
the space of \emph{regular $k$-jets} tangent to $V$,  there exists a natural  morphism 
\begin{align}\label{embedding}
	J_k^{\rm reg}V &\rightarrow   \pt_kV\\
	{j_kf} &\mapsto   f_{[k]}(0)\nonumber
\end{align}
whose image is an open set in $\pt_kV$ denoted by $\pt_kV^{\rm reg}$; in other words,  $\pt_kV^{\rm reg}\subset \pt_kV$ is the set of elements $f_{[k]}(0)$ in $\pt_kV$ which can be reached by  regular germs of
curves $f$. It was proved in \cite[Theorem 6.8]{Dem95} that $\gb_k$ acts transitively on $J_k^{\rm reg}V$, and thus $\pt_kV^{\rm reg}$ can be identified with the quotient $J_k^{\rm reg}/\mathbb{G}_k$.
 Moreover,  the \emph{singular $k$-jets}, denoted by $\pt_kV^{\rm sing}:=\pt_kV\setminus \pt_kV^{\rm reg}$,  is a divisor in $\pt_kV$. In summary,  $\pt_kV$ is a  smooth compactification of $J_k^{\rm reg}/\gb_k$. 
 As will become clear later, and as was observed in  \cite[\S 7]{Dem95}, when dealing with hyperbolicity questions, the locus \(\pt_kV^{\rm sing}\) is in some sense irrelevant.

Let us recall the following  theorem by Demailly which is a crucial tool in our paper.
\begin{thm}[\!\!\protect{\cite[Corollary 5.12, Theorem 6.8]{Dem95}}]
	Let $(X,V)$ be a directed variety. 
	\begin{thmlist}
		\item \label{para} For any $w_0\in \pt_kV$, there exists an open neighborhood $U_{w_0}$ of $w_0$ and a family of  germs of curves $(f_w)_{w\in U_{w_0}}$, tangent to $V$ depending holomorphically on $w$ such that
		$$
		(f_{w})_{[k]}(0)=w\quad \mbox{and}\quad (f_{w})'_{[k-1]}(0)\neq 0,\quad \forall w\in U_{w_0}.
		$$
		In particular, $(f_{w})'_{[k-1]}(0)$ gives a local trivialization of the tautological line bundle $\oc_{\pt_kV}(-1)$ on $U_{w_0}$. 
		\item For any \(k,m\geqslant 1\) one has
		\begin{eqnarray}\label{local isomorphism}
		(\pi_{0,k})_*\oc_{\pt_kV}(m)\cong \mathscr{E}_{k,m}V^*.
		\end{eqnarray}
		\end{thmlist}
\end{thm} 

In fact, the isomorphism \eqref{local isomorphism} can be understood explicitly in view of \cref{para}. With the notation  therein, for any given local invariant jet differential $P\in  \mathscr{E}_{k,m}V^*(U)$, the inverse image under $(\pi_{0,k})_*$ is the section 
 $\sigma_P\in \Gamma\big(U_{w_0}, \oc_{\pt_kV}(m)_{\upharpoonright U_{w_0}}\big)$ 
defined by
\begin{eqnarray}\label{define section}
\sigma_P(w):=P(j_kf_w)\big((f_{w})'_{[k-1]}(0)\big)^{-m}.
\end{eqnarray}

\subsection{Logarithmic Demailly-Semple  bundle}\label{sec:LogDSTower}
In \cite{DL01}, Dethloff-Lu extended the Demailly-Semple tower  to the logarithmic setting. They used it in particular to reprove the Brody hyperbolicity of  complements of  ample divisors in the abelian varieties. Following \cite{DL01}, a \emph{logarithmic directed manifold} is a triple $(X,D,V)$ where $(X,D)$ is a log manifold, and $V$ is a subbundle of $T_X(-\log D)$. In this section, we will recall for the reader's convenience Dethloff-Lu's construction of the \emph{logarithmic Demailly(-Semple)  $k$-jet tower}  associated to any logarithmic directed manifold.

Given log-manifolds $(X',D')$ and $(X,D)$, a holomorphic map $f:X'\rightarrow X$ such that $f^{-1}(D)\subset D'$
will be called a log-morphism from $(X',D')$ to $(X,D)$. It induces  morphisms 
$$
f_*:T_{X'}(-\log {D'})\rightarrow f^*T_X(-\log D) \ \ \text{and} \ \ f^*:f^*E_{k,m}^{\rm GG}\Omega_{X}(\log {D})\rightarrow E_{k,m}^{\rm GG}\Omega_{X'}(\log {D'}).
$$
A log directed morphism between log directed
manifolds $(X',D',V')$ and $(X,D,V)$ is a log morphism $f:(X',D')\rightarrow (X,D)$ such that $f_*V'\subset V$. 

 For any fixed order $k$, as the Demailly-Semple bundle, 
the \emph{logarithmic Demailly $k$-jet tower} 
$$\big({X}_k(D),D_k,V_k\big)\xrightarrow{{\pi}_{k-1,k}} \big({X}_{k-1}(D),D_{k-1},V_{k-1}\big)\xrightarrow{{\pi}_{k-2,k-1}} \ldots\rightarrow \big({X}_1(D),D_1,V_1\big)\xrightarrow{{\pi}_{0,1}} (X,D,V) $$
is constructed inductively. Define ${X}_k(D):=\pt({V}_{k-1})$, and let $ {{\pi}_{k-1,k}}:{X}_k(D) \rightarrow  {X}_{k-1}(D) $ be the natural projection. Set $D_{k}:=({\pi}_{k-1,k})^{-1}(D_{k-1})$ which is a simple normal crossing divisor, and induces a  morphism
$$
({{\pi}}_{k-1,k})_*:T_{{X}_k(D)}(-\log D_k)\rightarrow ({{\pi}}_{k-1,k})^* T_{{X}_{k-1}(D)}(-\log D_{k-1}).
$$
Define
$$
{V}_k:=({{\pi}}_{k-1,k})_*^{-1} \oc_{X_k(D)}(-1) \subset T_{{X}_k(D)}(-\log D_k),
$$
where $\oc_{X_k(D)}(-1):=\oc_{\pt({V}_{k-1})}(-1)$ is the  tautological line bundle,  which by definition is also a subbundle of $({{\pi}}_{k-1,k})^*{V}_{k-1}$.   We say that $\big({X}_k(D),D_k,V_k\big)\xrightarrow{{\pi}_{k-1,k}} \big({X}_{k-1}(D),D_{k-1},V_{k-1}\big)$ is the \emph{1-jet functor} of the log direct manifold $({X}_{k-1}(D),D_{k-1},V_{k-1})$.

Note that ${\rm ker}({{\pi}}_{k-1,k})_*=T_{{X}_k(D)/{X}_{k-1}(D)}$ by definition. This gives the following short exact
sequence of vector bundles over ${X}_k(D)$ 
$$
0\rightarrow T_{{X}_k(D)/{X}_{k-1}(D)}\rightarrow V_k\xrightarrow{({{\pi}}_{k-1,k})_*} \oc_{X_k(D)}(-1)\rightarrow 0.
$$
  Furthermore, we have the Euler exact sequence for projectivized bundles
  $$
  0\rightarrow \oc_{X_k(D)}\rightarrow ({{\pi}}_{k-1,k})^*{V}_{k-1}\otimes \oc_{X_k(D)}(1)\rightarrow T_{{X}_k(D)/{X}_{k-1}(D)}\rightarrow 0.
  $$
  By definition, there is a canonical line bundle morphism
\begin{eqnarray}\label{eq:short embedding}
 \oc_{X_k(D)}(-1)\hookrightarrow ({{\pi}}_{k-1,k})^*{V}_{k-1}\xrightarrow{({\pi}_{k-1,k})^*({{\pi}}_{k-2,k-1})_*} ({{\pi}}_{k-1,k})^*\oc_{{X}_{k-1}(D)}(-1)
\end{eqnarray}
which admits precisely $\Gamma_k:=\pt\big(T_{{X}_k(D)/{X}_{k-1}(D)}\big)\subset \pt(V_k)={X}_k(D)$ as its zero divisor:
\begin{equation}\label{eq:Gamma_k}
\oc_{X_k(D)}(1)= ({{\pi}}_{k-1,k})^*\oc_{{X}_{k-1}(D)}(1)\otimes \oc_{X_k(D)}(\Gamma_k).
\end{equation}
  Let us denote by  ${{\pi}_{j,k}}:{X}_k(D)\rightarrow {X}_j(D)$ 
the composition of  the   projections ${\pi}_{k}\circ \cdots\circ {\pi}_{j+1}$. Define
 $
{X}_k(D)^{\rm sing}:=\bigcup_{2\leqslant j\leqslant k} {{\pi}_{j,k}}^*(\Gamma_{j})
$, 
and
 $
{X}_k(D)^{\rm reg}:={X}_k(D) \setminus{X}_k(D)^{\rm sing}
$. 

\begin{dfn} For any open subset \(U\subset X\), a logarithmic differential operator $P\in \Gamma\big(U,\mathcal{E}^{\rm GG}_{k,m}\Omega_X(\log D)\big)$ is said to be
\emph{invariant by reparametrization group} $\gb_k$ if for any $g\in \gb_k$ and any $j_kf\in J_kX^{\rm reg}_{\upharpoonright X\setminus D}$, one has 
$$
P\big(j_k(f\circ g)\big)=g'(0)^m\cdot P(j_kf).
$$
	Let us define $\mathcal{E}_{k,m}\Omega_X(\log D)$ to be the subsheaf of $\mathcal{E}^{\rm GG}_{k,m}\Omega_X(\log D)$ which consists of invariant logarithmic differential operator. The associated vector bundle is denoted by ${E}_{k,m}\Omega_X(\log D)$.
\end{dfn}
The log Demailly tower is of great importance in the study of the algebraic degeneracy
of entire curves on $X\setminus D$, granting the following direct image formula in \cite[Proposition 3.9]{DL01} 
\begin{eqnarray}\label{eq:DirectImageFormula}
\mathscr{E}_{k,m}\Omega(\log D)=({\pi}_{k})_*\oc_{{X}_k(D)}(m).
\end{eqnarray} 

The following fundamental result shows that  the logarithmic jet differentials vanishing along some ample divisor  provides obstructions to the existence of entire curves in the complement.
\begin{thm}[Dethloff-Lu, Siu-Yeung]\label{thm:fundamental}
	Let \(X\) be a smooth complex projective variety  with \(D\subset X\) a normal crossings divisor on \(X\), and \({X}_k(D)\) denotes to be the log Demailly  \(k\)-jet tower of \(\big(X,D,T_X(-\log D)\big)\). For any non-constant entire curve \(f:\cb\rightarrow X\setminus D\) avoiding \(D\), any ample line 	bundle \(A\) on \(X\), any \(a_1,\ldots,a_k\in \mathbb{N} \) and any \[\omega\in H^0\big({X}_k(D),\oc_{{X}_k(D)}(a_1,\ldots,a_k) \otimes ({\pi}_{0,k})^*A^{-1} \big), \]
	one has \(f_{[k]}(\cb)\subseteq (\omega=0)   \).
\end{thm}

\subsection{Jet bundle of a line bundle} We recall here the basic definitions and properties of   jet bundles of a line bundle (we refer to \cite[\S16.7]{EGA4}  for a detailed presentation). Let \(X\) be a complex manifold, and let \(L\) be a line bundle on \(X\). For any integer \(k\geqslant 0\), on defines the \(k\)-th order jet bundle \(J^kL\) of \(L\) as follows. Consider the product \(X\times X\) with the canonical projections \(\pr_1,\pr_2\) on the first and second factors. Let \(\Delta_X\subset X\times X\) be the diagonal and   \(\mathscr{I}_{\Delta_X}\subset \oc_{X\times X}\) denotes its ideal sheaf. Then one defines 
\begin{align}\label{def:jet bundle}
J^kL:=\pr_{1*}\left(\oc_{X\times X}/\mathscr{I}^{k+1}_{\Delta_X}\otimes \pr_2^*L\right).
\end{align}
It can be shown that this is a locally free sheaf on \(X\) such that for an \(x\in X\), the fiber at \(x\) is  \(J^k_xL=L\otimes \oc_{X,x}/\mathfrak{m}_{X,x}^{k+1}\). This construction is also functorial in the following way: given a complex manifold \(Y\) and morphism \(\varphi:X\to Y\), one obtains a natural morphsim of \(\oc_X\)-modules
\[\varphi^*:\varphi^*J^kL\to J^k\varphi^*L,\]
induced by the commutativity of the diagram 
	\begin{displaymath}
	\xymatrix{
		X\times X \ar[d]_{\pr_1} \ar[r]^{\varphi\times \varphi}  & Y\times Y \ar[d]_{\pr_1} \\
		X  \ar[r]^{\varphi}     & Y }
	\end{displaymath}
and the fact that \((\varphi\times \varphi)^{-1}\mathscr{I}_{\Delta_Y}\subset \mathscr{I}_{\Delta_X}\). 

We shall need the following elementary proposition. 
\begin{proposition}
Let \(L\) and \(L'\) be line bundles on \(X\). Any morphism of \(\oc_X\)-modules \(h:L\to L'\) induces a morphism of  \(\oc_X\)-modules  
\[J^kL\stackrel{j^kh}{\to} J^kL'.\]
Moreover, \(j^kh\) is an isomorphism whenever \(h\) is an isomorphism.
\end{proposition}
\begin{proof}
The morphism \(j^kh\) is just the push-forward under \(\pr_1\) of the morphism
\[\pr_2^*h:\oc_{X\times X}/\mathscr{I}_{\Delta_X}\otimes L\to \oc_{X\times X}/\mathscr{I}_{\Delta_X}\otimes L'\]
induced by \(h\). The second assertion follows at once.
\end{proof}
Observe that there exists a \(\cb\)-linear morphism \[j_L^k:L\to J^kL,\] which is not a morphism of \(\oc_X\)-modules, defined, at the level of presheaves, as the composition, for any open subset  \(U\subset X\),
\[L(U)\stackrel{\pr_2^*}{\to}\pr_2^*L(U\times U)\to \oc_{U\times U}/\mathscr{I}_\Delta^{k+1}\otimes \pr_2^*L(U\times U)=J^kL(U).\] 

More explicitly, for any \(s\in L(U)\), the section \(j^k_L(s)\in J^kL(U)\) is such that  for any \(x\in U\), the element \(j_L^k(s)(x)\in J^k_xL=L\otimes \oc_{X,x}/\mathfrak{m}_{X,x}^{k+1}\) is precisely the image of \(s\) under the map \(L(U)\to L\otimes \oc_{X,x}/\mathfrak{m}_{X,x}^{k+1}\).  

This map can also be understood more explicitly in coordinates.  Take an open subset \(U\subset X\), up to considering a trivialization of \(L_{\upharpoonright U}\), one is reduced to understand \(\oc_U\to J^k\oc_U\). Observe that coordinates \((x_1,\dots, x_n)\) induce coordinates \((x_1,\dots, x_n,z_1,\dots, z_n)\) on \(U\times U\), from which one obtains that  the monomials \(((z-x)^I)_{|I|\leqslant k}\) form a local  frame for \(J^k\oc_U\). Here, we use the multi-index notation, \((z-x)^\alpha=(z_1-x_1)^{\alpha_1}\cdots (z_n-x_n)^{\alpha_n}\) for \(\alpha=(\alpha_1,\dots, \alpha_n)\) such that \(|\alpha|=\alpha_1+\cdots + \alpha_n\leqslant k\). The map \(j_{\oc_U}^k:\oc_U\to J^k\oc_U\) is then just given by computing, in each \(x\in U\),  the Taylor expansion up to order \(k\), namely, for any \(f\in \oc(U)\),
\[j^k_{\oc_U}(f)=\sum_{|\alpha|\leqslant k}\frac{1}{\alpha!}\frac{\partial^{|\alpha|}f}{\partial z^\alpha}(x)(z-x)^\alpha\in J^k{\oc_U}(U)\]
The following definition will be used in the sequel.
\begin{dfn}
Let \(X\) be a complex manifold and let \(L\) be a line bundle on \(X\). We say that \(L\) \emph{separates \(k\)-jets at every point of \(X\)} if the natural morphism
\[j^k_L:H^0(X,L)\otimes \oc_X\to J^kL\]
is surjective. Observe that this condition is equivalent to the surjectivity, for every \(x\in X\) of the natural map
 \[H^0(X,L)\otimes \oc_X\to L\otimes \oc_{X,x}/\mathfrak{m}^{k+1}_{X,x}.\]
\end{dfn}
Observe that if \(L\) is a very ample line bundle on \(X\), then \(L^k\) separates \(k\)-jets at every point of \(X\). 

\section{Higher order logarithmic connections and logarithmic Wronskians}\label{sec:connection and log Wronskian}
\subsection{Wronskians} Let us recall here the Wronskian constructions initiated by the first named author in  \cite{Bro17} and later reinterpreted by the second named author in an alternative way in \cite{Den17}. Let \(X\) be a complex manifold, and let \(L\) be a line bundle on $X$. Let \(k\geqslant 1\) be an integer and   take global sections \(s_0,\dots, s_k\in H^0(X,L)\). Then for every open subset \(U\subset X\) on which \(L\) is trivialized, by considering the holomorphic functions \(s_{0,U},\dots,s_{k,U}\in \mathscr{O}(U) \) associated to \(s_0,\dots, s_k\) under our choice of trivialization, we consider the Wronskian 
\[W_U(s_0,\dots, s_k):=\left|\begin{array}{ccc}s_{0,U}& \dots & s_{k,U} \\
d^1s_{0,U}& \dots & d^1s_{k,U} \\
\vdots&\ddots & \vdots\\
d^ks_{0,U}& \dots & d^ks_{k,U}
\end{array}\right|\in \oc(p_k^{-1}(U)).\]
Denote by \(k'=1+2+\cdots+k\). It was established in \cite{Bro17} that  \(W_U(s_0,\dots, s_k)\in \mathscr{E}_{k,k'}\Omega_X(U)\), and that those locally defined elements glue together into a global section 
\begin{align}\label{def:wronskian}
W_L(s_0,\dots,s_k)\in H^0(X,E_{k,k'}\Omega_X\otimes L^{k+1})
\end{align}
which is called \emph{Wronskian} in \cite[\S 2.2]{Bro17}. Moreover, in  \cite{Den17}, it was proved that there exists a  morphism of \(\oc_X\)-modules
\[j^kW_{L}:\bigwedge^{k+1}J^kL\to E_{k,k'}\Omega_X\otimes L^{k+1}\]
such that for any global section \(s_0,\dots, s_k\in H^0(X,L)\), one has 
\[j^kW_{L}(j^k_Ls_0\wedge\dots\wedge j^k_Ls_k)=W_L(s_0,\dots, s_k).\]
In \cref{sec:log wronskian}, we will construct a logarithmic counterpart of  Wronskians.
\subsection{Higher order logarithmic connections}\label{sec:connection}

Let \(X\) be a complex manifold. Let \(L\) be a line bundle on \(X\) and suppose that there exists \(\sigma\in H^0(X,L)\) such that \(D=(\sigma=0)\) is a smooth hypersurface of \(X\).
Then \(L\) is endowed with a natural logarithmic connection \(\nabla_{\!\! D}:L\to \Omega_X(\log D)\otimes L\), with logarithmic poles along \(D\), defined by 
\begin{eqnarray}\label{def:log connection}
\nabla_{\!\! D} s:=\sigma d\left(\frac{s}{\sigma}\right)=_{\rm loc}ds-s\frac{d\sigma}{\sigma}.
\end{eqnarray}
The second equality has to be understood locally, \emph{i.e.} if the open subset \(U\) over which \(L\) is trivialized, and if we denote by \(s_U,\sigma_U\in \mathscr{O}(U)\) the holomorphic functions associated to \(s,\sigma\), then one sets \[\nabla_{\!\!D}s_U=ds_U-s_U\frac{d\sigma_U}{\sigma_U}.\]
 This object is well defined since \(\frac{s}{\sigma}\) is a meromorphic function on \(X\) and that the local description shows that it has logarithmic poles along \(D\). Let us mention that in our paper \cite{BD17} we apply this construction   to prove  \cref{log Debarre}.

More generally, for every \(k\geqslant 0\), one can define  a \(\cb\)-linear map \(\nd^k:L\to E^{\rm GG}_{k,k}\Omega_X(\log D)\otimes L\) by
\begin{eqnarray}\label{def:higherorder}
\nd^ks:=\sigma d^k\!\left(\frac{s}{\sigma}\right).
\end{eqnarray}
Observe that \(\nd^0s=s\), and that \(\nd^1s=\nd s\) for any \(s\). Moreover, for any \(k\geqslant 1\) one has the local inductive description
\[\nd^ks=_{\rm loc}d\nd^{k-1}s-\nd^{k-1}s\cdot \frac{d\sigma}{\sigma}.\]
 We will need the following elementary, yet crucial, observation: for any \(k\geqslant 1\), one has
\begin{equation}\label{eq:VanishingNabla}
 \nd^k\sigma=0.
\end{equation}
Lastly let us observe that locally (with the above notation) one can use the Leibniz rule to \(s_U=\sigma_U\frac{s_U}{\sigma_U}\) to obtain
\begin{equation}\label{Leibniz}
d^ks_U=\sum_{i=0}^k\binom{k}{i}(\nd^is_U)\frac{d^{k-i}\sigma_U}{\sigma_U}.\end{equation}
While \(\nd^k:L\to E^{\rm GG}_{k,k}\Omega_X(\log D)\otimes L\) is only \(\cb\)-linear, we have the following proposition.
\begin{proposition}\label{prop:DefJetLogConnexion} With the above notation. There exists a  morphism of \(\oc_X\)-module
\[j^k\nd^k:J^kL\to \mathscr{E}_{k,k}^{\rm GG}\Omega_X(\log D)\otimes L,\]
such that \((j^k\nd^k)\circ j^k_L=\nd^k\).
\end{proposition}
\begin{proof}
The \(k\)-th order jet space of \(X\times X\) naturally splits as 
\[J_k(X\times X)\cong J_kX\times J_kX,\]
under the map \(j_k(f_1,f_2)\mapsto (j_kf_1,j_kf_2)\). Let us define, for any \(k,m\geqslant 0\) an operator \(d_2:\mathscr{E}^{\rm GG}_{k,m}\Omega_{X\times X}\to \mathscr{E}^{\rm GG}_{k+1,m+1}\Omega_{X\times X}\) by setting for any open subset \(U\subset X\times X\) and any \(P\in \Gamma(U,\mathscr{E}^{\rm GG}_{k,m}\Omega_{X\times X})\)
\[d_2P(j_k(f_1,f_2))=\frac{d}{dt}(P((f_1(0),j_{k-1}f_2(t)))(0).\]
We define for every \(k\in \mathbb{N}^*\) an \(\cb\)-linear morphism \(\nabla_2^k:\pi_2^*L\to \mathscr{E}^{\rm GG}_{k,k}\Omega_{X\times X}(\log \pi_2^{-1}(D))\otimes \pi_2^*L\) inductively by setting
\begin{align*}
\nabla^0_2s&=s\\
\nabla^{k}_2s&=\pi_2^*d_2\!\!\left(\frac{s}{\pi_2^*\sigma}\right)=_{\rm loc}d_2\nabla^{k-1}_2s-\nabla^{k-1}_2s\frac{d\pi_2^*\sigma}{\pi_2^*\sigma},\ \ \text{for any}\ \ k\geqslant 0.
\end{align*}
As before, the last equality has to be understood locally and   one verifies that this is well defined.  Observe that for any open subsets \(V\subset X\times X\) and \(U\subset X\) such that \(\pi_1(V)\subset U\), and for any \(f\in \mathscr{O}(U)\) one has \(d_2\pi_1^*f=0\). Therefore for any  \(s\in \pi_2^*L(U)\), one has 
\[\nabla_2^k(\pi_1^*f\cdot s)=\pi_1^*f\cdot \nabla_2^k s.\]
We can consider the composition \({\rm res}_\Delta\circ \nabla^k_2\):
\[\pi_2^*L\stackrel{\nabla_2^k}{\to} \mathscr{E}^{\rm GG}_{k,k}\Omega_{X\times X}(\log \pi_2^{-1}(D))\otimes \pi_2^*L\stackrel{{\rm res}_{\Delta}}{\to} \mathscr{E}^{\rm GG}_{k,k}\Omega_{\Delta}(\log \pi_2^{-1}(D)_{\upharpoonright \Delta})\otimes \pi_2^*L_{\upharpoonright \Delta}.\]
A local computation now proves  that for any open subset \(U\subset X\times X\) and every element \(s\in \Gamma(U,\pi_2^*L\otimes \mathscr{I}_{\Delta}^{k+1})\), one has 
\[{\rm res}_{\Delta}\circ \nabla^k_2(s)=0.\] 
Therefore we obtain a \(\cb\)-linear map 
\[\nabla_{2\Delta}^k:\pi_2^*L\times \mathscr{O}_{X\times X}/\mathscr{I}^{k+1}_{\Delta}\to \mathscr{E}^{\rm GG}_{k,k}\Omega_{\Delta}(\log \pi_2^{-1}(D)_{\upharpoonright \Delta})\otimes \pi_2^*L_{\upharpoonright \Delta},\]
and by applying \(\pi_{1*}\), we obtain a \(\cb\)-linear morphism
\[j^k\nd^k:J^kL=\pi_{1*}\left(\pi_2^*L\times \mathscr{O}_{X\times X}/\mathscr{I}^{k+1}_{\Delta}\right)\to  \pi_{1*}\left(\mathscr{E}^{\rm GG}_{k,k}\Omega_{\Delta}(\log \pi_2^{-1}(D)_{\upharpoonright \Delta})\otimes \pi_2^*L_{\upharpoonright \Delta}\right)\stackrel{\iota}{\cong}  \mathscr{E}^{\rm GG}_{k,k}\Omega_{X}(\log D)\otimes L.\]
Let us now prove that  \(j^k\nd^k\) is \(\oc_X\)-linear. Take an open subset \(U\subset X\) and elements \(f\in \mathscr{O}(U), \tau\in J^kL(U)\). By definition, one can consider (up to shrinking \(U\) if necessary) \(\tau\) as an element \(\tau\in\pi_2^*L\otimes \mathscr{O}_{X\times X}/\mathscr{I}_{\Delta}^{k+1}(\pi_1^{-1}(U))\). Up to shrinking \(U\) if necessary, take \(\tilde{\tau}\in \pi_2^*L(V)\) representing \(\tau\) for some neighborhood  \(V\subset X\times X\) of \(\Delta\cap \pi_1^{-1}(U)\). By definition, \[j^k\nd^k(f\tau)=\iota \nabla^k_{2\Delta}(\pi_1^*f\tilde{\tau})=\iota\circ {\rm res}_\Delta\circ \nabla^k_2(\pi_1^*f\cdot \tilde{\tau})=\iota\circ {\rm res}_\Delta\circ \left(\pi_1^*f\cdot \nabla^k_2(\tilde{\tau})\right)=\pi_1^*f\cdot \iota\circ {\rm res}_\Delta\nabla^k_2(\tilde{\tau})=f\cdot j^k\nd^k(\tau).\]

 To see   that \(\nd^k=j^k\nd^k\circ j^k_L\), it suffices, by definition of \(j^k_L\), to prove that for any open subset \(U\subset X\) and every element \(s\in L(U)\), one has 
 \[\iota\circ {\rm res}_{\Delta}\circ \nabla^k_2(\pi_2^*s)=\nd^k(s).\]
 But observe that \(\nabla^k_2(\pi_2^*s)=\pi_{2,k}^*\nabla^ks\), where \[\pi^*_{2,k}:\pi^*\left(\mathscr{E}^{\rm GG}_{k,k}\Omega_X(\log D)\otimes L\right)\to \mathscr{E}^{\rm GG}_{k,k}\Omega_{X\times X}(\log (\pi_2^{-1}D))\otimes \pi_2^*L \]
 is the map induced by \(\pi_2\). Moreover, if one denotes \(\sigma_1:X\to \Delta\) the canonical lift, one has by definition that \(\iota=\sigma_{1,k}^*\) is just the isomorphism induced by \(\sigma_1\) and that therefore \(\iota\circ {\rm res}_{\Delta}\circ \pi_{2,k}^*=(\pi_2\circ \sigma_1)_k^*=({\rm id}_X)_k^*\) is just the identity on \(\mathscr{E}^{\rm GG}_{k,k}\Omega_X(U)\).
\end{proof}
\subsection{Logarithmic Wronskians}\label{sec:log wronskian} Let \(X\) be  an \(n\)-dimensional complex manifold endowed with a line bundle \(L\). Suppose that there exists a smooth hypersurface \(D\in |L|\) defined by a section \(\sigma_D\in H^0(X,L)\). Fix a positive integer \(k\geqslant 1\). Given \(s_1,\dots,s_k\in H^0(X,L)\) we define the logarithmic Wronskian to be 

\begin{equation}\label{eq:defLogWronskian}
        W_{\!D}(s_1,\ldots,s_k):=\begin{vmatrix}
	\nd^{1}(s_1) & \cdots & \nd^{1}(s_k)  \\
	\vdots & \ddots & \vdots \\
	\nd^{k}(s_1) & \cdots & \nd^{k}(s_k) 
	\end{vmatrix} \in H^0\big(X, E_{k,k'}^{\rm GG} \Omega_X(\log D)\otimes  L^{k}\big).
\end{equation}
We shall shortly see that in fact these elements define in fact global sections of \(E_{k,k'}\Omega_X(\log D)\otimes L^k\).
 We also define a morphism of \(\oc_X\)-modules 
\begin{eqnarray}\label{definition:Wronskian}
j^kW_{\!D}:\bigwedge^kJ^kL\to E_{k,k'}^{\rm GG}\Omega_X(\log D)\otimes L^k,
\end{eqnarray} 
by setting 
\[j^kW_D(g_1\wedge\dots\wedge g_n)=\begin{vmatrix}
	j^1\nd^{1}(g_1) & \cdots & j^1\nd^{1}(g_k)  \\
	\vdots & \ddots & \vdots \\
	j^k\nd^{k}(g_1) & \cdots & j^k\nd^{k}(g_k) 
	\end{vmatrix}\] 
	for $g_1,\ldots,g_k\in J^kL(U)$.
Here  we use for any \(1\leqslant j\leqslant k\) the inclusion \(E_{j,m}^{\rm GG}\Omega_X(\log D)\subset E_{k,m}^{\rm GG}\Omega_X(\log D)\) and the truncating morphism $J^kL\to J^jL$.
This construction is related to the (non-logarithmic) Wronskian  in the following way.
\begin{lem}\label{lem: non-log} Same notation as above. For any open subset \(U\subset X\) and any \(g_1, \dots, g_k\in  L(U)\) one has
\begin{equation*}
j^kW_L(j^k_L\sigma_D\wedge j^k_Lg_1\wedge\dots\wedge j^k_Lg_k)=\sigma_D\cdot j^kW_D(j^k_Lg_1\wedge \cdots\wedge j^k_Lg_k)
\end{equation*}

\end{lem} 
\begin{proof}On proves by using elementary operations on the lines and  Leibniz relation \eqref{Leibniz}, that for any  \(g_0, \dots, g_k\in L(U)\) one has
\begin{align*}
j^kW_L(j^k_Lg_0\wedge\dots\wedge j^k_Lg_k)
&=\begin{vmatrix}
	j^0\nd^{0}(j^0_Lg_0) & \cdots & j^0\nd^{0}(j^0_Lg_k)  \\
	\vdots & \ddots & \vdots \\
	j^k\nd^{k}(j^k_Lg_0) & \cdots & j^k\nd^{k}(j^k_Lg_k) 
	\end{vmatrix}.
\end{align*}
Then one applies this equality to \(g_0=\sigma_D\) and use relation \eqref{eq:VanishingNabla} and Proposition \ref{prop:DefJetLogConnexion} to prove that for any \(p\geqslant 1\), \(j^p\nd^p(j^p\sigma_D)=0\). The lemma follows by expanding the determinant with respect to the first column.
\end{proof}
In particular,  we see that the morphism \(j^kW_D\) factors through a morphism
\begin{eqnarray*}
\bigwedge^kJ^kL\to E_{k,k'}\Omega_X(\log D)\otimes L^k\hookrightarrow E_{k,k'}^{\rm GG}\Omega_X(\log D)\otimes L^k
\end{eqnarray*}
which we shall denote (slightly abusively) by \(j^kW_D\) in the rest of this paper. Therefore, by  Proposition  \ref{prop:DefJetLogConnexion} we obtain also that for any \(s_1,\dots, s_k\in H^0(X,L)\), \[W_{\!D}(s_1,\ldots,s_k)\in H^0\big(X,E_{k,k'}\Omega_X(\log D)\otimes L^k\big).\]
By   \eqref{eq:DirectImageFormula}, there exists a unique global section  
 \begin{align}\label{def:omega}
\omega_{\!D}(s_1,\ldots,s_k)\in H^0\big({X}_k(D),\oc_{{X}_k(D)}(k')\otimes {\pi}_{0,k}^*L^k\big)
\end{align}   
such that $({\pi}_{0,k})_*\omega_{\!D}(s_1,\ldots,s_k)=W_{\!D}(s_1,\ldots,s_k)$. These observations will be refined even further in the next section.

If $Y$ is a submanifold of $X$ which is transverse to $D$, then $(Y,Y\cap D)$ is a sub-log manifold of $(X,D)$. Write $D_Y:=Y\cap D$. One  has the following commutative diagram
\begin{align*}\label{dia:fonctorial}
\xymatrix{L\ar[r]^-{\nabla_{\!\!D}^k}\ar[d] & E^{\rm GG}_{k,k}\Omega_X(\log D)\otimes L \ar[d]\\
	L_{\upharpoonright Y} \ar[r]^-{\nabla_{\!\!D_Y}^k} & E^{\rm GG}_{k,k}\Omega_Y(\log D_Y)\otimes L_{\upharpoonright Y}.
}
\end{align*}
In particular, for any $s_1,\ldots,s_k\in H^0(X,L)$, one has
\begin{eqnarray}
W_{\!D}(s_1,\ldots,s_k)_{\upharpoonright Y}=W_{D_Y}(s_{1\upharpoonright Y},\ldots,s_{k\upharpoonright Y}).
\end{eqnarray}
Since the  log Demailly $k$-jet tower ${Y}_k(D_Y)$ of $(Y,D)$ can be seen as a smooth subvariety of ${X}_k(D)$, it follows that
\begin{eqnarray}\label{equ:fonctorial}
\omega_{\!D}(s_1,\ldots,s_k)_{\upharpoonright {Y}_k(D_Y)}=\omega_{D_Y}(s_{1 \upharpoonright Y},\ldots,s_{k \upharpoonright Y}). 
\end{eqnarray}
\subsection{Higher order log connections as local functions  on the log Demailly   tower}\label{sec:intrinsic} 
Take  \(X\), \(L\) and \(D\) as in the previous subsection. Fix a positive integer \(k\geqslant 1\). Consider the log Demailly $k$-jet tower $X_k(D)$ associated to \(\big(X,D,T_X(-\log D)\big)\).
Recall that 
 given  any \(s_1,\dots, s_k\in H^0(X,L)\),  one can associate  to \(W_D(s_1,\dots, s_k)\) a \emph{unique} element 
$\omega_{\!D}(s_1,\dots, s_k)\in H^0\big({X}_k(D),\oc_{{X}_k(D)}(k')\otimes \pi_{0,k}^*L^k\big)$
 by \eqref{def:omega}.  The drawback of using \eqref{eq:DirectImageFormula} is that the element \(\omega_{\!D}(s_1,\dots, s_k)\) is not fully explicit, since the isomorphism in \emph{loc. cit.}  is not completely explicit. To be more precise, on \(X\setminus D\) this isomorphism coincides with \eqref{local isomorphism}, and can therefore be understood in view of Theorem \ref{para} and   \eqref{define section}. However, \eqref{eq:DirectImageFormula} is only obtained indirectly in a neighborhood of a point of \(D\). On the other hand, during the proof of our main result, we will need an explicit description of \(\omega_{\!D}(s_1,\dots, s_k)\) at every point. For this reason, we provide here an alternative way, closer to Demailly's philosophy of directed pairs, to describe this element. To be more precise, we will construct an element 
\[\omega_{\!D}'(s_1,\dots, s_k)\in H^0\big({X}_k(D),\oc_{{X}_k(D)}(k,k-1,\dots,1)\otimes \pi_{0,k}^*L\big)\]
which is sent to \(\omega_{\!D}(s_1,\dots, s_k)\) under the canonical inclusion \(\oc_{{X}_k(D)}(k,k-1,\dots,1)\hookrightarrow \oc_{{X}_k(D)}(k')\) induced by multiplication by 
\[k{\pi}_{2,k}^*\Gamma_2+(k+k-1){\pi}_{3,k}^*\Gamma_3+\cdots+(k+\cdots+3){\pi}^*_{k-1,k}\Gamma_{k-1}+(k+\cdots+2)\Gamma_k,\]
where $\Gamma_j\in H^0\big(X_j(D), \oc_{{X}_j(D)}(1)\otimes \pi^*_{j-1,j}\oc_{{X}_{j-1}(D)}(-1)\big)$ is the effective divisor defined in \eqref{eq:Gamma_k}.

The starting point of our construction is the following. Let \((X,D,V)\) be a log directed manifold, and let \((\widetilde{X},\widetilde{D},\widetilde{V})\) be the derived log directed manifold via the 1-jet functor as defined in  \Cref{sec:LogDSTower}. Consider open subsets \(U\subset X\) and \(\widetilde{U}\subset \widetilde{X}\) such that \(\tilde{\pi}(\widetilde{U})\subset U\). Suppose  that we are given a trivialization of \(\oc_{\widetilde{X}}(-1)_{\upharpoonright \widetilde{U}}\) induced by a nowhere vanishing section \(\xi\in \Gamma\big(\widetilde{U},\oc_{\widetilde{X}}(-1)\big)\) and suppose moreover, that we are given a trivialization of \(L_{\upharpoonright U}\) under which the section \(\sigma_D\) corresponds to a holomorphic function \(\sigma_U\in \oc(U)\). Then, given any \(f\in \oc(U)\), we can define 
\[\nabla_Uf=df-f\frac{d\sigma_U}{\sigma_U}\in \Gamma(U,\Omega_X(D)).\]
Thus we obtain an element \(\tilde{\pi}^*\nabla_Uf_{\upharpoonright \widetilde{U}}\in \Gamma(\widetilde{U},\tilde{\pi}^*\Omega_X(\log D))\). Note that since we have inclusions \[\oc_{\widetilde{X}}(-1)\hookrightarrow \tilde{\pi}^*V\hookrightarrow \tilde{\pi}^*TX(-\log D),\]
we can see \(\xi\) as an element in \(\Gamma(\widetilde{U},\tilde{\pi}^*TX(-\log D))\). Therefore we can define 
\(\nabla^{\rm DS}_{U,\sigma_U,\widetilde{U},\xi}(f)\in \oc(\widetilde{U})\) by setting 
\[\nabla^{\rm DS}_{U,\sigma_U,\widetilde{U},\xi}(f)(a)=\tilde{\pi}^*\nabla_Uf_{\upharpoonright \widetilde{U}}\big(\xi(a)\big) \quad \forall a\in \tilde{U}.\]
Observe that this \(\tilde{f}\) depends strongly on the choice of trivializations.

This procedure can now be extended by induction on the higher order  log Demailly tower.  Consider a log-manifold \((X,D)\) and write \(L=\oc_X(D)\). Consider the log Demailly tower \(\big({X}_k(D),D_k,V_k\big)_{k\geqslant 0}\) associated to the log directed manifold \((X_0,D_0,V_0)=\big(X,D,T_X(-\log D)\big)\). A \emph{trivialization tower of oder \(k\)}, \(\mathfrak{U}=\big((U_0,\sigma_U),(U_j,\xi_j)_{1\leqslant j\leqslant k}\big)\), consists of the following data:
\begin{thmlist}
\item An open subset \(U\subset X\) and for each \(1\leqslant j\leqslant k\), an open subset \(U_j\subset {X}_j(D)\) such that \({\pi}_{0,1}(U_1)\subset U\) and \({\pi}_{j,j+1}(U_{j+1})\subset U_j\) whenever \(j< k\). 
\item A trivialization of \(L_{\upharpoonright U}\) under which the section \(\sigma_D\) corresponds to a holomorphic function \(\sigma_U\in \oc(U_0)\).
\item For every \(1\leqslant j\leqslant k\), a nowhere vanishing section \(\xi_j\in \Gamma\big(U_j,\oc_{{X}_j(D)}(-1)\big)\) which therefore induces a trivialization of \(\oc_{{X}_j(D)}(-1)_{\upharpoonright U_j}\).
\end{thmlist}
Let  \(\mathfrak{U}=\big((U,\sigma_U),(U_j,\xi_j)_{1\leqslant j\leqslant k}\big)\) be a trivialization tower of order \(k\) and let \(f\in \oc(U)\) be a holomorphic function on \(U\). Then on can define for any \(0\leqslant j \leqslant k\), a holomorphic function 
\begin{align*}
\onu^jf\in \oc(U_j)
\end{align*}
inductively by setting 
\begin{align}\nonumber
\onu^0f&=f\\\label{eq:trivial high}
\onu^{j+1}f&=\nabla^{\rm DS}_{U_j,{\pi}^*_{0,j}\sigma_U,U_{j+1},\xi_{j+1}}\big(\onu^{j}f\big)\ \ \ \forall\ \ 0\leqslant j<k,
\end{align}
where  we observe that $({\pi}^*_{0,j}\sigma_U=0)\cap U_j$ defines $D_j\cap U_j$. 
Here again these functions all depend in a critical way of the choice of trivialization tower \(\mathfrak{U}\).  

Consider now global sections \(s_1,\dots, s_k\in H^0(X,L)\). Let us fix a trivialization tower \(\mathfrak{U}\) and let \(s_{1,U},\dots, s_{k,U}\in \oc(U)\) be the local representatives of \(s_1,\dots, s_k\) under our choice of trivialization for \(L_{\upharpoonright U}\) (i.e. \(s_{i,U}=\sigma_U\frac{s_i}{\sigma_D}\)) and define 
\[\omega_{\mathfrak{U}}(s_1,\dots s_k)=
\left|\begin{array}{ccc}
\onu^1s_{1,U}&\cdots & \onu^1s_{k,U}\\
\vdots & & \vdots \\
\onu^ks_{1,U}&\cdots & \onu^ks_{k,U}
\end{array}\right|\in \oc(U_{k}).
\]
Here we abusively write \(\onu^js_{i,U}\) instead of \({\pi}_{j,k}^*\onu^js_{i,U}\in \oc(U_k)\) for any \(1\leqslant j<k\). The key point is that these locally defined objects can be glued together.
\begin{proposition}\label{prop:DSWronskian} For any \(s_1,\dots, s_k\in H^0(X,L)\), the family of holomorphic functions \(\big(\omega_{\mathfrak{U}}(s_1,\dots,s_k)\big)_{\mathfrak{U}}\) define a global section 
\[\omega_{\!D}'(s_1,\dots,s_k)\in H^0\big({X}_k(D),\oc_{{X}_k(D)}(k,k-1,\dots,1)\otimes \pi_{0,k}^*L^k\big).\]
More precisely, for any trivialization tower $\mathfrak{U}$ of order $k$, one has 
\begin{align}\label{eq:another wronskian}
\omega_{\!D}'(s_1,\dots,s_k)_{\upharpoonright U_k}=\omega_{\mathfrak{U}}(s_1,\dots s_k)\cdot({\pi}_{1,k})^*\xi_{1}^{-k}\cdot ({\pi}_{2,k})^*\xi_{2}^{-(k-1)}\cdots({\pi}_{k-1,k})^*\xi_{k-1}^{-2}\cdot \xi_{k}^{-1}.
\end{align}
Moreover, under the natural inclusion 
\begin{align*}
H^0\big({X}_k(D),\oc_{{X}_k(D)}(k,k-1,\dots,1)\otimes \pi_{0,k}^*L^k\big)\hookrightarrow  H^0\big({X}_k(D),\oc_{{X}_k(D)}(k')\otimes \pi_{0,k}^*L^k\big)\\
\stackrel{\eqref{eq:DirectImageFormula}}{\simeq} H^0\big(X,E_{k,k'}\Omega_X(\log D)\otimes L^k\big),
\end{align*}
the element \(\omega_{\!D}'(s_1,\dots,s_k)\) is sent to \(W_D(s_1,\dots,s_k)\).
\end{proposition}
The proof of this result relies on the following technical lemma.
\begin{lem}\label{lem:Beta} For any  trivializing tower \(\mathfrak{U}=\big((U,\sigma_U),(U_j,\xi_j)_{1\leqslant j\leqslant k}\big)\) and any integers  \(1\leqslant j<p\leqslant k\), there exists a holomorphic function \(\beta_{j,p}\in \mathscr{O}({U}_{p+1})\) such that for any \(f\in \mathscr{O}(U)\) one has
\[\nabla^{\rm DS}_{(U_p,{\pi}_p^*\sigma_U,U_{p+1},\xi_{p+1})}\big({\pi}_{j,p}^*\onu^j(f)\big)=\beta_{j,p}\cdot {\pi}_{j+1,p+1}^*\onu^{j+1}(f)_{\upharpoonright U_{p+1}}.\]
\end{lem}

\begin{proof}
By definition of the log Demailly jet tower, the differential of the map  \({\pi}_{j,p}\) induces a morphism
\[d{\pi}_{j,p}:\mathscr{O}_{X_{p+1}(D)}(-1)\to \pi_{j+1,p+1}^*\oc_{{X}_{j+1}(D)}(-1).\]
Over the open subset \(U_{p+1}\), since \(\xi_{j+1}\) is nowhere vanishing on \(U_{j+1}\), there exists \(\beta_{j,p}\in \oc(U_{p+1})\) such that
\[d{\pi}_{j,p}(\xi_{p+1})=\beta_{j,p}{\pi}^*_{j+1,p+1}(\xi_{j+1}).\] 
Let us write \(g=\onu^j(f)\in\oc(U_j)\) for simplicity, so that by definition \eqref{eq:trivial high}
\[\onu^{j+1}(f)=\nabla^{\rm DS}_{U_j,{\pi}^*_{0,j}\sigma_U,U_{j+1},\xi_{j+1}}(g)={\pi}^*_{j,j+1}\left(dg-g\frac{d{\pi}^*_{0,j}\sigma_U}{{\pi}^*_{0,j}\sigma_U}\right)(\xi_{j+1}).\]
The proof of the lemma is then reduced to the following computation:
\begin{align*}
\nabla^{\rm DS}_{(U_p,{\pi}_p^*\sigma_U,U_{p+1},\xi_{p+1})}\big({\pi}_{j,p}^*g\big)&={\pi}^*_{p,p+1}\left(d{\pi}_{j,p}^*g-{\pi}_{j,p}^*g\frac{d{\pi}^*_{0,p}\sigma_U}{{\pi}^*_{0,p}\sigma_U}\right)(\xi_{p+1})\\
&={\pi}^*_{p,p+1}\left({\pi}_{j,p}^*dg-{\pi}_{j,p}^*g\cdot{\pi}_{j,p}^*\frac{d{\pi}^*_{0,j}\sigma_U}{{\pi}^*_{0,j}\sigma_U}\right)(\xi_{p+1})\\
&=\left[{\pi}^*_{j+1,p+1}{\pi}_{j,j+1}^*\left(dg-g\frac{d{\pi}^*_{0,j}\sigma_U}{{\pi}^*_{0,j}\sigma_U}\right)\right](\xi_{p+1})\\
&={\pi}^*_{j+1,p+1}{\pi}_{j,j+1}^*\left[\left(dg-g\frac{d{\pi}^*_{0,j}\sigma_U}{{\pi}^*_{0,j}\sigma_U}\right)(d{\pi}_{j,p}(\xi_{p+1}))\right]\\
&={\pi}^*_{j+1,p+1}\left[{\pi}_{j,j+1}^*\left(dg-g\frac{d{\pi}^*_{0,j}\sigma_U}{{\pi}^*_{0,j}\sigma_U}\right)\right](\beta_{j,p}\xi_{j+1})\\ 
&=\beta_{j,p}{\pi}_{j+1,p+1}^*\onu^{j+1}(f). \qedhere
\end{align*}
\end{proof}

\begin{proof}[Proof of Proposition \ref{prop:DSWronskian}]
Consider two trivialization towers \[\mathfrak{U}^1=\big((U^1,\sigma_{U^1}),(U^1_j,\xi^1_j)_{1\leqslant j\leqslant k}\big)\ \ \text{and}\ \ \mathfrak{U}^2=\big((U^2,\sigma_{U^2}),(U^2_j,\xi^2_j)_{1\leqslant j\leqslant k}\big).\]
 Writing \(U^{12}=U^1\cap U^2\) Let \(g\in \oc(U^{12})\) be the transition map   from \(U^2\) to \(U^1\) induced by our choice of trivializations for \(L\), so that for any global section \(s\in H^0(X,L)\), \begin{equation}\label{relationp0}s_{U^1 \upharpoonright U^{12}}=g\cdot s_{U^2 \upharpoonright U^{12}}.\end{equation}  
For any \(1\leqslant j\leqslant k\), let us also write \(U_j^{12}=U_j^1\cap U_j^2\) and consider the function \(\theta_j\in \oc(U_j^{12})\) such that \[\xi_j^1=\theta_j\xi_j^2.\] Therefore \(\theta_j\) is the transition map from \(U_j^2\) to \(U_j^1\) for the trivializations \(\oc_{{X}_j(D)}(1)_{\upharpoonright U^2_j}\cong U^2_j\times \mathbb{C}\) and \(\oc_{{X}_j(D)}(1)_{\upharpoonright U^1_j}\cong U^1_j\times \mathbb{C}\) induced by \(\xi_j^2\) and \(\xi_j^1\) respectively.
We are now going to establish that for any \(0\leqslant  p\leqslant k\), and for any \(0\leqslant j<p\), there exists \(P_j^p\in \oc(U^p)\) a holomorphic function such that for any \(s\in H^0(X,L)\) one has 
\begin{equation}\label{eq:PFormula}{\nabla}_{\mathfrak{U}^1}^p(s_{U^1})=\theta_p\theta_{p-1}\cdots \theta_1g{\nabla}_{\mathfrak{U}^2}^p(s_{U^2})+\sum_{j=0}^{p-1}P_j^p{\nabla}_{\mathfrak{U}^2}^j(s_{U^2}).\end{equation}
The key point in this formula is that \(P^p_j\) does not depend on \(s\). From this, and from elementary operations on the lines in the determinant defining \(\omega_{\mathfrak{U}^i}(s_1,\dots, s_k)\), it will follow that 
\[\omega_{\mathfrak{U}^1}(s_1,\dots, s_k)_{\upharpoonright U_k^{12}}=\theta_k\theta_{k-1}^2\cdots \theta_2^{k-1}\theta_1^kg\omega_{\mathfrak{U}^2}(s_1,\dots, s_k)_{\upharpoonright U_k^{12}}\]
which concludes the proof of the first statement of the proposition.

We will establish \eqref{eq:PFormula} by induction on \(p\). For \(p=0\) this is just \eqref{relationp0}. Take \(0\leqslant p < k\) and suppose that formula \eqref{eq:PFormula} holds for \(p\).  Take \(s\in H^0(X,L)\). Recall that 
\[{\nabla}_{\mathfrak{U}^1}^{p+1}(s_{U^1})=\nabla^{\rm DS}_{U^{12}_{p},{\pi}_{0,p}^*\sigma_{U^1},U^{12}_{p+1},\xi_{p+1}^1}\big({\nabla}_{\mathfrak{U}^1}^p(s_{U^1})\big).\] 
On the other hand, one has
\begin{align*}
\nabla^{\rm DS}_{U^{12}_{p},{\pi}_{0,p}^*\sigma_{U^1},U^{12}_{p+1},\xi_{p+1}^1}\big({\nabla}_{\mathfrak{U}^1}^p(s_{U^1})\big)&=\nabla^{\rm DS}_{U^{12}_{p},{\pi}_{0,p}^*(g\sigma_{U^2}),U^{12}_{p+1},\theta_{p+1}\xi_{p+1}^2}\big({\nabla}_{\mathfrak{U}^1}^p(s_{U^1})\big)\\
&=\theta_{p+1}\nabla^{\rm DS}_{U^{12}_{p},{\pi}_{0,p}^*(g\sigma_{U^2}),U^{12}_{p+1},\xi_{p+1}^2}\big({\nabla}_{\mathfrak{U}^1}^p(s_{U^1})\big)\\
&=\theta_{p+1}\nabla^{\rm DS}_{U^{12}_{p},{\pi}_{0,p}^*\sigma_{U^2},U^{12}_{p+1},\xi_{p+1}^2}\big({\nabla}_{\mathfrak{U}^1}^p(s_{U^1})\big)-\theta_{p+1}{\nabla}_{\mathfrak{U}^1}^p(s_{U^1}){\pi}^*_{p,p+1}\frac{d{\pi}^*_{0,p}g}{{\pi}^*_{0,p}g}(\xi^2_{p+1}).
\end{align*}
Observe that, using our induction hypothesis, the term \(-\theta_{p+1}{\nabla}_{\mathfrak{U}^1}^p(s_{U^1}){\pi}^*_{p,p+1}\frac{d{\pi}^*_{0,p}g}{{\pi}^*_{0,p}g}(\xi^2_{p+1})\) is  of the form allowed in formula \eqref{eq:PFormula}  to be considered as an error term for the rank \(p+1\). Therefore it only remains to prove that \(\theta_{p+1}\nabla^{\rm DS}_{U^{12}_{p},{\pi}_{0,p}^*\sigma_{U^2},U^{12}_{p+1},\xi_{p+1}^2}\big({\nabla}_{\mathfrak{U}^1}^p(s_{U^1})\big)\) is of the form announced in \eqref{eq:PFormula}. To lighten the  notation we will now write \(\nabla^{\rm DS}=\nabla^{\rm DS}_{U^{12}_{p},{\pi}_{0,p}^*\sigma_{U^2},U^{12}_{p+1},\xi_{p+1}^2}\). By induction one has 
\begin{align*}
\theta_{p+1}\nabla^{\rm DS}\big({\nabla}_{\mathfrak{U}^1}^p(s_{U^1})\big)&=\theta_{p+1}\nabla^{\rm DS}\Big(\theta_p\cdots \theta_1g{\nabla}_{\mathfrak{U}^2}^p(s_{U^2})+\sum_{j=1}^pP_j^p\pi_{j,p}^*{\nabla}_{\mathfrak{U}^2}^j(s_{U^2})\Big)\\
&=\theta_{p+1}\nabla^{\rm DS}\big(\theta_p\cdots \theta_1g{\nabla}_{\mathfrak{U}^2}^p(s_{U^2})\big) +\sum_{j=1}^p\theta_{p+1}\nabla^{\rm DS}\big(P_j^p{\nabla}_{\mathfrak{U}^2}^j(s_{U^2})\big).
\end{align*}
Before continuing, observe that for any \(f_1,f_2\in \oc(U_p^{12})\), one has (by an immediate computation) 
\[\nabla^{\rm DS}(f_1f_2)=f_1\nabla^{\rm DS}(f_2)+f_2 df_1(\xi_{p+1}^2).\]
Applying this to \(f_1=\theta_p\cdots \theta_1g\) and \(f_2={\nabla}_{\mathfrak{U}^2}^p(s_{U^2})\) we obtain 
\begin{align*}\theta_{p+1}\nabla^{\rm DS}\big(\theta_p\cdots \theta_1g{\nabla}_{\mathfrak{U}^2}^p(s_{U^2})\big)&=\theta_{p+1}\theta_p\cdots \theta_1g\nabla^{\rm DS}\big({\nabla}_{\mathfrak{U}^2}^p(s_{U^2})\big)+{\nabla}_{\mathfrak{U}^2}^p(s_{U^2})d(\theta_p\cdots \theta_1g)(\xi_{p+1}^2)\\
&=
\theta_{p+1}\theta_p\cdots \theta_1g{\nabla}_{\mathfrak{U}^2}^{p+1}(s_{U^2})+{\nabla}_{\mathfrak{U}^2}^p(s_{U^2})d(\theta_p\cdots \theta_1g)(\xi_{p+1}^2).\end{align*}
Observe that the term \({\nabla}_{\mathfrak{U}^2}^p(s_{U^2})d(\theta_p\cdots \theta_1g)(\xi_{p+1}^2)\) is of the form allowed in the last term of formula \eqref{eq:PFormula} at rank \(p+1\). Therefore, the proof of formula \eqref{eq:PFormula} will be completed if one proves that for each \(j<p\) the term \(\theta_{p+1}\nabla^{\rm DS}\big(P_j^p{\nabla}_{\mathfrak{U}^2}^j(s_{U^2})\big)\) is also of the form of an error term if  \eqref{eq:PFormula} at rank \(p+1\). To see this, observe that for each \(j<p\) one has
\begin{align*}
\theta_{p+1}\nabla^{\rm DS}\big(P_j^p{\nabla}_{\mathfrak{U}^2}^j(s_{U^2})\big)&= P_j^p\theta_{p+1}\nabla^{\rm DS}\big({\nabla}_{\mathfrak{U}^2}^j(s_{U^2})\big)+{\nabla}_{\mathfrak{U}^2}^j(s_{U^2})dP_j^p(\xi_{p+1}^2)\\
&=  P_j^p\theta_{p+1}\beta_{j,p}{\nabla}_{\mathfrak{U}^2}^{j+1}(s_{U^2})+{\nabla}_{\mathfrak{U}^2}^j(s_{U^2})dP_j^p(\xi_{p+1}^2),
\end{align*}
where \(\beta_{j,p}\) is the function appearing in Lemma \ref{lem:Beta} applied to the trivialization tower \(\mathfrak{U}^2\). This concludes the proof of \eqref{eq:PFormula}.

To conclude the proof of the proposition, it remains to prove that \(\omega_{\!D}'(s_1,\dots, s_k)\) is sent  to \(W_D(s_1,\dots,s_k)\) under the above natural map. By continuity, it suffices to prove this over the open subset \(X_k(D)^{\rm reg}\cap {\pi}_k^{-1}(X\setminus D)\).  Take \(w_k\in X_k(D)^{\rm reg}\cap {\pi}_k^{-1}(X\setminus D)\). From the previous part of the proposition, we are allowed to choose any trivialization tower in order to make the computation of \(\omega'_{\!D}(s_1,\dots,s_k)\) in a neighborhood of \(w_k\). On the other hand, to compute the element \(\omega_{\!D}(s_1,\dots, s_k)\in H^0({X}_k(D),\oc_{{X}_k(D)}(k')\otimes \pi_{0,k}^*L)\) associated to \(W_D(s_1,\dots, s_k)\) under the isomorphism \eqref{eq:DirectImageFormula}, we are allowed to use the explicit description isomorphism \eqref{local isomorphism}. Indeed, outside \(D\) the logarithmic and absolute jet towers coincide.  Let us therefore apply Theorem \ref{para}.

Let $U\subset X\setminus D$ be an open set with local coordinates $(z_1,\ldots,z_n)$. Take a trivialization of $L_{\upharpoonright U}$ such that $\sigma_D$ is identically equal to $1$. It follows from \cite[Proof of Theorem 6.8]{Dem95} that ${X}_k(D)^{\rm reg}\bigcap {\pi}^{-1}_{0,k}(U)$ can be covered by   open sets $U\times \cb^{(n-1)k}$. Indeed, consider the family of holomorphic curves 
\begin{align*}
\gamma:U\times\cb^{(n-1)k}\times \cb&\rightarrow  U
\\
(z,w,t)&\mapsto \gamma_{(w,z)}(t)
\end{align*}
defined by
\[\gamma_{(w,z)}^i(t)= \begin{cases}
z_i+\frac{w^{(1)}_i}{1!}t+\frac{w^{(2)}_i}{2!}t^2+\cdots+\frac{w^{(k)}_i}{k!}t^k & \mbox{ if  }\  1\leqslant i\leqslant n-1\\
z_n+t & \mbox{ if  }\  i=n,
\end{cases}\]
where  \(w:=(w_i^{(j)})_{1\leqslant i\leqslant n-1}^{1\leqslant j\leqslant k}\) and  \(\gamma_{(w,z)}:=(\gamma_{(w,z)}^1,\dots, \gamma_{(w,z)}^n)\). To be precise, the map \(\gamma\) is only defined on a open neighborhood of \(U\times\cb^{(n-1)k}\times \{0\}\) in \(U\times\cb^{(n-1)k}\times  \cb\), but this subtlety will be irrelevant as we will only consider the  \(k\)-jets of each \(\gamma_{(w,z)}\) at the point \(0\).

In this setting,  we will prove that its $k$-th lift \((\gamma_{(w,z)})_{[k]}(0)\) gives a holomorphic embedding
 \begin{eqnarray}\label{parameter lift}
\tau_k:U\times\cb^{(n-1)k}\rightarrow  {X}_k(D)^{\rm reg}
\end{eqnarray} whose image is an open subset.

Let us take a special trivialization tower of order $k$, denoted by \(\mathfrak{U}=\big((U_0,\sigma_U),(U_j,\xi_j)_{1\leqslant j\leqslant k}\big)\) in the following way:
\begin{thmlist}
	\item $(U_0,\sigma_U)=(U,1)$.
	\item Set $U_1=\Big\{\Big(\big[z_1^{(1)}\frac{\d}{\d z_1}+\cdots+z_{n-1}^{(1)}\frac{\d}{\d z_{n-1}}+\frac{\d}{\d z_{n}}\big];z\Big)\in \pt(T_U)\Big\}\simeq \cb^{n-1}\times U$ with the coordinate $\Big(z;z_1^{(1)},\ldots,z_{n-1}^{(1)}\Big)$. Define \(\xi_1=z_1^{(1)}\frac{\d}{\d z_1}+\cdots+z_{n-1}^{(1)}\frac{\d}{\d z_{n-1}}+\frac{\d}{\d z_{n}}\in \Gamma\big(U_1,\oc_{ {X}_1(D)}(-1)\big)\), and take the basis for $V_{1\upharpoonright U_1}$ as
	\[
	e^{(1)}_1=\frac{\d}{\d z^{(1)}_1},  \ldots, \ \ \  e^{(1)}_{n-1}=\frac{\d}{\d z^{(1)}_{n-1}},\ \  \ e^{(1)}_n=\frac{\d}{\d z_{n}}+z_1^{(1)}\frac{\d}{\d z_1}+\cdots+z_{n-1}^{(1)}\frac{\d}{\d z_{n-1}}.
	\]
	Then one has $({\pi}_{0,1})_*\big(e^{(1)}_n\big)=\xi_1$, and ${\pi}_{0,1}\big(z;z_1^{(1)},\ldots,z_{n-1}^{(1)}\big)=z$.
	\item Set $U_2=\{([z_1^{(2)}e^{(1)}_1+\cdots+z_{n-1}^{(2)}e^{(1)}_{n-1}+ e^{(1)}_n])\in \pt({V}_1)_{\upharpoonright U_1}\}\simeq U\times \cb^{2(n-1)}$ with the coordinate $\big(z;z^{(1)},z_1^{(2)},\ldots,z_{n-1}^{(2)}\big)$. Here we write $z^{(1)}=(z_1^{(1)},\ldots,z_{n-1}^{(1)})$ for short. Define \(\xi_2=z_1^{(2)}e^{(1)}_1+\cdots+z_{n-1}^{(2)}e^{(1)}_{n-1}+ e^{(1)}_n\in \Gamma\big(U_2,\oc_{X_2}(-1)\big)\), and take the basis for $V_{2\upharpoonright U_2}$ as
	\[
	e^{(2)}_1=\frac{\d}{\d z^{(2)}_1}, \ldots,\ \ \ e^{(2)}_{n-1}=\frac{\d}{\d z^{(2)}_{n-1}},\ \ \ e^{(2)}_n=\frac{\d}{\d z_{n}}+z_1^{(1)}\frac{\d}{\d z_1}+\cdots+z_{n-1}^{(1)}\frac{\d}{\d z_{n-1}}+z_1^{(2)}\frac{\d}{\d z_1^{(1)}}+\cdots+z_{n-1}^{(2)}\frac{\d}{\d z_{n-1}^{(1)}}.
	\]
	Then one has $({\pi}_{1,2})_*\big(e^{(2)}_n\big)=\xi_2$, and ${\pi}_{1,2}\big(z;z^{(1)},z_1^{(2)},\ldots,z_{n-1}^{(2)}\big)=\big(z;z^{(1)}\big)$. 
	\item Inductively, one can define $U_p$ with coordinates \(\big(z;z^{(1)},\ldots,z^{(p)})\in U\times \cb^{(n-1)p} \) such that for any $j<p$, 
	\({\pi}_{j,p}\big(z;z^{(1)},\ldots,z^{(p)}\big)=\big(z;z^{(1)},\ldots,z^{(j)}\big)\) and \[\xi_p=\frac{\d}{\d z_{n}}+z_1^{(1)}\frac{\d}{\d z_1}+\cdots+z_{n-1}^{(1)}\frac{\d}{\d z_{n-1}}+z_1^{(2)}\frac{\d}{\d z^{(1)}_1}+\cdots+z_{n-1}^{(2)}\frac{\d}{\d z^{(1)}_{n-1}}+\cdots+ z_1^{(p)}\frac{\d}{\d z^{(p-1)}_1}+\cdots+z_{n-1}^{(p)}\frac{\d}{\d z^{(p-1)}_{n-1}} \]
	Then $({\pi}_{j,p})_*(\xi_p)=\xi_j$. In particular, by \eqref{eq:short embedding} and \eqref{eq:Gamma_k}, for any $1\leqslant j<p$, there is an isomorphism
\begin{eqnarray*}\label{dia:same trivialization}
	\xymatrix{
	\oc_{X_p(D)}(-1)_{\upharpoonright U_p}\ar[rr]^-{({\pi}_{j,p})^*({\pi}_{j-1,p-1})_*}_-{\simeq}&&({\pi}_{j,p})^*\oc_{{X}_j(D)}(-1)_{\upharpoonright U_p}\\
	&	\oc_{U_p}\ar[ul]^-{\cdot \xi_p}_-{\simeq}\ar[ur]_-{\cdot ({\pi}_{j,p})^*\xi_j}^-{\simeq}&
}
\end{eqnarray*}
and
\begin{eqnarray}\label{eq:relation}
\xi_j^{-1}\cdot ({\pi}_j)^*\xi_{j-1}=\xi_j^{-1}\cdot ({\pi}_j)^*({\pi}_{j-1})_*(\xi_j)=\Gamma_{j \upharpoonright U_j}.
\end{eqnarray}
\end{thmlist}
In this setting, one can prove that, within the coordinates for $U_k$, the $k$-th lift \((\gamma_{(w,z)})_{[k]}(0)=(z;w) \). Hence $\tau_k:\cb^{(n-1)k}\times U\rightarrow X_k(D)^{\rm reg}$ whose image is the open subset $U_k$, and under the trivialization tower of order $k$, $\tau_k$ is an identity map. Moreover,  \((\gamma_{(w,z)})_{[k-1]}'(0)=\xi_{k \upharpoonright (\gamma_{(w,z)})_{[k]}(0)} \). Hence a straightforward computation  shows that
	if we identify the parameter space $\cb^{(n-1)k}\times U$ of $\gamma_{(w,z)}(t)$ with $U_k$  by \begin{align*}
	\tau_k:\cb^{(n-1)k}\times U&\rightarrow  U_k\\
	\big(w^{(1)},\ldots,w^{(k)};z\big)&\mapsto  (\gamma_{(w,z)})_{[k]}(0), 
	\end{align*}
	then for any $f\in \Gamma(U,\oc_U)$ and for any \(j=1,\dots, k\) one  has
	\[
	d^{j}(f)(j_k\gamma_{(w,z)})={\nabla}_{\mathfrak{U}}^j(f)(w,z)\in  \oc({U_j}).
	\]

For any $s_1,\ldots,s_k\in\Gamma(U,L_{\upharpoonright U})$. Write \(s_{1,U},\dots, s_{k,U}\in  \oc(U) \) for the local representatives of \(s_1,\dots, s_k\) under our choice of trivialization for \(L_{\upharpoonright U}\) (i.e. \(s_{i,U}=\frac{s_i}{\sigma_D}\)). Then by \eqref{define section}, \(\omega_{\!D}(s_{1,U},\dots,s_{k,U})\in  \Gamma\big(U_k,\oc_{{X}_k(D)}(k')\big) \) is defined by
\begin{align*}
\omega_{\!D}(s_{1},\dots,s_{k})(w,z)&=W_D(s_{1,U},\dots,s_{k,U})\big(j_k\gamma_{(w,z)}\big)\cdot \big((\gamma_{(w,z)})_{[k-1]}'(0)\big)^{-k'}\\
&=\left|\begin{array}{ccc}
\nd^1(s_{1,U})&\cdots & \nd^1(s_{k,U})\\
\vdots &\ddots & \vdots \\
\nd^k(s_{1,U})&\cdots & \nd^k(s_{k,U})
\end{array}\right|\big(j_k\gamma_{(w,z)}\big)\cdot (\xi_k)^{-k'}\\
&=\left|\begin{array}{ccc}
d^{1}(s_{1,U})&\cdots & d^{1}(s_{k,U})\\
\vdots &\ddots & \vdots \\
d^{k}(s_{1,U})&\cdots & d^{k}(s_{k,U})
\end{array}\right|\big(j_k\gamma_{(w,z)}\big)\cdot (\xi_k)^{-k'}
\end{align*}
where the last equality is due to $\sigma_U=1$. Note that $d^{j}:\Gamma(U,\oc_U)\rightarrow E_{j,j}^{\rm GG}\Omega_U$.  
Write
\begin{align}\label{eq:exceptional}
\mathbf{\Gamma}_k=k{\pi}_{2,k}^*\Gamma_2+(k+k-1){\pi}_{3,k}^*\Gamma_3+\cdots+(k+\cdots+3){\pi}^*_{k-1,k}\Gamma_{k-1}+(k+\cdots+2)\Gamma_k
\end{align} for short. Hence
\begin{align*}
	\omega_{\!D}(s_{1},\dots,s_{k})_{\upharpoonright U_k}&= \left|\begin{array}{ccc}
		{\nabla}^{1}_{\mathfrak{U}}(s_{1,U})&\cdots & {\nabla}^{1}_{\mathfrak{U}}(s_{k,U})\\
		\vdots &\ddots & \vdots \\
		{\nabla}^{k}_{\mathfrak{U}}(s_{1,U})&\cdots & {\nabla}^{k}_{\mathfrak{U}}(s_{k,U})
	\end{array}\right|\cdot (\xi_k)^{-k'}\\
&=\omega_{\mathfrak{U}}(s_1,\dots s_k)\cdot  (\xi_k)^{-k'}\\
&\stackrel{\eqref{eq:relation}}{=}\omega_{\mathfrak{U}}(s_1,\dots s_k)\cdot({\pi}_{1,k})^*\xi_{1}^{-k}\cdot ({\pi}_{2,k})^*\xi_{2}^{-(k-1)}\cdots({\pi}_{k-1,k})^*\xi_{k-1}^{-2}\cdot \xi_{k}^{-1}\cdot  \mathbf{\Gamma}_k\\
&\stackrel{\eqref{eq:another wronskian}}{=}\omega_{\!D}'(s_1,\dots,s_k)_{\upharpoonright U_k} \cdot  \mathbf{\Gamma}_k.
\end{align*}
 Since \(U_k\) is dense in \(\pi_{0,k}^{-1}(U)\) and since \(\omega_{\!D}(s_{1},\dots,s_{k})\) and \(\omega_{\!D}'(s_1,\dots,s_k) \cdot \mathbf{\Gamma}_k \) coincide on \(U_k\), it follows by continuity that they also coincide on \(\pi_{0,k}^{-1}(U)\). Therefore it follows that these two sections coincide on \(\pi_{0,k}^{-1}(X\setminus D)\) and therefore they also coincide on the whole space \(X_k(D)\):
\begin{align}\label{eq:relation of two wronskian}
\omega_{\!D}(s_{1},\dots,s_{k})=\omega_{\!D}'(s_1,\dots,s_k) \cdot \mathbf{\Gamma}_k.
\end{align}
The proposition is thus proved.
\end{proof}

\subsection{Logarithmic Wronskian ideal sheaf}\label{sec:wronskian ideal}
Recall that in \eqref{definition:Wronskian}, we defined the \emph{log Wronskian morphism}
\[
j^kW_{\!D}:\bigwedge^kJ^kL\to E_{k,k'}\Omega_X(\log D)\otimes L^k,
\]
which is a morphism of  $\oc_X$-module. We denote by $\mathscr{W}_{k,L}:=j^kW_D(\bigwedge^kJ^kL)$, which is a subsheaf of  $E_{k,k'}\Omega_X(\log D)\otimes L^k$. 
By \eqref{eq:DirectImageFormula} for any $m\in \mathbb{N}$ there exists a natural morphism 
\begin{align*}
 {\pi}_{0,k}^{-1}\big(E_{k,k'}\Omega_X(\log D)\otimes L^m\big)\simeq {\pi}_{0,k}^{-1}({\pi}_{0,k})_*\big(\oc_{{X}_k(D)}(k')\otimes \pi_{0,k}^*L^m\big)\to \oc_{{X}_k(D)}(k')\otimes \pi_{0,k}^*L^m.
\end{align*}
We denote by $ {\wk}_{{X}_k(D)}$ the image of the composition
\[
{\pi}_{0,k}^{-1}\mathscr{W}_{k,L}\otimes \oc_{{X}_k(D)}(-k')\otimes \pi_{0,k}^*L^{-k}\hookrightarrow{\pi}_{0,k}^{-1}E_{k,k'}\Omega_X(\log D)\otimes \oc_{{X}_k(D)}(-k')\to \oc_{{X}_k(D)}, 
\]
which is a coherent ideal sheaf on ${{X}_k(D)}$.  $ {\wk}_{{X}_k(D)}$ will be called the \emph{$k$-th logarithmic Wronskian ideal sheaf} associated to the log manifold $(X,D)$. Let us denote by 
\[
\mathbb{W}_{{X}_k(D)}:={\rm Span}\{\omega_{\!D}(s_1,\ldots,s_k)\mid s_1,\ldots,s_k\in H^0(X,L)\}
\]
the sub-linear system of $|\oc_{{X}_k(D)}(k')\otimes \pi_{0,k}^*L^k|$. 
One thus has the following result
\begin{proposition}\label{prop:coincide}
	When $L$ generates $k$-jets everywhere on $X$, $ {\wk}_{{X}_k(D)}$ is the base ideal of \(\mathbb{W}_{{X}_k(D)}\). It satisfies moreover
	$$
	{\rm Supp}(\oc_{{X}_k(D)}/{\wk}_{{X}_k(D)})\subset X_k(D)^{\rm sing}\bigcup \pi_{0,k}^{-1}(D).
	$$ 
\end{proposition}
\begin{proof}
	For any $s_1,\ldots,s_k\in H^0(X,L)$, let us define the natural linear map
\begin{eqnarray*}
	\bigwedge^kH^0(X,L)&\rightarrow& H^0(X,\bigwedge^kJ^kL)\\
	s_1\wedge\cdots\wedge s_k &\mapsto& j_L^ks_1\wedge \cdots \wedge j_L^ks_k.
\end{eqnarray*}
It follows from the definition that
\[
W_D(s_1,\ldots,s_k)=j^kW_D\big(j_L^ks_1\wedge \cdots \wedge j_L^ks_k \big)\in H^0(X,\mathscr{W}_{k,L}).
\]
Hence by the definition of $ {\wk}_{{X}_k(D)}$ and the fact $({\pi}_{0,k})_* \omega_{\!D}(s_1,\ldots,s_k) =W_D(s_1,\ldots,s_k)$, one concludes that
\[
\omega_{\!D}(s_1,\ldots,s_k)\in H^0\big({X}_k(D),\oc_{{X}_k(D)}(k')\otimes {\pi}_{0,k}^*L^k\otimes   {\wk}_{{X}_k(D)}  \big).
\]
In other words, the base ideal of $\mathscr{W}_{k,L}$ belongs to $ {\wk}_{{X}_k(D)}$.

\medskip

On the other hand, since $L$ separates $k$-jets everywhere on $X$, the set of global sections  \[{\rm Span}\{j_L^ks_1\wedge \cdots \wedge j_L^ks_k\in H^0(X,\bigwedge^kJ^kL)\mid s_1,\ldots,s_k\in H^0(X,L) \}\] thus generates the locally free sheaf $\bigwedge^kJ^kL$ everywhere on $X$. Recall that $J^kW_D:\bigwedge^kJ^kL\rightarrow \mathscr{W}_{k,L}$ is a surjective morphism between sheaves of $\oc_X$-modules. Therefore, the set of sections 
\[
{\rm Span}\{j^kW_D\big(j_L^ks_1\wedge \cdots \wedge j_L^ks_k\big)\in H^0(X,\mathscr{W}_{k,L})\mid s_1,\ldots,s_k\in H^0(X,L) \}
\]
generates  the sheaf of $\oc_X$-module $\mathscr{W}_{k,L}$, which implies that $ {\wk}_{{X}_k(D)}$ belongs to the base ideal of $\mathbb{W}_{{X}_k(D)}$ by the definition of ${\wk}_{{X}_k(D)}$.  In conclusion,   the base ideal of \(\mathbb{W}_{{X}_k(D)}\) is $ {\wk}_{{X}_k(D)}$. The second assertion follows from \cite[Lemma 2.4]{Bro17}.
\end{proof}

Let us now give a more detailed local description of these objects.
Let $\mathbb{D}^{\!n}$ be the (unit) polydisc, and denote by  $E:=\{(z_1,\ldots,z_n)\in \mathbb{D}^{\!n}\mid z_1=0 \}$.  
As in \eqref{definition:Wronskian}, we define a  morphism of $\oc_{\mathbb{D}^{\!n}}$-module associated to the log pair $(\mathbb{D}^{\!n},E)$
\begin{eqnarray*}
J^kW_E:\bigwedge^{k}J^k\oc_{\mathbb{D}^{\!n}} \rightarrow E_{k,k'}\Omega_{\mathbb{D}^{\!n}}(\log E).
\end{eqnarray*}
 Set
\[
J^kW: \bigwedge^{k}J^k\oc_{\mathbb{D}^{\!n}} \rightarrow E_{k,k'}\Omega_{\mathbb{D}^{\!n}}
\]
to be the morphism of $\oc_{\mathbb{D}^{\!n}}$-module induced by the following map
\begin{eqnarray*}
\bigwedge^{k}H^0(\mathbb{D}^{\!n},\oc_{\mathbb{D}^{\!n}})&\rightarrow&  H^0(\mathbb{D}^{\!n},E_{k,k'}\Omega_{\mathbb{D}^{\!n}})\\
f_1\wedge\cdots\wedge f_k&\mapsto&
\begin{vmatrix}
d^{1} f_1  & \cdots & d^{1} f_k   \\
\vdots & \ddots & \vdots \\
d^{k} f_1  & \cdots & d^{k} f_k  
\end{vmatrix}.
\end{eqnarray*}
Fix an open covering $\mathfrak{U}$ of $X$ such that for any open set $U\in \mathfrak{U}$, $L_{\upharpoonright U}$ can be trivialized and such that  one  has the following dichotomy:
\begin{enumerate}[leftmargin=0.7cm]
	\item \label{intersect}$(U,D\cap U)$ is biholomorphic to $(\mathbb{D}^{\!n},E)$, and under the trivialization of $L_{\upharpoonright U}$, $\sigma_D=z_1$.
	\item \label{not intersect}$U\cap D=\emptyset$, and  under the trivialization of $L_{\upharpoonright U}$, $\sigma_D=1$.
\end{enumerate}  
It follows from the very definition that one has the following local trivialization of the morphism $J^kW_D$ in Case \eqref{intersect},
\[\xymatrix{
\bigwedge^{k}J^k L_{\upharpoonright U}\ar[d]_{ {\rotatebox{90}{$\simeq$}}}  \ar[r]^-{J^kW_D}  & 	E_{k,k'}\Omega_X(\log D)\otimes L_{\upharpoonright U}^{k+1}\ar[d]^-{\rotatebox{270}{$\simeq$}}\\
\bigwedge^{k}J^k\oc_{\mathbb{D}^{\!n}} \ar[r]^-{J^kW_E} & E_{k,k'}\Omega_{\mathbb{D}^{\!n}}(\log E); 
}\]
and in Case \eqref{not intersect}
\[\xymatrix{
	\bigwedge^{k}J^k L_{\upharpoonright U}\ar[d]_{\rotatebox{90}{$\simeq$}}  \ar[r]^-{J^kW_D}  & 	E_{k,k'}\Omega_X(\log D)\otimes L_{\upharpoonright U}^{k+1}\ar[d]^-{\rotatebox{270}{$\simeq$}}\\
	\bigwedge^{k}J^k\oc_{\mathbb{D}^{\!n}} \ar[r]^-{J^kW} & E_{k,k'}\Omega_{\mathbb{D}^{\!n}}.
}\]
 Therefore, we conclude that the local models of the logarithmic Wronskians are universal.

Let us denote by $\mathbb{D}^{\!n}_{k}$ (resp. ${\mathbb{D}^{\!n}_k}(E)$) the   (resp. logarithmic) Demailly-Semple $k$-jet tower of $(\mathbb{D}^{\!n},T_{\mathbb{D}^{\!n}})$ (resp. $\big(\mathbb{D}^{\!n},E,T_{\mathbb{D}^{\!n}}(-\log E) \big)$). By \eqref{eq:DirectImageFormula},
there are natural morphisms of $\oc_{{\mathbb{D}^{\!n}_k}(E)}$-module and $\oc_{ {\mathbb{D}^{\!n}_k}}$-module \[{\pi}_{0,k}^{-1}E_{k,k'}\Omega_{\mathbb{D}^{\!n}}(\log E)\to \oc_{{\mathbb{D}^{\!n}_k}(E)}(k'), \ \ \  {\pi}_{0,k}^{-1}E_{k,k'}\Omega_{{\mathbb{D}^{\!n}}}\to \oc_{ {\mathbb{D}^{\!n}_k}}(k').\]
Set  $\mathscr{W}_{k,E}$ and $\mathscr{W}_{k}$
to be the images of   $J^kW_E$ and $J^kW$.   Therefore,  
\begin{eqnarray*}
 {\pi}_{0,k}^{-1}\mathscr{W}_{k,E}\otimes \oc_{{\mathbb{D}^{\!n}_k}(E)}(-k')&\hookrightarrow&{\pi}_{0,k}^{-1}E_{k,k'}\Omega_{\mathbb{D}^{\!n}}(\log E)\otimes \oc_{{\mathbb{D}^{\!n}_k}(E)}(-k')\to \oc_{{\mathbb{D}^{\!n}_k}(E)},\\
  {\pi}_{0,k}^{-1}\mathscr{W}_{k}\otimes \oc_{\mathbb{D}^{\!n}_k}(-k')&\hookrightarrow& {\pi}_{0,k}^{-1}E_{k,k'}\Omega_{\mathbb{D}^{\!n}}\otimes \oc_{{\mathbb{D}_k^n}}(-k')\to \oc_{\mathbb{D}_k^n},
\end{eqnarray*}
whose images are coherent ideal sheaves, which we denote by ${\wk}_{\mathbb{D}^{\!n}_{k}(E)}\subset \oc_{{\mathbb{D}^{\!n}_k}(E)}$ and ${\wk}_{\mathbb{D}^{\!n}_{k}}\subset \oc_{\mathbb{D}^{\!n}_k}$. Then for any $U\in \mathfrak{U}$, under the trivialization  ${X}_k(D)_{\upharpoonright U}\simeq {\mathbb{D}^{\!n}_k}(E)$ in Case \eqref{intersect} and ${X}_k(D)_{\upharpoonright U}\simeq {\mathbb{D}^{\!n}_k}$ in Case \eqref{not intersect}, one has the isomorphisms ${\mathfrak{w}}_{{X}_k(D)\upharpoonright U}\simeq {\wk}_{\mathbb{D}^{\!n}_{k}(E)}$ and ${\mathfrak{w}}_{{X}_k(D)\upharpoonright U}\simeq {\wk}_{\mathbb{D}^{\!n}_{k}}$  respectively. This local description will be  used to establish a certain universal property in \Cref{sec:Universal family}.

\subsection{Universal property of logarithmic Wronskian ideal sheaves}\label{sec:universal}
Let us begin with the following setting. Let $A$  be a very ample line bundle over a smooth projective manifold $Y$, and let $\lbb$ be  the total space of the line bundle $A^m$ for some $m\in \mathbb{N}^*$. Denote by $p:\lbb\rightarrow Y$ the natural projection map with $L:=p^*A^m$, and $T\in H^0(\lbb,L)$ the tautological section such that $T(x)=x$ for any $x\in \lbb$. Note that $Y$ can be seen as the  smooth hypersurface of $\lbb$ defined by $ \{x\in \lbb\mid T(x)=0\}$. Then according to \Cref{sec:connection}, there exists for any \(k\in \nb\) a natural   higher order logarithmic connection $\nabla^k:\oc_{\lbb}(L)\rightarrow \oc_{\lbb}(L)\otimes E_{k,k}^{\rm GG}\Omega_{\lbb}(\log Y)$ associated to the log manifold $({\lbb},Y)$. For any sections $s_1,\ldots,s_k\in H^0(Y,A^m)$, it follows from  \eqref{eq:defLogWronskian}  that one has the associated logarithmic Wronskian
$$W_{{\lbb},Y} (p^*s_1,\ldots,p^*s_k)\in H^0\big({\lbb},E_{k,k'}\Omega_{\lbb}(\log Y)\otimes L^k).$$  By \eqref{eq:DirectImageFormula}, there exists a unique section in $H^0\big(\lbb_{k} ,\oc_{\lbb_{k} }(k')\otimes {\pi}_{0,k}^*L^{k}\big)$, denoted by $\omega_{\log} (s_1,\ldots,s_k)$, such that
\[({\pi}_{0,k})_*\omega_{\log} (s_1,\ldots,s_k)=W_{{\lbb},Y} (p^*s_1,\ldots,p^*s_k),\] 
where $\lbb_{k} $ denotes to be the log Demailly  $k$-jet tower of $\big({\lbb},Y,T_{\lbb}(-\log Y)\big)$.

Let us denote by \(\omega'_{\log}(s_1,\ldots,s_k)\) the logarithmic Wronskian defined in \Cref{prop:DSWronskian}. Then 
\begin{eqnarray}\label{eq:  two wronskian}
\omega_{\log}(s_1,\ldots,s_k)=\omega'_{\log}(s_1,\ldots,s_k)\cdot \bm{\Gamma}_k,
\end{eqnarray}
where \(\bm{\Gamma}_k\) is an effective divisor of \(\lbb_k\) defined in \eqref{eq:exceptional}.  Consider the linear systems
\begin{eqnarray*}
\mathbb{W}_{k,{\lbb},Y}:={\rm Span}\{\omega_{\log} (s_1,\ldots,s_k)\mid s_1,\ldots,s_k\in H^0(Y,A^m) \}\\
\mathbb{W}'_{k,{\lbb},Y}:={\rm Span}\{\omega'_{\log} (s_1,\ldots,s_k)\mid s_1,\ldots,s_k\in H^0(Y,A^m) \}
\end{eqnarray*}
and define  ${\wk}_{k,{\lbb},Y}$ and ${\wk}'_{k,{\lbb},Y}$ to be their base ideal. By \eqref{eq:  two wronskian}, one has
\begin{eqnarray}\label{eq:ideal relation}
{\wk}_{k,{\lbb},Y}={\wk}'_{k,{\lbb},Y}\cdot\oc_{\lbb_k}(-\bm{\Gamma}_k).
\end{eqnarray}
It follows from the definition of $W_{{\lbb},Y}$ that, there exists a morphism of $\oc_{\lbb}$-module
\begin{align}\label{eq:image}
j^kW_{{\lbb},Y}:\bigwedge^kp^*(J^kA^m)\rightarrow E_{k.k'}\Omega_{\lbb}(\log Y)\otimes L^k
\end{align}
such that $W_{{\lbb},Y}$ factors through this morphism. Set $\mathscr{W}_{k,{\lbb},Y}$ to be the image of $j^kW_{{\lbb},Y}$. We will study the properties of $j^kW_{{\lbb},Y}$ locally.

Take an open set $U$ with coordinates $(z_1,\ldots,z_n)$ such that $A_{\upharpoonright U}$ can be trivialized. Then there are local coordinates $(t,z_1,\ldots,z_n)$ for $p^{-1}(U)\simeq U\times \cb$, such that $L_{\upharpoonright p^{-1}(U)}$ is trivialized with $T_{\upharpoonright p^{-1}(U)}=t$. Hence  the divisor $Y\cap p^{-1}(U)$  is  defined by the local equation $(t=0)$.  One thus can regard $U$  as a smooth divisor in $p^{-1}(U)$ defined by $(t=0)$.  For any \(k\in \nb\),  write $\nabla_U^{k}:\oc_{U\times \cb}\rightarrow E^{\rm GG}_{k,k}\Omega_{U\times \cb}(\log U)$ for the higher order logarithmic connection defined in \eqref{def:higherorder}. 
In view of \eqref{eq:defLogWronskian} we define 
\begin{eqnarray}\label{eq: local wronskian}
W_{p^{-1}(U),U}:	\bigwedge^k \oc(U)&\rightarrow& \Gamma\big(U\times \cb,E_{k,k'}\Omega_{U\times \cb}  (\log U)\big)\\
	f_1\wedge\cdots\wedge f_k&\mapsto & \begin{vmatrix}
		\nabla^{1}_{U}(p_U^*f_1) & \cdots & \nabla^{1}_U(p_U^*f_k)  \\\nonumber
		\vdots & \ddots & \vdots \\
		\nabla^{k}_U(p_U^*f_1) & \cdots & \nabla^{k}_U(p_U^*f_k) 
	\end{vmatrix}
\end{eqnarray}
where $p_U:U\times \cb\rightarrow U$ is the natural projection map. By \cref{lem: non-log}, one has
\begin{eqnarray}\label{eq: constant}
W_{p^{-1}(U),U}(	f_1\wedge\cdots\wedge f_k)=  \begin{vmatrix}
		1& d^{0}(p_U^*f_1) & \cdots & d^{0}(p_U^*f_k)  \\
		\frac{1}{t}d^{1}t& d^{1}(p_U^*f_1) & \cdots & d^{1}(p_U^*f_k)  \\
		\vdots& \vdots & \ddots & \vdots \\
		\frac{1}{t}d^{k}t&d^{k}(p_U^*f_1) & \cdots & d^{k}(p_U^*f_k) 
	\end{vmatrix}.
\end{eqnarray}
Observe that $d^{i}(p_U^*f_j)$ does not contain any $d^{\ell}\log t$, therefore \eqref{eq: local wronskian}  induces a morphism between locally free sheaves of $\oc_{p^{-1}(U)}$-modules $$j^kW_{p^{-1}(U),U}:\bigwedge^kp^*(J^k\oc_U)\rightarrow E_{k,k'}\Omega_{U\times \cb}(\log U)$$
so that $W_{p^{-1}(U),U}$ factors through this morphism. If we use the  basis for the local trivialization of $E_{k,k'}^{GG}\Omega_{U\times \cb}(\log U)$ in \Cref{lem:new basis} and the standard basis for the trivialization of $J^k\oc_U$ induced by the coordinates system $(z_1,\ldots,z_n)$, then by \eqref{eq: constant}, $j^kW_{p^{-1}(U),U}$ is represented by a \emph{constant matrix} with respect to these trivializations. In particular, the image of $j^kW_U$, denoted by $\mathscr{W}_{k, U}$, is a  locally free sheaf. In this setting, $j^kW_{p^{-1}(U),U}$ trivializes $j^kW_{{\lbb},Y}$.
Hence $\mathscr{W}_{k,{\lbb},Y}\subset E_{k,k'}\Omega_{\lbb}(\log Y)\otimes L^k$ is a locally free sheaf of $\oc_{\lbb}$-module on ${\lbb}$. As in \Cref{prop:coincide}, one has the following 
\begin{proposition}\label{prop:ideal sheaf}
For $L:=p^*A^m$, when $m\geqslant k$,	the ideal sheaf ${\wk}_{k,{\lbb},Y}$ coincides with the image
\begin{align*}
\varphi_k:	{\pi}_{0,k}^{-1}\mathscr{W}_{k,{\lbb},Y} \otimes  \oc_{ {\lbb}_{k}}(-k')\otimes \pi_{0,k}^*L^{-k}\hookrightarrow {\pi}_{0,k}^{-1}\ E_{k,k'}\Omega_{\lbb}(\log Y) \otimes \oc_{ {\lbb}_{k}}(-k')\rightarrow \oc_{ {\lbb}_{k}}, 
\end{align*}
	where $\oc_{ {\lbb}_{k}}(1)$ denotes to be the tautological line bundle defined in \cref{sec:LogDSTower}.
\end{proposition}
\begin{proof}
		For any $s_1,\ldots,s_k\in H^0(Y,A^m)$, one has the following natural linear map from the global sections to their $k$-jets
	\begin{eqnarray*}
		\bigwedge^kH^0(Y,A^m)&\rightarrow& H^0(Y,\bigwedge^kJ^kA^m)\\
		s_1\wedge\cdots\wedge s_k &\mapsto& j^ks_1\wedge \cdots \wedge j^ks_k.
	\end{eqnarray*}
Here we write $j^k$ instead of  $j^k_{A^m}$ to lighten the notation.
Recall that 
	\[
	W_{{\lbb},Y}(s_1,\ldots,s_k)=j^kW_{{\lbb},Y}\big(p^*j^ks_1\wedge \cdots \wedge p^*j^ks_k\big)\in H^0({\lbb},\mathscr{W}_{k,{\lbb},Y}),
	\]
	and thus by the definition of ${\wk}_{k,{\lbb},Y}$ and the fact $({\pi}_{0,k})_* \omega_{\log}(s_1,\ldots,s_k) =W_{{\lbb},Y}( s_1,\ldots, s_k)$, one has
	\[
{\wk}_{k,{\lbb},Y}\subseteq	\Im(\varphi_k).
	\]
	On the other hand, since $A$ is very ample and $m\geqslant k$, then $A^m$ generates $k$-jets, and the set \[{\rm Span}\{p^*j^ks_1\wedge \cdots \wedge p^*j^ks_k)\in H^0({\lbb},\bigwedge^k p^*J^kA^m)\mid s_1,\ldots,s_k\in H^0(Y,A^m) \}\] generates the  locally free  sheaf of $\oc_{{\lbb}}$-module $\bigwedge^kp^*J_kA^m$ everywhere on ${\lbb}$. It follows from the definition that $J^kW_{{\lbb},Y}: \bigwedge^kp^*J^kA^m\rightarrow \mathscr{W}_{k,{\lbb},Y}$ is a surjective morphism between sheaves of $\oc_{\lbb}$-modules. Therefore, the set of  sections 
	\[
{\rm Span}\{j^kW_{{\lbb},Y}\big(p^*j^ks_1\wedge \cdots \wedge p^*j^ks_k\big)\in H^0({\lbb},\mathscr{W}_{k,{\lbb},Y})\mid s_1,\ldots,s_k\in H^0(Y,A^m) \}
	\]
	generates  the sheaf of $\oc_{\lbb}$-module $\mathscr{W}_{k,{\lbb},Y}$, and one  thus has ${\wk}_{k,{\lbb},Y}\supseteq	\Im(\varphi_k)$. This implies the result.
\end{proof}

\subsection{Universal family of  log Demailly towers of general log pairs and its blow-up}\label{sec:Universal family}

As we did in \cite{BD17}, the construction of the log pair $(\lbb,Y)$ enables us to ``linearize" the family of log manifolds $(Y,D)$ with $D$ varying in the linear system $|A^m|$.   Indeed, for any $\sigma\in H^0(Y,A^m)$, consider the hypersurface $H_\sigma\subset {\lbb}$ defined to be the zero locus of the section
$$
T-p^*\sigma\in H^0({\lbb},p^*A^m).
$$
When the zero locus $D_\sigma$ of $\sigma$ is a smooth hypersurface on $Y$, $H_\sigma$ will also be smooth.  
A crucial observation is  that \[p_\sigma=p_{\upharpoonright H_\sigma}:(H_\sigma,D_\sigma)\rightarrow (Y,D_\sigma)\] is a biholomorphism between log manifolds, and the hyperbolicity of $Y\setminus D_\sigma$ is therefore equivalent to that of  $H_\sigma\setminus D_\sigma$. Moreover, we have  the functoriality of the logarithmic Wronskians  ideal sheaves.

\begin{lem}\label{isomorphism ideal}
We denote by $ {H}_{\sigma,k}$   the log Demailly $k$-jet  tower of   $\big(H_\sigma,D_\sigma,T_{H_{\sigma}}(-\log D_\sigma)\big)$, and let $ {\mathfrak{w}}_{ {H}_{\sigma,k}}$ be the $k$-th logarithmic Wronskian ideal sheaf of ${H}_{\sigma,k}$ defined in \Cref{sec:wronskian ideal}. When $m\geqslant k$,  we have
$$ {\wk}_{k,{\lbb},Y \upharpoonright  {H}_{\sigma,k}}= {\mathfrak{w}}_{ {H}_{\sigma,k}}.$$
\end{lem}
\begin{proof}
Recall that \[p_\sigma=p_{\upharpoonright H_\sigma}:(H_\sigma,D_\sigma)\rightarrow (Y,D_\sigma)\] is a biholomorphism between log manifolds. Write $L_\sigma:=p_\sigma^*A^m=L_{\upharpoonright H_\sigma}$. Then $p_\sigma$ induces an isomorphism of linear spaces of global sections
\begin{align*}
H^0(Y,A^m)&\xrightarrow{\simeq} H^0(H_\sigma, L_\sigma)\\
s & \mapsto p_\sigma^*s.
\end{align*}
By the functoriality of the logarithmic Wronskians in \eqref{equ:fonctorial}, for any  $s_1,\ldots,s_k\in H^0(Y,A^m)$, we have
$$
\omega_{\log}(s_1,\ldots,s_k)_{\upharpoonright {H}_{\sigma,k}}=\omega_{\!D_\sigma}(p_\sigma^*s_1,\ldots,p_\sigma^*s_{k })\in H^0\big({H}_{\sigma,k}, \oc_{{H}_{\sigma,k}}(k')\otimes \pi_{0,k}^*L_\sigma^k\big).
$$ 
Hence $ {\wk}_{k,{\lbb},Y \upharpoonright  {H}_{\sigma,k}}$ is the base ideal sheaf of the linear system 
\[
\mathbb{W}_{{H}_{\sigma,k};L_\sigma}:={\rm Span}\{\omega_{\!D_\sigma}(s_1',\ldots,s_k')\mid s'_1,\ldots,s'_k\in H^0(H_\sigma, L_\sigma)\}.
\]
We note that when $m\geqslant k$, $A^m$ generates $k$-jets everywhere on $Y$, therefore so does  $L_{\sigma}$. The lemma then follows immediately from \cref{prop:coincide}.
\end{proof}

Now let us consider the \emph{universal family} of hypersurfaces ${\hs}^{\rm u}\subset {\lbb}\times H^0(Y,A^m) $ defined by 
\[
{\hs}^{\rm u}:=\{(x,\sigma)\in {\lbb}\times H^0(Y,A^m)\mid  T(x)-p^*\sigma(x)=0\},
\]
and define the family of hypersurfaces in $Y$ by
\[
 {\ds}^{\rm u}:=\{(y,\sigma)\in  Y\times H^0(Y,A^m)\mid  \sigma(y)=0\}.
\]
One can take a non-empty Zariski open set $\ab^{\rm s}$ of the  parameter space $ {\ab}^{\rm u}:=H^0(Y,A^m)$ such that, the shrinking log family over $\ab^{\rm s}$, denote by of $({\hs}^{\rm s}, {\ds}^{\rm s})\to \ab^{\rm s}$,  is \emph{smooth}. 
Set ${\hs}_k^{\rm s}$ to be the  log Demailly $k$-jet  tower of  $\big(\hs^{\rm s},\ds^{\rm s},T_{\hs^{\rm s}/\ab^{\rm s}}(-\log  \ds^{\rm s})\big)$, and denote by $q_k:{\hs}_k^{\rm s}\rightarrow  \ab^{\rm s}$ the natural projection. By the choice of $\ab^{\rm s}$, one notes that for any $\sigma\in \ab^{\rm s}$ , the fiber  $q_k^{-1}(\sigma)$ is $ {H}_{\sigma,k}$. Observe that we have an embedding 
${\hs}_k^{\rm s}\hookrightarrow  {\lbb}_{k}\times \ab^{\rm s}$. Let us denote by
$$
 {\wk}_{{\hs}_k^{\rm s}}:={\rm pr}_1^*( {\wk}_{k,{\lbb},Y})_{\upharpoonright{\hs}_k^{\rm s}},
$$
where ${\rm pr}_1: {\lbb}_{k}\times \ab^{\rm u}\rightarrow  {\lbb}_{k}$ is the natural projection map. By \cref{isomorphism ideal}, we have
$$
{\wk}_{{\hs}_k^{\rm s}\upharpoonright H_{\sigma, k}}={\mathfrak{w}}_{ {H}_{\sigma,k}}.
$$
In some sense, the ideal ${\wk}_{k,{\lbb},Y}$ is the obstruction to the positivity of $\oc_{ {\lbb}_{k}}(1)$. Therefore,   let us define  $\nu_k:\widetilde{{\lbb}}_k \rightarrow  {\lbb}_{k}$ to be the blow-up of the ideal sheaf $ {\wk}_{k,{\lbb},Y}$. It follows from \eqref{eq:ideal relation}  that $\nu_k$ is also the blow-up for \( {\wk}'_{k,{\lbb},Y}\) (see \cite[Chapter \rom{2}, Exercise 7.11]{Har77}). We denote by $F$ and $F'$  the effective divisors in \(\widetilde{\lbb}_k\) such that
\begin{eqnarray}\label{eq:two exceptional}
\nu_k^* {\wk}_{k,{\lbb},Y}=\oc_{\widetilde{{\lbb}}_k}(-{F}) \quad\mbox{and}\quad\nu_k^* {\wk}'_{k,{\lbb},Y}=\oc_{\widetilde{{\lbb}}_k}(-{F'}).
\end{eqnarray}
We  define
$\mu_k:\widetilde{\hs}_k^{\rm s} \rightarrow {\hs}_k^{\rm s}$  to be the blow-up of the ideal sheaf ${\wk}_{{\hs}_k^{\rm s}}$ with  $\oc_{\widetilde{\hs}_k^{\rm s}}(-\tilde{F}):=\mu_k^* {\wk}_{{\hs}_k^{\rm s}}$. 
By the universal property of the blow-up, one has the commutative diagram
\[
\xymatrix{\widetilde{\hs}_k^{\rm s} \ar[d]_{\mu_k}\ar@{^{(}->}[r]&  \widetilde{{\lbb}}_k\times \ab^{\rm s}\ar[d]^{\nu_k\times \mathds{1}} \\
{\hs}_k^{\rm s}	\ar@{^{(}->}[r]&  {\lbb}_{k}\times \ab^{\rm s}.
}
\]
The following lemma enables us to  reduce the desired ``general Kobayashi hyperbolicity" to  a construction of a particular example satisfying a strong Zariski open property.
\begin{lem}\label{lem:fiberwise blow-up}
When $\mu_k$ is restricted on each fiber $\widetilde{H}_{\sigma,k}$ of $\widetilde{\hs}_k^{\rm s}\rightarrow \ab^{\rm s}$, $\mu_{\sigma,k}=\mu_{k \upharpoonright \widetilde{H}_{\sigma,k}}:\widetilde{H}_{\sigma,k}\rightarrow  {H}_{\sigma,k}$ is nothing but the blow-up of the ideal sheaf $ {\mathfrak{w}}_{ {H}_{\sigma,k}}$.
\end{lem}
\begin{proof}
	Let us first  observe that as a consequence of the local inverse theorem in several complex variables we obtain that in the analytic category, families of smooth pairs $({\hs}^{\rm s}, {\ds}^{\rm s})\stackrel{\rho}{\to}\ab^{\rm s}$ are \emph{locally trivial} in the following sense:  for any \(x\in {\hs}^{\rm s}\) there exists a neighborhood \(\Omega\subset {\hs}^{\rm s}\),  a neighborhood \(V\subset \ab^{\rm s}\) of \(\sigma:=\rho(x)\) and an open subset \(U\subset \cb^n\) with coordinates  \((z_1,\dots,z_n)\) such that there exists an isomorphism 
		\[\Phi:U \times V\stackrel{\simeq}{\to} \Omega \]
		satisfying \(\rho\circ \Phi=\pr_2\) (where \(\pr_2:U\times V\to V\) is the projection on the first factor) and 
		such that  
		\[\Phi^*\ds^{\rm s}=(z_1 =0)  \quad {\rm or} \quad \ds^{\rm s}\cap \Omega=\varnothing.   \]  
By the local description of the logarithmic Wronskian ideal sheaves established in \cref{sec:wronskian ideal}, via the isomorphism $\Phi$ 
 one has  
\[
{\wk}_{{\hs}_k^{\rm s} \upharpoonright   \pi_{0,k}^{-1}(\Omega)}\simeq {\rm pr}_1^* {\wk}_{\mathbb{D}^{\!n}_{k}(E)}  \quad {\rm or} \quad {\wk}_{{\hs}_k^{\rm s} \upharpoonright   \pi_{0,k}^{-1}(\Omega)}\simeq {\rm pr}_1^* {\wk}_{\mathbb{D}^{\!n}_{k}}.
\]
This implies the result of the lemma.
\end{proof}
On the other hand, for any sections $s_1,\ldots,s_k\in H^0(Y,A^m)$, by   \eqref{eq:ideal relation} and \eqref{eq:two exceptional}  there exists  a (unique)  section
\begin{align*} 
\tilde{\omega}_{\log}(s_1,\ldots,s_k)&\in   H^0\Big(\widetilde{{\lbb}}_k,\nu_k^*\big(\oc_{ {\lbb}_{k}}(k')\otimes {\pi}_{0,k}^*L^k\big)\otimes \oc_{\widetilde{{\lbb}}_k}(-{F})\Big), \\\nonumber
&= H^0\Big(\widetilde{{\lbb}}_k,\nu_k^*\big(\oc_{ {\lbb}_{k}}(k,k-1,\ldots,1)\otimes {\pi}_{0,k}^*L^k\big)\otimes \oc_{\widetilde{{\lbb}}_k}(-{F'})\Big) 
\end{align*}such that
\begin{align}\label{eq:pull back relation1}
\nu_k^*\omega_{\log}(s_1,\ldots,s_k)&=\tilde{\omega}_{\log}(s_1,\ldots,s_k)\cdot F,\\\label{eq:pull back relation2}
\nu_k^*\omega'_{\log}(s_1,\ldots,s_k)&=\tilde{\omega}_{\log}(s_1,\ldots,s_k)\cdot F'.
\end{align}
We will also need the following crucial lemma.
\begin{lem}\label{k-jets}
	For any sections $s_1,\ldots,s_k,s_1',\ldots,s_k'\in H^0(Y,A^m)$ and any point $y\in Y$, if the $k$-jets $j^ks_i(y)=j^ks_i'(y)\in (J^kA^m)_y$ for each $i=1,\ldots,k$, then on the fiber $\widetilde{{\lbb}}_{k,y}:=( p\circ {\pi}_{0,k}\circ\nu_k)^{-1}(y)$ of $p\circ {\pi}_{0,k}\circ\nu_k:\widetilde{{\lbb}}_k\rightarrow Y$, one has
$$ \widetilde{\omega}_{\log}(s_1,\ldots,s_k)_{\upharpoonright \widetilde{{\lbb}}_{k,y}}= \widetilde{\omega}_{\log}(s'_1,\ldots,s'_k)_{\upharpoonright \widetilde{{\lbb}}_{k,y}}.$$
\end{lem}
\begin{proof}

Define $ {{\lbb}}_{k,x}$ to be the fiber $( {\pi}_{0,k} )^{-1}(x)$ of $  {\pi}_{0,k} : {\lbb}_{k}\rightarrow {\lbb}$. Note that the natural morphism 
$$E_{k,k'}\Omega_{\lbb}(\log Y)\otimes k(x)\xrightarrow{\simeq} ({\pi}_{0,k})_*\oc_{ {\lbb}_{k}}(k')\otimes k(x)\rightarrow H^0\big( {{\lbb}}_{k,x},\oc_{ {\lbb}_{k}}(k')_{\upharpoonright  {{\lbb}}_{k,x}}\big)$$ is an isomorphism, where $k(x)$ is the residue field of ${\lbb}$ at $x$. By the assumption that $j^ks_i(y)=j^ks_i'(y)$ for each $i=1,\ldots,k$, one has 
$$
j^ks_1\wedge\cdots\wedge j^ks_k(y)=j^ks'_1\wedge\cdots\wedge j^ks'_k(y)\in \bigwedge^k J^kA^m\otimes k(y).
$$
Hence for any $x\in p^{-1}(y)$, one has
$$
j^kW_{k,{\lbb},Y}\big(p^* j^ks_1\wedge\cdots\wedge p^*j^ks_k \big)(x)=j^kW_{k,{\lbb},Y}\big(p^* j^ks'_1\wedge\cdots\wedge p^*j^ks'_k \big)(x)\in E_{k,k'}\Omega_{\lbb}(\log Y)\otimes L^k\otimes k(x),
$$
and 
we conclude that $\omega_{\log}(s_1,\ldots,s_k)_{\upharpoonright  {{\lbb}}_{k,y}}=\omega_{\log}(s'_1,\ldots,s'_k)_{\upharpoonright  {{\lbb}}_{k,y}}$.

It now suffices to observe that the co-support of the ideal sheaf \( {\mathfrak{w}}_{k,\lbb,Y}\) does not contain the fiber \( {\lbb}_{k,x}\). Indeed, the announced statement will follow at once by continuity. To see this, it suffices to take coordinates \((z_1,\dots, z_n)\) centered at \(y=p(x)\) and consider the functions \(z_1,\dots, z_1^k\) in  a neighborhood of \(y\). A direct computation then shows that \(\omega_{\log}(z_1,\dots, z_1^k)_{\upharpoonright  {\lbb}_{k,x}}\) is not identically zero, which implies the announced result.
\end{proof}

\section{Main constructions}\label{sec:main construction}
\subsection{Fermat type hypersurfaces and associated pairs} \label{construction}
 To begin with, we construct a family of hypersurfaces in $Y$ parametrized by certain Fermat type   as in \cite{Bro17}. 
 Let \(A\) be a very ample line bundle on \(Y\).
 For an integer $N\geqslant n={\rm dim}(Y)$, we fix $N+1$ sections in general position $\tau_0,\ldots,\tau_N\in H^0(Y,A)$. By ``general position" we mean that the divisors defined by $(\tau_j=0)_{j=0,\ldots,N} $ are all smooth and meet transversally. For any two   positive integers $\varepsilon,\delta$,  set \[\ib:=\{I=(i_0,\ldots,i_N)\mid|I|=\delta \}\] and 
 \[\af:=\big(a_{I}\in H^0(Y,A^{\varepsilon })\big)_{|I|=\delta}\in \ab:=\bigoplus_{I\in\ib}H^0(Y,A^{\varepsilon}).\] For two positive integers $r$ and $k$ fixed later according to our needs, consider the family ${\mathscr{D}}\rightarrow \ab$ of hypersurfaces in $\big\lvert\big(\varepsilon+(r+k)\delta\big)A\big\rvert$ defined by the zero locus of the bihomogenous sections
 \begin{eqnarray}\nonumber
 \sigma(\af)(y):y\mapsto\sum_{|I|=\delta}a_{I}(y)\tau(y)^{(r+k)I},
 \end{eqnarray}
 where $(a_{I})_{|I|=\delta}$ varies in the parameter space $\ab$, and $\tau:=(\tau_0,\ldots,\tau_N)$. For any \(\af\in \ab\), let us write $D_{\af}$ for the fiber of the family ${\mathscr{D}}\rightarrow \ab$.    
 
 Write $m:=\varepsilon+(r+k)\delta$. Consider the total space \(p:\lbb\to Y\) of \(A^m\) defined in \cref{sec:universal}, and write $L:=p^*A^m$. With the same notation in \emph{loc. cit.},  consider the family of hypersurfaces $ {\hs}\rightarrow \ab$ in $\lbb$ defined by the vanishing of  the section
 $$
 T-p^*\sigma(\af)\in  H^0(\lbb,L).
 $$
For any \(\af\in \ab\),  write \(H_{\af}:=( T-p^*\sigma(\af)=0)\subset \lbb\). By \cite{BD15} there exists a non-empty Zariski open subset \({\ab}_{\rm sm}\subset \ab\) such that \(D_{\af}\) is a smooth hypersurface for any \(\af\in {\ab}_{\rm sm}\), and so is \(H_{\af}\).   Let us now shrink the family \(\hs\) (resp. $\ds$) to \({\ab}_{\rm sm}\), and let us denote abusively \(\hs\to {\ab}_{\rm sm}\) (resp. \(\ds\to {\ab}_{\rm sm}\)) this restricted family.  Since we can see \(\ds\) as a hypersurface in \(\hs\) defined by the equation \((T=0)\). Then, by the choice of  $\ab_{\rm sm}$,   \(\ds\) is a smooth hypersurface of \(\hs\) and moreover, 
$(\hs,\ds)\to \ab_{\rm sm}$ is a smooth family of log pairs. 

 Let us define ${\hs}_{k}$ to be the   log  Demailly  $k$-jet tower of $(\hs,\ds,T_{\hs/{\ab}_{\rm sm}}(-\log \ds))$.  Under the natural inclusive morphism
 \begin{align*}
{\ab}_{\rm sm}&\hookrightarrow  {\ab}^{\rm s}\\
\af&\mapsto  \sigma(\af)
\end{align*} 
to the universal family defined in \Cref{sec:Universal family}, and by the flat base change theorem, one simply obtains \({\hs}_{k}:= {\hs}_k^{\rm s}\times_{\ab^{\rm s}}\ab_{\rm sm} \). Let us define 
  \begin{align*}
 \widetilde{\hs}_{k}:= \widetilde{\hs}^{\rm s}_k\times_{\ab^{\rm s}}\ab_{\rm sm}.
 \end{align*}
For any \(\af\in \ab_{\rm sm}\), the fibers of \({\hs}_{k}\rightarrow \ab_{\rm sm}\) and \(\widetilde{\hs}_{k}\rightarrow \ab_{\rm sm}\) are denoted by \({H}_{\af,k}\) and \(\widetilde{H}_{\af,k}\) respectively. Observe that in view of \cref{lem:fiberwise blow-up}, this notation is consistent in the sense that \(\widetilde{H}_{\af,k}\) is indeed the blow-up of \(H_{\af,k}\) along the logarithmic Wronskian ideal sheaf $\wk_{H_{\af,k}}$.

\subsection{Mapping to the Grassmannians} 
Consider the log pair $(\lbb,Y)$ defined in \cref{construction} equipped with the line bundle $L:=p^*A^m$. By \eqref{def:higherorder}, one can define the higher order logarithmic connection $\nabla^j:L\to E^{\rm GG}_{j,j}\Omega_{\lbb}(\log Y)\otimes L$. As in \cref{sec:intrinsic}, let us take a  trivialization tower  \(\mathfrak{U}=\big((U,T_U),(U_j,\xi_j)_{1\leqslant j\leqslant k}\big)\) of order \(k\). A straightforward induction implies the following 
\begin{lem}\label{dI}
	For any $I\in \ib$ and for any $1\leqslant j\leqslant k$, there exist  $\cb$-linear maps 
\begin{align*}
\nabla^{j}_{\!\!I}&:H^0(Y,A^{\varepsilon} )\rightarrow  H^0\big(\lbb,E_{j,j}^{\rm GG}\Omega_{\lbb}(\log Y)\otimes p^*A^{\varepsilon+k\delta}\big),\\
{\nabla}_{\!\!\uk,I}^{j}&:H^0(Y,A^{\varepsilon} )\rightarrow \oc({U_k}), 
\end{align*}
	such that for any $a_I\in H^0(Y,A^{\varepsilon})$, one has
	\begin{align*}
	\nabla^{j}(p^*a_I\cdot{(p^*\tau)}^{(r+k)I})&= (p^*\tau)^{rI}\cdot  \nabla^{j}_{\!\!I }(a_I),\\
	\onu^j(p^*a_I\cdot{(p^*\tau)}^{(r+k)I})&= (\tau_{U_k})^{rI} \cdot {\nabla}_{\!\!\uk,I}^{j}(a_I).
	\end{align*}
	Here we denote by $\tau_{U_k}$ the pull-back of trivialization of $p^*\tau$ under $L_{\upharpoonright U}$ to $U_k$.
\end{lem}
Therefore, for any $I_1,\ldots,I_k\in \ib$ and any $a_{I_1},\ldots,a_{I_k}\in H^0(Y,L^{\varepsilon})$ one can define
\begin{align*} 
	W_{\log,I_1,\ldots,I_k}( a_{I_1},\ldots, a_{I_k}):=\begin{vmatrix}
		\nabla^{1}_{\!\!I_1 }(a_{I_1}) & \cdots & 	\nabla^{1}_{\!\!I_k }(a_{I_k}) \\
		\vdots & \ddots & \vdots \\
		\nabla^{k}_{\!\!I_1 }(a_{I_1}) & \cdots & 	\nabla^{k}_{\!\!I_k }(a_{I_k}) 
	\end{vmatrix} \in H^0({\lbb}, E_{k,k'}  \Omega_{\lbb}(\log Y)\otimes p^*A^{k(\varepsilon+k\delta ) }).
\end{align*}
It then follows from \Cref{dI} that 
\begin{align}\label{eq:modified Wronskian}
W_{{\lbb},Y}(p^*a_{I_1}\cdot{(p^*\tau)}^{(r+k)I_1},\ldots,p^*a_{I_k}\cdot{(p^*\tau)}^{(r+k){I_k}})=(p^*\tau)^{r(I_1+\cdots+I_k)}\cdot W_{\log,I_1,\ldots,I_k}( a_{I_1},\ldots, a_{I_k}).
\end{align}
Set $$\omega_{\log,I_1,\ldots,I_k}( a_{I_1},\ldots, a_{I_k})\in H^0\big({{\lbb}}_k,\oc_{{{\lbb}}_k}(k')\otimes (p\circ{\pi}_{0,k})^*A^{k(\varepsilon+k\delta)}\big),$$ 
to be the inverse image of $({\pi}_{0,k})_*$ under the isomorphism \eqref{eq:DirectImageFormula}, then
\begin{align}\label{eq:two wronskian}
\omega_{\log}(a_{I_1}\tau^{(r+k)I_1},\dots,a_{I_k}\tau^{(r+k)I_k})=\omega_{\log,I_1,\ldots,I_k}( a_{I_1},\ldots, a_{I_k})\cdot (p\circ{\pi}_{0,k})^*\tau^{r(I_1+\cdots+I_k)}.
\end{align}
Moreover from  \Cref{prop:ideal sheaf} one can deduce at once
\begin{lem}\label{lem:ideal contain}
The section \(\omega_{\log,I_1,\ldots,I_k}( a_{I_1},\ldots, a_{I_k})\) vanishes along \({\wk}_{k,{\lbb},Y}\). In other words, \[\omega_{\log,I_1,\ldots,I_k}( a_{I_1},\ldots, a_{I_k})\in H^0\big({{\lbb}}_k,\oc_{{{\lbb}}_k}(k')\otimes (p\circ{\pi}_{0,k})^*A^{k(\varepsilon+k\delta)}\otimes {\wk}_{k,{\lbb},Y}\big),\] where ${\wk}_{k,{\lbb},Y}$ is the ideal sheaf defined in \Cref{sec:universal}.
\end{lem} 
\begin{proof}
By \eqref{eq:modified Wronskian}, one has 
$$
	W_{\log,I_1,\ldots,I_k}( a_{I_1},\ldots, a_{I_k})\in H^0(\lbb, \mathscr{W}_{k,{\lbb},Y}\otimes p^*A^{-rk\delta}),
$$
where we recall that $\mathscr{W}_{k,{\lbb},Y}$ is the image of the morphism $j^kW_{{\lbb},Y}$ defined in \eqref{eq:image}. Then the base ideal of $\omega_{\log,I_1,\ldots,I_k}( a_{I_1},\ldots, a_{I_k})$ belongs to the ideal sheaf of $\oc_{ {\lbb}_{k}}$ defined by the image of the morphism
	$$
{\pi}_{0,k}^{-1}(\mathscr{W}_{k,{\lbb},Y}\otimes p^*A^{-rk\delta})\otimes (p\circ{\pi}_{0,k})^* A^{-k(\varepsilon+k\delta)} \otimes  \oc_{ {\lbb}_{k}}(-k')\otimes \hookrightarrow {\pi}_{0,k}^*\ E_{k,k'}\Omega_{\lbb}(\log Y) \otimes \oc_{ {\lbb}_{k}}(-k')\rightarrow \oc_{ {\lbb}_{k}}.
$$
Note that the image of the above morphism coincides with that of the following one 
	$$
{\pi}_{0,k}^*\mathscr{W}_{k,{\lbb},Y}\otimes {\pi}_{0,k}^* L^{-k} \otimes  \oc_{ {\lbb}_{k}}(-k')\otimes \hookrightarrow {\pi}_{0,k}^*\ E_{k,k'}\Omega_{\lbb}(\log Y) \otimes \oc_{ {\lbb}_{k}}(-k')\rightarrow \oc_{ {\lbb}_{k}}.
$$
The lemma  follows immediately from  \Cref{prop:ideal sheaf}.
\end{proof}
 Recall that we  define  $\nu_k:\widetilde{{\lbb}}_k \rightarrow  {\lbb}_{k}$ to be the blow-up of the ideal sheaf $ {\wk}_{k,{\lbb},Y}$. By \cref{lem:ideal contain} and \eqref{eq:pull back relation1} there  exists a unique
\begin{align*}
	\widetilde{\omega}_{\log,I_1,\ldots,I_k}( a_{I_1},\ldots, a_{I_k})\in H^0\Big(\widetilde{{\lbb}}_k,\nu_k^*\big(\oc_{{{\lbb}}_k}(k')\otimes (p\circ{\pi}_{0,k})^*A^{k(\varepsilon+k\delta)}\big)\otimes \oc_{\widetilde{{\lbb}}_k}(-F)\Big)
\end{align*} 
such that
\begin{align*}
	\nu_k^*\omega_{\log,I_1,\ldots,I_k}( a_{I_1},\ldots, a_{I_k})=F\cdot\widetilde{\omega}_{\log,I_1,\ldots,I_k}( a_{I_1},\ldots, a_{I_k}).
\end{align*}
By definition $\omega_{\log,I_1,\ldots,I_k}( a_{I_1},\ldots, a_{I_k})$ is \emph{alternating} with respect to $(I_1,\ldots,I_k)$. We then can  define a rational map
\begin{align*}
	 \Phi:  \ab\times {{\lbb}}_k    &\dashrightarrow  \pt\big(\Lambda^{k}(\cb^{\mathbb{I}})\big)\\
	(\af,w) &\mapsto  \big[\big(\omega_{\log,I_1,\ldots,I_k}(a_{I_1},\ldots,a_{I_k})(w)\big)_{I_1,\ldots,I_k\in \ib}\big].
\end{align*}
The map $\Phi$ can also  be interpreted explicitly using our intrinsic construction in \Cref{sec:intrinsic}. 
Let us fix a tower trivialization \(\mathfrak{U}\) of order \(k\).  If we  denote by
\[
{\nabla}^i_{\!\!\uk,\bullet}(\af,w):=\big(	{\nabla}_{\!\!\uk,I}^{j}(a_I)(w)\big)_{I\in \ib}\in \bigoplus_{I\in\ib}\oc(\ab\times U_k),
\] 
for any  $1\leqslant i\leqslant k$, then we can define another rational map locally by
\begin{align*}
	\Phi_{\!\uk}: \ab  \times U_k &\dashrightarrow   {\rm Gr}_k(\cb^{\mathbb{I}}) \\
	(\af,w)&\mapsto  {\rm Span}\big({\nabla}^1_{\!\!\uk,\bullet}(\af,w),\cdots,{\nabla}^k_{\!\!\uk,\bullet}(\af,w)\big)_{I_1,\ldots,I_k\in \ib}
\end{align*}
and this is indeed the localization of $\Phi$.
\begin{lem}\label{lem:two rational}
	One has $\Phi_{\upharpoonright \ab\times U_k}={\rm Pluc}\circ \Phi_{\!\uk}$, where  ${\rm Pluc}:{\rm Gr}_{k}(\cb^{\ib})\hookrightarrow \textnormal{P}\big(\Lambda^{k}(\cb^{\mathbb{I}})\big)$ denotes the Pl\"ucker embedding.
	\end{lem}
\begin{proof}
Let us define
	\[\omega_{\mathfrak{U},I_1,\ldots,I_k}(a_{I_1},\ldots, a_{I_k}):=
	\left|\begin{array}{ccc}
	{\nabla}_{\!\!\uk,I_1}^{1}(a_{I_1})&\cdots & {\nabla}_{\!\!\uk,I_1}^{1}(a_{I_k})\\
	\vdots & \ddots& \vdots \\
	{\nabla}_{\!\!\uk,I_1}^{k}(a_{I_1})&\cdots & {\nabla}_{\!\!\uk,I_k}^{k}(a_{I_k})
	\end{array}\right|\in \oc(\ab\times U_{k}),
	\]
	which corresponds to  the Pl\"ucker coordinate of ${\rm Pluc}\circ \Phi_{\!\uk}$. 
		By \Cref{dI}, one has \[\omega_{\mathfrak{U},I_1,\ldots,I_k}(a_{I_1},\ldots, a_{I_k})\cdot \tau_{U_k}^{r(I_1+\cdots+I_k)}=	\left|\begin{array}{ccc}
		\onu^1(p^*a_{I_1}\cdot{(p^*\tau)}^{(r+k)I_1})&\cdots & \onu^1(p^*a_{I_k}\cdot{(p^*\tau)}^{(r+k)I_k})\\
		\vdots &\ddots& \vdots \\
		\onu^k(p^*a_{I_1}\cdot{(p^*\tau)}^{(r+k)I_1})&\cdots & \onu^k(p^*a_{I_k}\cdot{(p^*\tau)}^{(r+k)I_k})
		\end{array}\right|.		
		\]
		It follows from \Cref{prop:DSWronskian} that under the trivialization of $\uk$, one has
		\[ \omega_{\log}(a_{I_1}\tau^{(r+k)I_1},\ldots,a_{I_k}\tau^{(r+k)I_k})_{\upharpoonright U_k}=\omega_{\mathfrak{U},I_1,\ldots,I_k}(a_{I_1},\ldots, a_{I_k})\cdot \tau_{U_k}^{r(I_1+\cdots+I_k)}\cdot \bm{\Gamma}_{k,\uk}, \]
		here \(\bm{\Gamma}_{k,\uk}\in\oc(U_k)\) is the holomorphic function defining \(\bm{\Gamma}_{k}\) via the trivialization of $\uk$.
 By \eqref{eq:two wronskian} we conclude the proof of the lemma.
\end{proof}
\begin{rem}
By the proof of \cref{lem:two rational}, one can glue \(\omega_{\mathfrak{U},I_1,\ldots,I_k}(a_{I_1},\ldots, a_{I_k})\)  together to obtain a global section 
	\begin{align}\label{eq:new wronskian}
	\omega'_{\log,I_1,\ldots,I_k}( a_{I_1},\ldots, a_{I_k})\in H^0\big({{\lbb}}_k,\oc_{{{\lbb}}_k}(k,k-1,\dots,1)\otimes (p\circ{\pi}_{0,k})^*A^{k(\varepsilon+k\delta)}\big)
	\end{align}
	such that 
	\[\omega_{\log,I_1,\ldots,I_k}( a_{I_1},\ldots, a_{I_k})=\omega'_{\log,I_1,\ldots,I_k}( a_{I_1},\ldots, a_{I_k})\cdot\bm{\Gamma}_k. \]
	It follows from \eqref{eq:two exceptional} that 
\begin{align}\label{eq:pull-back}
\nu_k^*\omega'_{\log,I_1,\ldots,I_k}( a_{I_1},\ldots, a_{I_k})=F'\cdot \widetilde{\omega}_{\log,I_1,\ldots,I_k}( a_{I_1},\ldots, a_{I_k}). 
\end{align}
	\end{rem}
Consider the following rational map
\begin{align}\label{eq:map to grass}
 \widetilde{\Phi}: \ab\times \widetilde{\lbb}_k   &\dashrightarrow   \textnormal{P}\big(\Lambda^{k}(\cb^{\mathbb{I}})\big)\\\nonumber
 (\af,\tilde{w})&\mapsto  \big[\big(\widetilde{\omega}_{\log,I_1,\ldots,I_k}(a_{I_1},\ldots,a_{I_k})(\tilde{w})\big)_{I_1,\ldots,I_k\in \ib}\big].
\end{align}
Since $\widetilde{\lbb}_k\setminus{\rm Supp}(F)\xrightarrow{\nu_k}{{\lbb}}_k\setminus {\rm Supp}(\oc_{{{\lbb}}_k}/{\wk}_{k,\lbb,Y})$ is a isomorphism,  one has
\(\widetilde{\Phi}=\Phi\circ \nu_k  \) outside \({\rm Supp}(F)\), and by the fact that \(\widetilde{\lbb}_k\) is irreducible, this implies
that $\widetilde{\Phi}$ also factors through the Pl\"ucker embedding, which is also denoted by
 \(\widetilde{\Phi}\).   One thus has the following commutative diagram
	\begin{displaymath}
	\xymatrix{
 \ab\times	\widetilde{\lbb}_k \ar[d]_{\mathds{1}\times \nu_k } \ar@{-->}[dr]^{\widetilde{\Phi}}  &  \\
	 \ab\times	{\lbb}_k          \ar@{-->}[r]^-{\Phi}     & {\rm Gr}_{k}(\cb^{\ib})  }
	\end{displaymath} 

\subsection{Partially resolving the  indeterminacy}\label{partial resolving}

In this subsection, we will find a local and linear description for \(\widetilde{\Phi}\), and use this to prove that \(\nu_k\) \emph{partially} resolves the indeterminacies of rational map \(\widetilde{\Phi}\) in the same spirit as \cite[Lemmata 3.6 \& 3.7]{Bro17}.
 
\begin{lem}\label{partial}
	Fix any $\varepsilon\geqslant k$ and any $N>n$. For any $\tilde{w}_0\in \widetilde{{\lbb}}_k$,  there exists an open neighborhood $\widetilde{U}_{\tilde{w}_0}$ of $\tilde{w}_0$ such that we can define $\cb$-linear maps
	$$
	\ell_{I}^p: H^0(Y,A^{\varepsilon})\rightarrow \oc(\widetilde{U}_{\tilde{w}_0})
	$$
 for any $I\in  \ib$ and $p=1,\ldots,k$ satisfying the following conditions. 
	\begin{thmlist}
		\item \label{linear-claim 1} Write $\ell_{\bullet}^p(\af,\tilde{w}):=\big(\ell_{I}^p(a_I)(\tilde{w})\big)_{I\in \ib}\in\cb^{\ib}$.  The Pl\"ucker coordinates of $\widetilde{\Phi}(\af,\tilde{w})$ in $ {\rm Gr}_{k}(  \cb^{\ib})$   all vanishes if and only
		$$
		{\rm dim\ Span}\big(\ell_{\bullet}^1(\af,\tilde{w}),\ldots, \ell_{\bullet}^k(\af,\tilde{w}) \big)<k.
		$$
		\item \label{linear-claim 2} When $
{\rm dim\ Span}\big(\ell_{\bullet}^1(\af,\tilde{w}),\ldots, \ell_{\bullet}^k(\af,\tilde{w}) \big)=k
		$, one has
		$$
		\widetilde{\Phi}(\af,\tilde{w})=\big(\ell_{\bullet}^1(\af,\tilde{w}),\ldots, \ell_{\bullet}^k(\af,\tilde{w}) \big).
		$$
		\item \label{linear-claim 3}
		Set $y:=p\circ{\pi}_{0,k}\circ\nu_k(\tilde{w}_0)\in Y$ and define $\rho_{y}:(\cb^{{\ib}})^{k}\rightarrow (\cb^{\ib_{y}})^{k}$ to be the natural projection map, where 
		$$\ib_{y}:=\{I\in \ib\mid \tau^I(y)\neq0 \}.$$
		 Define a linear map 
		\begin{align}\label{linear map:non-truncated}
		\widetilde{\varphi}_{\tilde{w}_0}:\ab&\rightarrow  (\cb^{{{\ib}}})^{k}\\\nonumber
		\af&\mapsto \big(\ell_{\bullet}^1(\af,\tilde{w}_0),\ldots, \ell_{\bullet}^k(\af,\tilde{w}_0) \big).
		\end{align}
		 Then one has
\begin{align}\label{eq:rank}
	{\rm rank\ } \rho_{y}\circ\widetilde{\varphi}_{\tilde{w}_0}=k\cdot \#\ib_{y},
\end{align}
		where $\#\ib_{y}$ denotes to be the cardinality of $\ib_{y}$.
	\end{thmlist}
\end{lem}

\begin{proof}
	Set   $w_0:=\nu_k(\tilde{w}_0)$, $x_0={\pi}_{0,k}(w_0)$ and thus $y=p(x_0)$.  Since $m:=\varepsilon+(r+k)\delta\geqslant k$, by \cref{prop:ideal sheaf}   there exist $b_1,\ldots,b_k\in H^0(Y,A^m)$ such that $\tilde{\omega}_{\log}(b_1,\ldots,b_k)(\tilde{w})\neq 0$ on some neighborhood $\widetilde{U}_{{w}_0}$ of $\tilde{w}_0$ in $\widetilde{\lbb}_{k}$.  Pick a trivialization tower \(\uk\) of order \(k\) such that \(w_0\in U_k\). We shrink $\widetilde{U}_{{w}_0}$ such that \(\widetilde{U}_{{w}_0}\subset\nu_k^{-1}(U_k) \). For any $\sigma\in H^0(Y,A^{m})$, one has $\onu^i(\sigma)\in \oc({U_k})$, thus by abuse of notation, we also write \(\onu^i(\sigma)\)   as a holomorphic function on $\widetilde{U}_{\tilde{w}_0}$ under the pull-back $\nu_k:\widetilde{U}_{\tilde{w}_0}\rightarrow U_k$. 
	
It follows from \Cref{dI} that, for any $p=1,\ldots,k$, one can define
\[\omega'_{\log,p,I}(b_1,\ldots,b_{p-1},a_I,b_{p+1},\ldots,b_k) \in H^0\big({{\lbb}}_k,\oc_{{{\lbb}}_k}(k,k-1,\dots,1)\otimes (p\circ{\pi}_{0,k})^*A^{km-r\delta}\big)\]
such that 
\begin{align*}
\omega'_{\log} ( b_1,\ldots,b_{p-1},a_I\cdot \tau^{(r+k)I},b_{p+1},\ldots,b_k)=(p\circ{\pi}_{0,k})^*\tau^{rI}   \cdot \omega'_{\log,p,I}(b_1,\ldots,b_{p-1},a_I,b_{p+1},\ldots,b_k).
\end{align*}
Indeed, locally \(\omega'_{\log,p,I}(b_1,\ldots,b_{p-1},a_I,b_{p+1},\ldots,b_k)\) is defined by
\[\left|\begin{array}{ccccccc}
	{\nabla}^{1}_{\mathfrak{U}}(b_1)&\cdots & {\nabla}^{1}_{\mathfrak{U}}(b_{p-1}) &{\nabla}^{1}_{\mathfrak{U},I}(a_I)&	{\nabla}^{1}_{\mathfrak{U}}(b_{p+1})&\cdots & {\nabla}^{1}_{\mathfrak{U}}(b_{k})\\
	\vdots &\ddots & \vdots &\vdots& \vdots &\ddots&\vdots \\
	{\nabla}^{k}_{\mathfrak{U}}(b_1)&\cdots & {\nabla}^{k}_{\mathfrak{U}}(b_{p-1}) &{\nabla}^{k}_{\mathfrak{U},I}(a_I)&	{\nabla}^{k}_{\mathfrak{U}}(b_{p+1})&\cdots & {\nabla}^{k}_{\mathfrak{U}}(b_{k})
\end{array}\right|\]
By the relation between \(\omega_{\log}(\bullet)\) and \(\omega'_{\log}(\bullet)\) in \Cref{sec:intrinsic} and  similar arguments as \Cref{lem:ideal contain}, one deduces that
\[\omega'_{\log,p,I}(b_1,\ldots,b_{p-1},a_I,b_{p+1},\ldots,b_k) \in H^0\big({{\lbb}}_k,\oc_{{{\lbb}}_k}(k,k-1,\dots,1)\otimes (p\circ{\pi}_{0,k})^*A^{km-r\delta}\otimes  {\wk}'_{k,\lbb,Y}\big),\]
where ${\wk}'_{k,\lbb,Y}$ is the ideal sheaf of $\lbb_{k}$ defined in \Cref{sec:universal}. By \eqref{eq:pull back relation2}, there exists a unique holomorphic section
$$
\tilde{\omega}_{\log,p,I}(b_1,\ldots,b_{p-1},a_I,b_{p+1},\ldots,b_k)\in H^0\Big({\widetilde{\lbb}}_k,\nu_k^*\big(\oc_{{{\lbb}}_k}(k,k-1,\ldots,1)\otimes (p \circ {\pi}_{0,k})^*A^{km-r\delta}\big)\otimes  \oc_{ \widetilde{\lbb}_{k}}(-F')\Big)
$$
such that
$$
\nu_k^*\omega'_{\log,p,I}(b_1,\ldots,b_{p-1},a_I,b_{p+1},\ldots,b_k) =F'\cdot \tilde{\omega}_{\log,p,I}(b_1,\ldots,b_{p-1},a_I,b_{p+1},\ldots,b_k).
$$
On $\widetilde{U}_{\tilde{w}_0}$, within the trivialization of $\uk$, we now define
	\begin{align*}
		\ell_{I}^p(a_I):&=\frac{\nu_k^*\omega'_{\log,p ,I}(b_1,\ldots,b_{p-1},a_I,b_{p+1},\ldots,b_k)}{\nu_k^*\omega'_{\log}(b_1,\ldots,b_k)}\\
		&=\frac{F'\cdot \tilde{\omega}_{\log,p,I}(b_1,\ldots,b_{p-1},a_I,b_{p+1},\ldots,b_k) }{ F'\cdot \widetilde{\omega}_{\log}(b_1,\ldots,b_k) }\\
		&=\frac{\tilde{ \omega}_{\log,p ,I}(b_1,\ldots,b_{p-1},a_I,b_{p+1},\ldots,b_k)}{ \widetilde{\omega}_{\log}(b_1,\ldots,b_k) }
	\end{align*}
where the second equality is  due to  \eqref{eq:pull back relation2}.
	Hence \(\ell_{I}^p(a_I)\) are all holomorphic  functions over $\widetilde{U}_{\tilde{w}_0}$.
Consider the matrix of functions  $G(\tilde{w}) $ over \(\widetilde{U}_{\tilde{w}_0}\) defined by 
	$$
	G(\tilde{w}):=	\begin{pmatrix}
\onu^{1}(b_1)&\ldots,&\onu^{1}(b_k)\\
	\vdots &\ddots & \vdots\\
\onu^{k}(b_1)&\ldots&\onu^{k}(b_k)	
	\end{pmatrix},
	$$
	then	by definition, one has
	$$
	G\cdot	\begin{pmatrix}
	\ell_{I}^1(a_I)\\
	\vdots\\
	\ell_{I}^k(a_I)\\
	\end{pmatrix}
	=\frac{\det G}{\nu_k^*\omega'_{\log}(b_1,\ldots,b_k)}\cdot \begin{pmatrix}
		\nabla^{1}_{\!\!\uk,I}(a_I)\\
	\vdots\\
	\nabla^{k}_{\!\!\uk,I}(a_I)
	\end{pmatrix}=\begin{pmatrix}
	\nabla^{1}_{\!\!\uk,I}(a_I)\\
	\vdots\\
	\nabla^{k}_{\!\!\uk,I}(a_I)
	\end{pmatrix}.
	$$ 
	For any $I_1,\ldots,I_k\in {\ib}$, on $\widetilde{U}_{\tilde{w}_0}$ one has
\begin{align}\label{matrix relation}
	\begin{vmatrix}
\ell_{I_1}^1(a_{I_1}) &\ldots &	\ell_{I_k}^1(a_{I_k})\\
\vdots &\ddots &\vdots\\
\ell_{I_1}^k(a_{I_1})&\ldots&		\ell_{{I_k}}^k(a_{I_k})
\end{vmatrix}=\frac{ \nu_k^*{\omega}'_{\log,I_1,\ldots,I_k}( a_{I_1},\ldots, a_{I_k})}{\nu_k^*\omega'_{\log}(b_1,\ldots,b_k)}=\frac{ \widetilde{\omega}_{\log,I_1,\ldots,I_k}( a_{I_1},\ldots, a_{I_k})}{ \tilde{ \omega}_{\log}(b_1,\ldots,b_k)},
\end{align}
	where the last equality is due to \eqref{eq:pull back relation2} and \eqref{eq:pull-back}. This implies  \cref{linear-claim 1,linear-claim 2}.
	\medskip
	
	In order to prove   \cref{linear-claim 3}, we first observe that the linear map $\widetilde{\varphi}_{\tilde{w}_0}$ is   block with respect to $I\in \ib$. Thus set
	\begin{align*}
	\tilde{\varphi}_I:H^0(Y,A^{\varepsilon })&\rightarrow  \cb^{k}\\
	a_I&\mapsto \big(	\ell_{I}^1(a_I)(\tilde{w}_0),\ldots, 	\ell_{I}^k(a_I)(\tilde{w}_0) \big).
	\end{align*}
Note that $A^{\varepsilon }$ generates $k$-jets everywhere on $Y$ by the assumption that $\varepsilon \geqslant k$.  For  any $I\in \ib_{y}$, by the definition of \(\ib_{y}\) one has $\tau^I(y)\neq 0$, and  one can therefore take $c_1,\ldots, c_k\in H^0(Y,A^{\varepsilon }) $ such that the $k$-jets 
$$j^kc_i(y)=j^k(\frac{b_i }{\tau^{(r+k)I}})(y)$$ 
for any $i=1,\ldots,k$.
	Then 
	$$
	j^k b_i (y)= j^k( c_i\cdot {\tau^{(r+k)I}})(y)
	$$
for any $i=1,\ldots,k$.  It follows from Lemma \ref{k-jets} that
$$
\tilde{\omega}_{\log }(c_1\cdot{\tau}^{(r+k)I},\ldots, c_k\cdot{\tau}^{(r+k)I})(\tilde{w}_0)=\tilde{\omega}_{\log }( {b}_1,\ldots, {b}_k)(\tilde{w}_0).
$$
Hence 
	\begin{align*}
	\begin{vmatrix}
		\ell_{I}^1(c_1) &\ldots &	\ell_{I}^1(c_k)\\
 \vdots &\ddots&\vdots\\
		\ell_{I}^k(c_1)&\ldots&		\ell^k_{{I}}(c_k)
\end{vmatrix}(\tilde{w}_0)
&\stackrel{\eqref{matrix relation}}{=} \frac{ \widetilde{\omega}_{\log,I,\ldots,I}( c_1,\ldots, c_k)(\tilde{w}_0)}{\widetilde{\omega}_{\log}( {b}_1,\ldots, {b}_k)(\tilde{w}_0)}\\
&\stackrel{\eqref{eq:two wronskian}}{=}\frac{\widetilde{\omega}_{\log }(c_1\cdot{\tau}^{(r+k)I},\ldots, c_k\cdot{\tau}^{(r+k)I})(\tilde{w}_0)}{{\tau(y)^{krI}}\cdot \widetilde{\omega}_{\log }( {b}_1,\ldots, {b}_k)(\tilde{w}_0)} \\
	&=\frac{1}{ {\tau(y)^{krI}}}\neq 0.
	\end{align*}
This implies that ${\rm rank}\,  \widetilde{\varphi}_I =k$.
\cref{linear-claim 3} immediately  follows from  that $\widetilde{\varphi}_{\tilde{w}_0}:=\bigoplus_{I\in \ib}\widetilde{\varphi}_I$. We finish the proof of the whole lemma.
\end{proof}
Let us apply \cref{partial} to show that $\tilde{\Phi}(\af,\bullet):\widetilde{\lbb}_k\dashrightarrow {\rm Gr}_{k}( \cb^{\ib})$ is a regular morphism for general $\af\in \ab$   when we choose the parameters $N,\delta,k$ properly.

\begin{lem}\label{regular1}
	Assume that $N>n$, $\delta\geqslant (k+1)n+k$. Then there exists a Zariski dense open set $\ab_{\rm def}\subset \ab_{\rm sm}$ such that $\widetilde{\Phi} : \ab_{\rm def}\times \widetilde{\lbb}_k\rightarrow {\rm Gr}_{k}( \cb^{\ib})$ is a regular morphism.
\end{lem}
\begin{proof}
	By \eqref{eq:map to grass} the indeterminacy locus of $\widetilde{\Phi} $ is contained in the subvariety 
	$$
	Z:=\{(\af ,\tilde{w})\in \ab \times \widetilde{\lbb}_k\mid  \tilde{\omega}_{\log ,I_1,\ldots,I_k}( a_{I_1},\ldots, a_{I_k})(\tilde{w})=0 \ \forall I_1,\ldots,I_k\in \ib \}.
	$$
Denote ${\rm pr}_1:\ab \times \widetilde{\lbb}_k\rightarrow \ab$ and ${\rm pr}_2:\ab \times \widetilde{\lbb}_k\rightarrow \widetilde{\lbb}_k$ to be the projection maps. It then suffices to show that
	 $
	{\rm pr}_1(Z)\subsetneq \ab.
	$ 
	Fix any $\tilde{w}_0\in \widetilde{\lbb}_k$.  	Set $w_0:=\nu_k(\tilde{w}_0)$,  $x={\pi}_{0,k}(w_0)$ and $y=p(x_0)$. Define $Z_{\tilde{w}_0}:=Z\cap {\rm pr}_2^{-1}(\tilde{w}_0)$. For the linear map defined in  \cref{linear-claim 3},  it follows from \cref{linear-claim 1} that 
	$$
	{\rm pr}_1(Z_{\tilde{w}_0})=(\tilde{\varphi}_{\tilde{w}_0})^{-1}(\Delta),
	$$
	where
	$$
	\Delta:=\{ (v_{\bullet}^1,\ldots,v_{\bullet}^k)\in (\cb^{\ib})^{k} \mid {\rm dim\ Span}(v_{\bullet}^1,\ldots,v_{\bullet}^k)<k \}.
	$$
Define a linear subspace of $(\cb^{\ib_{y}})^{k}$ by
	$$
	\Delta_{y}:=\{ (v_{\bullet}^1,\ldots,v_{\bullet}^k)\in (\cb^{\ib_{y}})^{k} \mid {\rm dim\ Span}(v_{\bullet}^1,\ldots,v_{\bullet}^k)<k \},
	$$
	and one has 
 $
	\Delta\subseteq \rho_{y}^{-1}(\Delta_y)
	$,
	where 
	$\rho_{y}:(\cb^{{\ib}})^{k}\rightarrow (\cb^{\ib_{y}})^{k}$ is the natural projection map. 
Hence
	$$
	{\rm pr}_1(Z_{\tilde{w}_0})\subset (\rho_{y}\circ \tilde{\varphi}_{\tilde{w}_0})^{-1}\Delta_{y},
	$$
and
	\begin{align*}
	{\rm dim}\, Z_{\tilde{w}_0} &\leqslant  {\rm dim}\,  {\rm ker}(\rho_{y}\circ \tilde{\varphi}_{\tilde{w}_0})+{\rm dim}\, \Delta_{y}\\
	&= {\rm dim}\, \ab -{\rm rank\ }(\rho_{y}\circ \tilde{\varphi}_{\tilde{w}_0})+ {\rm dim}\, \Delta_{y}\\
	&\stackrel{\eqref{eq:rank}}{=}   {\rm dim}\, \ab -k\#\ib_{y}+(k-1)(\#\ib_{y}+1)\\
	&= {\rm dim}\, \ab +(k-1)-\#\ib_{y}.
	\end{align*}
Therefore,
	\begin{align*}
	{\rm dim}\,  Z &\leqslant  {\rm dim}\,  \widetilde{\lbb}_k +\max_{\tilde{w}_0\in \widetilde{\lbb}_k} {\rm dim}\,  Z_{\tilde{w}_0}\\
	&\leqslant  {\rm dim}\,  \widetilde{\lbb}_k +{\rm dim}\, \ab +(k-1) -\min_{y\in Y }\#\ib_{y}\\
	&\leqslant  (k+1)n+1+(k-1)+{\rm dim}\, \ab -(\delta+1)\\
	&< {\rm dim}\,  \ab .
	\end{align*}
Here we observe that $\#\ib_{y}\geqslant \binom{N-n+\delta}{\delta}\geqslant \delta+1$ for any $y\in Y$ when $N>n$.
Let us define $\ab_{\rm def}:=\big(\ab\setminus \overline{{\rm pr}_1(Z) }\big)\cap \ab_{\rm sm}$, which is a Zariski dense open set of $\ab$. By the definition of $Z$, we conclude that $\tilde{\Phi}(\af,\bullet):\widetilde{\lbb}_k\rightarrow {\rm Gr}_{k}( \cb^{\ib})$ is a regular morphism for any $\af\in \ab_{\rm def}$.
\end{proof}

\section{Proof of the main results}\label{sec:proof of main}
\subsection{Associated universal complete intersection variety} We are now in position to introduce the main geometric framework used during the proof of our main result.  As in \cite{BD15,Bro17,Den17,BD17} we rely on the universal complete intersection variety associated to our problem defined by
\[\ys:=\left\{(\Delta,[z])\in {\rm Gr}_{k}\Big(H^0\big(\pb^k,\oc_{\pb^k}(\delta)\big)\Big)\times \pb^k\ | \ \forall \ P\in \Delta,\ P([z])=0 \right\},\]
where we fix the parameter $N=k=n+1$ now.  Let us write \({\rm Gr}_k:={\rm Gr}_{k}\Big(H^0\big(\pb^k,\oc_{\pb^k}(\delta)\big)\Big) \)  for simplicity.
For technical reasons, we will also need to adapt this construction to the stratification on \(Y\) induced by the vanishings of the \(\tau_j\)'s. To do this, let us define for any $J\subset \{0,\ldots,k\}$, 
\begin{align*}
 \mathbb{P}_{J}&:=\{[z]\in \mathbb{P}^{k}\ |\  z_j=0 \mbox{ if } j\in J\},\\
	 Y_J&:=\{y\in Y\ |\  \tau_j(y)=0\Leftrightarrow j\in J   \},\\
	 \ib_J &:=\{I\in \ib \ |\ {\rm Supp}(I)\subset  \{0,\ldots,k\}\setminus J  \},
\end{align*}
and let us also consider the restricted \emph{universal  complete intersection varieties}
 \begin{align*}
 \ys_{J}:&=\{(\Delta,[z])\in {\rm Gr}_k\times \pb_{J}\ |\ \forall P\in \Delta, P([z])=0 \}\\
 &=\ys\cap({\rm Gr}_k\times \pb_J)\subset \ys.
 \end{align*}
 Let us denote by  \(\pr:\ys\to {\rm Gr}_k\) the canonical projection, and for any \(J\subset \{0,\dots, k\}\) we set \(\pr_J:=\pr_{\upharpoonright \ys_J}:\ys_J\to {\rm Gr}_k\). Observe that the \(\pr_J\) is \emph{generically finite}. This observation is crucial in the rest of the argument which highly rests on the understanding of the geometry of the \emph{non-finite locus} of \(\pr_J\):
 \[E_J:=\left\{y\in \ys_J\ | \ \dim_y\pr_J^{-1}(\pr_J(y))>0\right\},\]
 and its image in \({\rm Gr}_k\):
 \[G_J^{\infty}:=\pr_J(E_J)=\left\{\Delta\in {\rm Gr}_k\ | \ \dim\pr_J^{-1}(\Delta)>0\right\}.\]
\subsection{Factorization through the universal complete intersection variety} Let us now relate this universal complete intersection to our special families of Fermat type pairs constructed in the previous sections by considering the morphism
\begin{align}\label{eq:map to univ}
	\tilde{\Psi}:  \ab_{\rm def}\times \widetilde{\lbb}_k &\to {\rm Gr}_{k}\times \pb^k\\\nonumber
	(\af,\tilde{w})&\mapsto \Big(\tilde{\Phi}(\af,\tilde{w}), \big[\tau_0^r(p\circ {\pi}_{0,k}\circ \nu_k)(\tilde{w}),\ldots,\tau_k^r(p\circ {\pi}_{0,k}\circ \nu_k)(\tilde{w})\big]\Big).
\end{align}
For any \(J\subset\{0,\dots, k\}\) let us write,
\[\lbb_{k,J}:=(p\circ {\pi}_{0,k})^{-1}(Y_J)   \ \ \text{and}\ \ \widetilde{\lbb}_{k,J}:=\nu_k^{-1}(\lbb_{k,J}).\]
Recall the definition of the (restricted) families \(\hs_{k}\rightarrow \ab_{\rm def}\) and \(\widetilde{\hs}_{k}\rightarrow \ab_{\rm def}\) and denote, for any \(J\subset \{0,\dots, k\}\),
 \begin{eqnarray*}
\hs_{k,J}=\hs_{k}\cap (\ab_{\rm def} \times \widetilde{\lbb}_{k,J})\quad {\rm and} \quad \widetilde{\hs}_{k,J}=\widetilde{\hs}_{k}\cap (\ab_{\rm def} \times \widetilde{\lbb}_{k,J}). 
 \end{eqnarray*}
 For any \(\af\in \ab_{\rm def}\), let us denote by \(H_{\af,k,J}\) and \(\widetilde{H}_{\af,k,J}\) the fiber above \(\af\) of \(\hs_{k,J}\) and \(\widetilde{\hs}_{k,J}\) respectively. One then has the crucial factorizing property of $\widetilde{\Psi}$.
\begin{lem}\label{factorization}
For any  $J\subset \{0,\ldots,k\}$, when restricted to $\widetilde{\hs}_{k,J}$, the morphism $\widetilde{\Psi}$ factors through 
	${\ys}_{J}\subset {\rm Gr}_{k}\times \pb^k$. 
\end{lem}
\begin{proof}
	 It suffices to prove that \(\widetilde{\Psi}\) restricted to \(\widetilde{\lbb}_{k,J}\times \ab_{\rm def}\)
	factors through
	\({\rm Gr}_k\times  \pb_J\) and that
	\(\widetilde{\Psi}\) restricted to 
	\(\widetilde{\hs}_{k}\)
	factors through \(\ys\). The first claim is
	straightforward to prove. For the second one, since \(\widetilde{\Phi}=\Phi\circ\nu_k\), it suffices to prove that the rational map
\begin{align*}
{\Psi}:  \ab_{\rm def}\times \lbb_k &\dashrightarrow {\rm Gr}_k\times \pb^k\\
(\af,w)&\mapsto \bigg( {\Phi}( \af,w),\Big[\tau_0^r\big((p\circ {\pi}_{0,k} )( {w})\big),\dots,\tau_k^r\big((p\circ {\pi}_{0,k} )( {w})\big)\Big]\bigg)
\end{align*}
factors through \(\ys\) when restricted to \(\hs_k\). 

Let us take a  trivialization tower \(\mathfrak{U}\) of order \(k\) as in \cref{sec:intrinsic}.  Pick any \(\af\in \ab_{\rm def}\).   Recall that $H_\af$ is defined by   the vanishing of  the section
$$
T-p^*\sigma(\af)=T-	 p^*(\sum_{|I|=\delta }a_I\cdot \tau ^{(r+k)I})\in  H^0(\lbb,p^*L).
$$
Then over $H_{\af,k}\cap U_k$,    for any $i=1,\ldots,k$ one has
$$
0=\onu^i\big(T-p^*\sigma(\af)\big)\stackrel{\eqref{eq:VanishingNabla}}{=}-\onu^i\ p^*\sigma(\af)= - \sum_{|I|=\delta }(\tau_{U_k})^{rI} \cdot {\nabla}_{\!\!\uk,I}^{j}(a_I),
$$
where the last equality is due to  \cref{dI},  and we denote by $\tau_{U_k}$ the pull-back of trivialization of $\tau$ under $L_{\upharpoonright U}$ to $U_k$.  By the alternative definition of $\Phi$ in  \cref{lem:two rational}, we conclude the first claim. 
The second claim of the lemma follows directly from \Cref{regular1}.
\end{proof}

\subsection{An effective Nakamaye type result}
Let us denote by   \(\pr_1,\pr_2\) the projection on the first and second factor of \({\rm Gr}_k\times \pb^k\),  and let us consider the Pl\"ucker line bundle \(\ls\) on \({\rm Gr}_k\). By definition, one has \(\ls:={\rm Pluc}^*\oc_{\text{P}(\Lambda^k(\cb^{\ib}))}(1)\) where \({\rm Pluc}:{\rm Gr}_k\hookrightarrow \text{P}\big(\Lambda^k(\cb^{\ib})\big)\) denotes the Pl\"ucker embedding. The use of the universal complete intersection in our situation is justified by the following formula: for any \(d\in \mathbb{N}\), one has
\begin{align}\label{pull-back}
			\widetilde{\Psi}^*\big(\ls^d\boxtimes  \oc_{\pb^{k}}(-1)\big)=\nu_k^*\big(\oc_{\lbb_k}(dk')\otimes (p\circ{\pi}_{0,k})^*A^{dk(\varepsilon +k\delta)-r}\big)\otimes \oc_{\widetilde{\lbb}_k}(-dF).
	\end{align}
where we write \( \ls^d\boxtimes  \oc_{\pb^{k}}(-1):={\rm pr}_1^*\ls^d\otimes {\rm pr}_2^*\oc_{\pb^{k}}(-1)\). In particular, if \(r>dk(\varepsilon +k\delta)\), every global section of \( \ls^d\boxtimes  \oc_{\pb^{k}}(-1) \) gives rise to a logarithmic jet differentials vanishing along an ample divisor of  \(Y\). Of course, there may not exist  such global sections due to the presence of the negative twist \(\pr_2^*\oc_{\pb^k}(-1)\). However, observe that  the line bundle \(\ls\) is ample on \({\rm Gr}_k\) and the projection \(\pr_J:\ys_J\to {\rm Gr}_k\) is generically finite, therefore, \(\pr_J^*\ls=\pr^*\ls_{\upharpoonright \ys_J}\) is big and nef for any \(J\subset \{0,\dots, k\}\) and therefore, for \(d\) large enough, there are many global sections of the line bundle \(\ls^d\boxtimes  \oc_{\pb^{k}}(-1)_{\upharpoonright \ys_J}\) . In view of the factorization property established in the previous section, we obtain that for any \(J\subset \{0,\dots, k\}\),  any integer \(m\) and any \(\af\in \ab_{\rm def}\),
\begin{eqnarray*}
	{\rm Bs}\Big(\nu_k^*\big(\oc_{H_{\af,k,J}}(dk')\otimes (p\circ{\pi}_{0,k})^*A^{dk(\varepsilon +k\delta)-r}\big)\otimes \oc_{\widetilde{H}_{\af,k,J}}(-dF)\Big)\subseteq		\widetilde{\Psi}^{-1}\Big({\rm Bs}\big(\ls^d\boxtimes  \oc_{\pb^{k}}(-1)_{\upharpoonright \ys_J}\big)\Big).
	\end{eqnarray*}
These considerations lead us to study the right hand side in this formula. Since \(\pr_J^*\ls\) is big and nef for any \(J\subset \{0,\dots, k\}\), Nakamaye's theorem \cite{Nak00}  on the augmented base locus guaranties that 
\[E_J=\mathbf{B}_+(\pr_J^*\ls)={\rm Bs}\big(\ls^d\boxtimes  \oc_{\pb^{k}}(-1)_{\upharpoonright \ys_J}\big)\]
for \(d\) large enough. To determine an explicit bound for the values of \(d\) satisfying this formula is critical in order to obtain an effective bound on the degree in our main theorem. While we don't know a bound for this exact problem, the second named author was able in \cite{Den17}  to obtain the following bound for a slightly weaker inclusion sufficient for our purposes. 
\begin{thm}[\!\!{\cite{Den17}}]\label{nakmaye type}
	For any  $d\geqslant  \delta^{k-1}$, and any $J\subset\{0,\ldots,k\}$, the  base locus of the line bundle $ \ls^d\boxtimes  \oc_{\pb^{k}}(-1)_{\upharpoonright \ys_J}$ satisfies
	\begin{align}\label{nakmaye}
	{\rm Bs}\big(\ls^d\boxtimes  \oc_{\pb^{k}}(-1)_{\upharpoonright \ys_J}\big)\subset \pr_J^{-1}(G_{J}^{\infty}).
	\end{align}
Recall that $G^\infty_J$ is the set of points in ${\rm Gr}_k$ such that the fiber of  $\pr_J:\ys_J\rightarrow {\rm Gr}_k$ is not a finite set.
\end{thm}

\subsection{Avoiding the exceptional locus}
In this section we explain how one can control \(\widetilde{\Psi}^{-1}\big(\pr_J^{-1}(G_J^{\infty})\big)\).

\begin{lem}\label{exceptional locus}
	For any $J\subset \{0,\ldots,k\}$, when  $\delta\geqslant n(k+1)+1$,  there exists a non-empty Zariski open subset $\ab_{J}\subset \ab_{\rm def}$ such that
	$$
\widetilde{\Phi}^{-1}(G^\infty_{J})\cap (\ab_{J}\times \widetilde{\lbb}_{k,J})=\varnothing.
	$$
\end{lem}
\begin{proof}
	Take any $\tilde{w}_0\in \widetilde{\lbb}_{k,J}$.  Set $w_0:=\nu_k(\tilde{w}_0)$,  $x={\pi}_{0,k}(w_0)$ and $y=p(x_0)$.  Then we have $\ib _J=\ib _{y}$, and  we define the following analogues of $\ys$ parametrized by affine spaces
	\begin{align*} 
	\tl{\ys}_J:=\{\big(P_1,\ldots,P_k,[z]\big)\in  (\cb^{\ib } )^{k}\times \pb_{ J} \  \big\lvert \ P_1([z])=\cdots=  P_k([z])=0  \},\\ 
	\tl{\ys}_{y}:=\{\big(P_1,\ldots,P_k,[z]\big)\in (\cb^{\ib_J } )^{k} \times \pb_{ J}  \  \big\lvert \ P_1([z])=\cdots=  P_k([z])=0 \}.
	\end{align*}
	Here we use the identification $\cb^{\ib }\cong H^0\big(\pb^k,\oc_{\pb^k}(\delta)\big)$ and  $\cb^{\ib_J }\cong H^0\big(\pb_{J},\oc_{\pb_{J}}(\delta )\big)$. By analogy with $G_J^\infty$, we denote by $V_{1,J}^\infty$ (resp. $V_{2,J}^\infty$) the set of points in $(\cb^{\ib } )^{k}$ (resp. $(\cb^{\ib_J } )^{k}$) at which the fiber in $\tl{\ys}_J$ (resp. $\tl{\ys}_{y}$) is positive dimensional. 
	
	We take the linear map $\tilde{\varphi}_{\tilde{w}_0} :\ab\rightarrow (\cb^{\ib } )^{k}$ 
	\begin{align}\nonumber
	\tilde{\varphi}_{\tilde{w}_0}:\ab&\rightarrow  (\cb^{\ib})^{k}\\\nonumber
	\af&\mapsto  \big(\ell_{\bullet}^1(\af,\tilde{w}_0),\ldots, \ell_{\bullet}^k(\af,\tilde{w}_0) \big)
	\end{align}
	defined  in Lemma \ref{partial} so that,   for any $\af\in \ab_{\rm def}$, we have
	$$
	\tilde{\Phi}(\af,\tilde{w}_0)={\rm Span}\big(\ell_{\bullet}^1(\af,\tilde{w}_0),\ldots, \ell_{\bullet}^k(\af,\tilde{w}_0) \big).
	$$
Then we have
	$$
	\tilde{\Phi}^{-1}(G^\infty_J)\cap (\ab_{\rm def} \times \{\tilde{w}_0\})= \tilde{\varphi}_{\tilde{w}_0}^{-1}(V_{1,J}^\infty)\cap \ab_{\rm def}=(\rho_{y}\circ \tilde{\varphi}_{\tilde{w}_0})^{-1}(V_{2,J}^\infty)\cap \ab_{\rm def}.
	$$
 Recall that we have $\ib _J=\ib _{y}$. 
	Therefore
	\begin{align*}
		{\rm dim}\, \big(\tilde{\Phi}^{-1}(G_J^\infty)\cap (\ab_{\rm def}\times \{\tilde{w}_0\})\big)&\leqslant  {\rm dim}\,  \big((\rho_{y}\circ \tilde{\varphi}_{\tilde{w}_0})^{-1}(V_{2,J}^\infty)\big)\\
		&\leqslant  {\rm dim}\,   V_{2,J}^\infty +{\rm dim}\ker\, (\rho_{y}\circ\tilde{\varphi}_{\tilde{w_0}})\\
		&\leqslant {\rm dim}\,   V_{2,J}^\infty +{\rm dim}\,  \ab-{\rm rank}\, (\rho_{y}\circ\tilde{\varphi}_{\tilde{w_0}})\\
		&= {\rm dim}\,   V_{2,J}^\infty +{\rm dim}\,  \ab-k\#\ib_{y}.
	\end{align*}
	Since
	\begin{align*}
		{\rm dim}\,  (V_{2,J}^\infty)&= {\rm dim}\,   (\cb^{\ib_{J} } )^{k} -{\rm codim}\,  \big(V_{2,J}^\infty,(\cb^{\ib_{y} } )^{k}\big)\\
		&= k\#\ib_J-{\rm codim}\,  \big(V_{2,J}^\infty,(\cb^{\ib_{y} } )^{k}\big)\\
		&= k\#\ib_{y}-{\rm codim}\,  \big(V_{2,J}^\infty,(\cb^{\ib_{y} } )^{k}\big),
	\end{align*}
by putting the above inequalities together, one obtains
	$$
	{\rm dim}\,  \big(\tilde{\Phi}^{-1}(G_J^\infty)\cap \ab_{\rm def}\times \{\tilde{w}_0\}\big)\leqslant {\rm dim}\,   \ab -{\rm codim}\,  (V_{2,J}^\infty,(\cb^{\ib_{y} } )^{k}),
	$$
	which yields
	$$
	{\rm dim}\,  \big(\tilde{\Phi}^{-1}(G_J^\infty)\cap \ab_{\rm def}\times \widetilde{\lbb}_{k,J}\big)\leqslant {\rm dim}\,   \ab -{\rm codim}\,  (V_{2,J}^\infty,(\cb^{\ib_{y} } )^{k})+{\rm dim}\,  \widetilde{\lbb}_{k,J}.
	$$
	By a result due to   Benoist (see \cite{Ben11} or \cite[Corollary 3.2]{BD15}), we have
	$$
	{\rm codim}\,  \big(V_{2,J}^\infty,(\cb^{\ib_{y} } )^{k}\big)\geqslant   \delta+1.
	$$
	Therefore, if 
	\begin{equation}\label{avoid} 
	{\rm dim}\,  \widetilde{\lbb}_{k,J}\leqslant (k+1)n+1 < \delta +1,
	\end{equation}
	$\tilde{\Phi}^{-1}(G_J^\infty)$ doesn't dominate $\ab_{\rm def}$ via the projection $\ab_{\rm def}\times\widetilde{\lbb}_{k,J}\rightarrow \ab_{\rm def}$, and hence there exists a non-empty Zariski open subset $\ab_J\subset \ab_{\rm def}$ such that
 	\begin{align*}
\tilde{\Phi}^{-1}(G_J^\infty)\cap (\ab_J\times \widetilde{\lbb}_{k,J} )=\varnothing. 
\end{align*} 
\end{proof}

\subsection{Proof of the logarithmic Kobayashi conjecture} We are now in position to conclude the proof of our first main result. With the notation of \cref{exceptional locus}, set 
\[\ab_{\rm nef}:=\bigcap_{J\subset\{0,\dots,k\}}\ab_J,\]
which is a non-empty Zariski open subset $ \ab_{\rm def} $  by \cref{exceptional locus}. Fix $\delta= (k+1)n+k$ so that the conditions in \cref{regular1,exceptional locus} are fulfilled.
\begin{thm}\label{TheoremIntermediaire} Same notation as above. 
	For any $\af\in \ab_{\rm nef}$, the line bundle 
		$$
\nu_k^*\big(\oc_{{\lbb}_k}(\delta^{k-1}k')\otimes (p\circ{\pi}_{0,k})^*A^{\delta^{k-1}k(\varepsilon+k\delta)-r}\big)\otimes \oc_{\widetilde{\lbb}_k}(-\delta^{k-1}F)_{\upharpoonright \widetilde{H}_{\af,k}}
		$$
		is nef on $\widetilde{H}_{\af,k}$.
\end{thm}

\begin{proof}    It suffices to  show that for any   irreducible curve \(C\subset \widetilde{H}_{\af,k}\), one has 
\[C\cdot \nu_k^*\big(\oc_{{\lbb}_k}(\delta^{k-1}k')\otimes (p\circ{\pi}_{0,k})^*A^{\delta^{k-1}k(\varepsilon+k\delta)-r}\big)\otimes \oc_{\widetilde{\lbb}_k}(-\delta^{k-1}F)\geqslant 0.\]
By  \eqref{pull-back}  this is equivalent to 
\begin{equation}\label{eq:IntersectionGoal}C\cdot\widetilde{\Psi}^*\big(\ls^{\delta^{k-1}}\boxtimes  \oc_{\pb^{k}}(-1) \big)\geqslant 0.\end{equation}
Let \(J\subset \{0,\dots,k\}\) be such that \(\widetilde{H}_{\af,k,J}\) contains a non-empty open subset \(C^{\circ}\) of \(C\) (there exists a unique such \(J\)).  By the factorization property  in \cref{factorization}, one has 
$\widetilde{\Psi}(C)\subset \ys_J$. 
Moreover, by \cref{exceptional locus}, we see that \[\widetilde{\Psi}(C^{\circ})\cap E_J\subset \widetilde{\Psi}(C^{\circ})\cap \pr_J^{-1}({G}_J^{\infty})=\varnothing.\]
In particular, \(\widetilde{\Psi}(C)\not\subset \pr_J^{-1}(G_J^{\infty})\) and therefore it  follows from 	\cref{nakmaye type} applied to \(m=\delta^{k-1}\) that 
\[\widetilde{\Psi}(C)\not\subset {\rm Bs}\big(\ls^{\delta^{k-1}}\boxtimes   \oc_{\pb^{k}}(-1)_{\upharpoonright \ys_J}\big),\]
from which \eqref{eq:IntersectionGoal} follows at once.
\end{proof}
Observe that \cref{TheoremIntermediaire}  implies the following result.
\begin{cor}\label{cor:Intermediaire} Same notation as above. There exists \(\beta,\tilde{\beta}\in \nb\) such that for any \(\alpha\geqslant 0\), and for a general hypersurface \(D\in |A^{\varepsilon+(r+k)\delta}|\), denoting by \(Y_k(D)\) the  log Demailly $k$-jet tower associated to   \(\big(Y,D,T_{Y}(-\log D)\big)\),  the stable base locus 
\[\mathbf{B}\big(\oc_{Y_k(D)}(\beta+\alpha\delta^{k-1}k')\otimes {\pi}_{0,k}^*A^{\tilde{\beta}+\alpha(\delta^{k-1}k(\varepsilon +k\delta)-r)}\big)\subseteq Y_k(D)^{\rm sing}\cup{\pi}_{0,k}^{-1}(D).\] 
\end{cor}
\begin{proof}
Fix  any  \(\af\in \ab_{\rm nef}\).  Observe now that there exists \(\tilde{\beta},a_1,\dots,a_k,q\in \nb\) such that the line bundle 
\[\mu_{\af,k}^*\big(\oc_{{H}_{\af,k}}(a_1,\dots,a_k)\otimes {\pi}_{0,k}^*A^{\tilde{\beta}}\big)\otimes \oc_{\widetilde{H}_{\af,k}}(-qF)\]
is ample, where $\mu_{\af,k}:\widetilde{H}_{\af,k}\to H_{\af,k}$ is the blow-up of the  logarithmic Wronskian ideal sheaf $\wk_{{H}_{\af,k}}$. By \cref{TheoremIntermediaire} as well as  the functorial properties for the restriction of Wronskians in  \eqref{equ:fonctorial} and the blow-up of logarithmic Wronskian ideal sheaves in \cref{lem:fiberwise blow-up}, for any \(\alpha\in \nb\) the line bundle
\begin{align*}
\Big(\nu_{k}^*\big(\oc_{\lbb_k}(a_1,\dots,a_k+\alpha\delta^{k-1}k')\otimes (p\circ{\pi}_{0,k})^*A^{\tilde{\beta}+\alpha(\delta^{k-1}k(\varepsilon+k\delta)-r)}\big)\otimes \oc_{\widetilde{\lbb}_{k}}\big(-(q+\delta^{k-1})F\big) \Big)_{\upharpoonright \widetilde{H}_{\af,k}}=\\
\mu_{\af,k}^*\big(\oc_{{H}_{\af,k}}(a_1,\dots,a_k+\alpha\delta^{k-1}k')\otimes  L_\af^{\tilde{\beta}+\alpha(\delta^{k-1}k(\varepsilon+k\delta)-r)}\big)\otimes \oc_{\widetilde{H}_{\af,k}}\big(-(q+\delta^{k-1})F\big) 
\end{align*}
is ample, where $L_\af:=(p\circ \pi_{0,k})^*A_{\upharpoonright H_{\af,k}}$.   Recall that $\widetilde{H}_{\af,k}$ is a fiber of the smooth family  $\widetilde{\hs}_k^{\rm s}\rightarrow \ab^{\rm s}$, where $\widetilde{\hs}_k^{\rm s}$ is the functorial blow-up of the  \emph{universal family} of log Demailly towers of general log pairs defined in \cref{sec:Universal family}.  Since ampleness is an open condition in families \cite[Theorem 1.2.17]{Laz04I}, then by \cref{lem:fiberwise blow-up} again we conclude that there exists an non-empty Zariski open subset $\ab_{\rm amp}\subset \ab^{\rm s}$ such that for any $\sigma\in \ab_{\rm amp}$, the line bundle
\begin{align*}
\Big(\nu_{k}^*\big(\oc_{\lbb_k}(a_1,\dots,a_k+\alpha\delta^{k-1}k')\otimes (p\circ{\pi}_{0,k})^*A^{\tilde{\beta}+\alpha(\delta^{k-1}k(\varepsilon+k\delta)-r)}\big)\otimes \oc_{\widetilde{\lbb}_{k}}\big(-(q+\delta^{k-1})F\big) \Big)_{\upharpoonright \widetilde{H}_{\sigma,k}}=\\
\mu_{\sigma,k}^*\big(\oc_{{H}_{\sigma,k}}(a_1,\dots,a_k+\alpha\delta^{k-1}k')\otimes  L_\sigma^{\tilde{\beta}+\alpha(\delta^{k-1}k(\varepsilon+k\delta)-r)}\big)\otimes \oc_{\widetilde{H}_{\sigma,k}}\big(-(q+\delta^{k-1})F\big) 
\end{align*}
is ample. Here  $ {H}_{\sigma,k}$  is  the log  Demailly  $k$-jet tower of   $\big(H_\sigma,D_\sigma,T_{H_\sigma}(-\log D_\sigma)\big)$,   $L_\sigma:=(p\circ \pi_{0,k})^*A_{\upharpoonright H_{\sigma,k}}$, and  $\mu_{\sigma,k}:\widetilde{H}_{\sigma,k}\to H_{\sigma,k}$ denotes to be the blow-up of the $k$-th log Wronskian ideal sheaf $\wk_{H_{\sigma,k}}$.  By our construction of  $H_\sigma$, the log pairs \((H_{\sigma},D_\sigma)\) and \((Y,D_{\sigma})\) are isomorphic. Hence the line bundle 
\[\nu_k^*\big(\oc_{{Y}_{k}(D_\sigma)}(a_1,\dots,a_k+\alpha\delta^{k-1}k')\otimes {\pi}_{0,k}^*A^{\tilde{\beta}+\alpha(\delta^{k-1}k(\varepsilon+k\delta)-r)}\big)\otimes \oc_{\widetilde{Y}_{k}(D_\sigma)}(-(q+\delta^{k-1})F)\]
is ample as well, where we denote  by   \(\widetilde{Y}_{k}(D_\sigma)\)  the blow-up along the Wronskian ideal sheaf $\wk_{{Y}_k(D_\sigma)}$. In particular, its stable base locus is empty, which implies  that the stable locus  
\[\mathbf{B}\Big(\oc_{{Y}_{k}(D_\sigma)}(a_1,\dots,a_k+\alpha\delta^{k-1}k')\otimes {\pi}_{0,k}^*A^{\tilde{\beta}+\alpha(\delta^{k-1}k(\varepsilon+k\delta)-r)}\Big) \]
is contained in the cosupport of the logarithmic Wronskian ideal sheaf $\wk_{{Y}_k(D_\sigma)}$, which is  contained in \(Y_k(D_\sigma)^{\rm sing}\cup{\pi}_{0,k}^{-1}(D_\sigma)\) by \cref{prop:coincide}. Now it suffices to take \(\beta=a_1+\dots+a_k\) and apply the relation \eqref{eq:Gamma_k} to conclude that the stable base locus of the line bundle
\[\oc_{{Y}_{k}(D_\sigma)}(\beta+\alpha\delta^{k-1}k')\otimes {\pi}_{0,k}^*A^{\tilde{\beta}+\alpha(\delta^{k-1}k(\varepsilon+k\delta)-r)}\]
is also contained in  \(Y_k(D_\sigma)^{\rm sing}\cup{\pi}_{0,k}^{-1}(D_\sigma)\).  Recall that $\ab^{\rm s}$ is the Zariski open subset of  $\ab^{\rm u}:=H^0(Y,A^m)$ parameterizing all smooth hypersurfaces, where we recall $m:=\varepsilon+(r+k)\delta$.  $\ab_{\rm amp}$ therefore parametrizes a \emph{general}  hypersurface in  $|A^m|$,  whence the result.
\end{proof}
From \cref{cor:Intermediaire} the first statement of our main theorem follows at once. 
\begin{cor}\label{effective estimate}
	Let $Y$ be a projective manifold  of dimension $n\geqslant 2$, and  $A$ a very ample line bundle over $Y$. Then for any 
	\[m\geqslant (n^2+3n+1)^{n+3}\sim_{n\to \infty}e^3n^{2n+6},\]
	if $D\in |A^m|$ is a general smooth hypersurface, then $Y\setminus D$ is hyperbolically embedded in \(Y\).
\end{cor}
\begin{proof}
Recall first that a result of Green \cite{Gre77} guaranties that if \(D\) is a smooth hypersurface in \(Y\) such that \(D\) and \(Y\setminus D\) are both Brody hyperbolic, then \(Y\setminus D\) is hyperbolically embedded. Moreover, under the assumption of the corollary, it was established in \cite{Bro17,Den17}, that \(D\) is (Brody) hyperbolic, therefore it remains to prove that \(Y\setminus D\) is Brody hyperbolic. To see this we will just give an explicit bound on the degrees \(\varepsilon+(r+k)\delta\) covered by \cref{cor:Intermediaire}. Therefore recall that we have \(k=n+1\), and take \(\delta=(k+1)n+k=n^2+3n+1\) and set  $$r_0=\delta^{k-1}(\delta+1)^{2}=(n^2+3n+1)^{n}(n^2+3n+2)^{2}.$$
By the basic inequality
\begin{eqnarray}\label{basic}
k(k+\delta-1+k\delta)<(\delta+1)^2,
\end{eqnarray}
 one can show  that any \(m\geqslant (r_0+k)\delta+2\delta\) can be written in the form
\[m=\varepsilon+(r+k)\delta\]
with \(k\leqslant \varepsilon \leqslant k+\delta-1\), and \(r>\delta^{k-1}k(\varepsilon +k\delta)\).  In particular, applying \cref{cor:Intermediaire} for \(\alpha\) large enough and applying \cref{thm:fundamental},  we see that for   general hypersurface \(D\in |A^m|\),  $Y\setminus D$ are Brody hyperbolic. In order to obtain an explicit bound on \(m\)  it then suffices to give a bound on \((r_0+k)\delta+2\delta\): 
\begin{align*}
(r_0+k)\delta+2\delta&= \big(\delta^{k-1}(\delta+1)^{2}+k+2\big)\delta\\
&<  (n +2)^{n+3}(n+1)^{n+3}\sim_{n\to \infty} e^3n^{2n+6}. \qedhere
\end{align*} \end{proof} \subsection{Application to value distribution theory}
In this section, we show how   \cref{cor:Intermediaire} allows us to obtain a result in Nevanlinna theory. 
Let us recall the main definition used in Nevanlinna theory and refer the reader to the book \cite{NW14} for a detailed presentation. 
Let \(X\) be a   projective manifold and let \(A\) be an ample line bundle on \(X\) endowed with a smooth hermitian metric \(h\) whose   curvature tensor $\sqrt{-1}\Theta_{h,A}$ satisfies $\sqrt{-1}\Theta_{h,A}\geqslant \omega$ for some K\"ahler form $\omega$. For any entire curve $f:\cb\rightarrow X$, the \emph{Nevanlinna order function} is defined by
$$
T_f(r,A):=\int_{1}^{r}\frac{dt}{t}\int_{\Delta(t)}f^*(\sqrt{-1}\Theta_{h,A}),
$$
where $\Delta(t)$ is the disc of radius $t$ in $\cb$. For any simple normal crossing divisor $D$ such that $f$ is not contained in $D$, and for any \(k\in \nb^*\cup\{\infty\}\) one sets \[n^{(k)}_f(t,D):=\sum_{|z|<t} {\rm min}\big\{k,{\rm mult}_{z}(f^*D)\big\} \quad \text{for any \ }t>0,
\]
where \({\rm mult}_{z}(f^*D)\) denotes the multiplicity of \(f^*D\) at the point \(z\). For \(k=\infty\) one just writes \(n_f(t,D)\). One then defines the \emph{truncated counting function} at order $k$ by
\begin{eqnarray*}
	N^{(k)}_f(r,D):=\int_{1}^{r}n^{(k)}_f(t,D)dt\quad r>1
\end{eqnarray*}
In the case \(k=\infty\) we simply write \(N_f(r,D)\) and call it \emph{Nevanlinna's counting function}.
 One of the purpose of Nevanlinna theory is to compare the order function and the counting functions. If for instance \(D\in |A|\) then it is known that for all \(r\geqslant 0\), one has
 \[N_{f}(r,D)\leqslant T_f(r,A)+O(1),\]
 where \(O(1)\) is some bounded function. The so called ``Second Main Theorems''   are inequalities in the opposite direction of the form
 	\[T_f(r,K_X)+T_f(r,A)\leqslant N_f(r,D)+S_f(r) \quad \lVert,\]
	where \(S_f(r)\) is a small term compared to \(T_f(r,A)\), and where \( \lVert\) means that  the inequality holds outside a set of finite Lebesgue measure in \(\rb^+\).
Those inequalities are mainly conjectural and we refer to \cite{NW14} for a detailed account on the main known second main theorem type results. In the rest of this section we will consider the following weaker version of the Second Main Theorem, which consists in establishing inequalities of the form 
 \[T_f(f,A)\leqslant c N_f(r,D)+S_f(r)\quad \lVert\]
 for some constant  \(c\).  The  theory of jet differentials provides a direct way to produce such inequalities. 
 This relies mainly on  the lemma on \emph{logarithmic derivatives} and appears in several places in the literature more or less explicitly (see \emph{e.g.} \cite[Corollary 4.9]{Yam15}). Here we will apply the following precise statement recently established  in \cite[Theorem 3.1]{HVX17}. 
\begin{lem}\label{second} Let \((X,D)\) be a smooth logarithmic  pair, and let \(A\) be an ample line bundle on \(X\). For any positive integers \(k,N,N'\), for any global jet differential $P\in H^0\big(Y,E_{k,N}^{\rm GG}\Omega_X(\log D)\otimes A^{-N'}\big)$,  and for any entire curve \(f:\cb\rightarrow X\) which is not contained in \({\rm Supp} (D)\), if $f^*P\not\equiv 0$, then there exists a constant \(C\) such that
	\begin{eqnarray}
	T_f(r,A)\leqslant\frac{N}{N'}\cdot  N^{(1)}(f,D)+C\big(\log T_f(r,L)+\log r\big) \quad \lVert.
	\end{eqnarray}
Here the symbol \(\lVert\) means that	the  inequality holds outside a Borel subset of \((0,+\infty )\) of finite Lebesgue measure.
\end{lem}
Let us mention that in \cite{HVX17} the authors only state their result in the case \(X=\pb^n\), but this restriction is unnecessary. 

As an immediate consequence of \cref{second} we obtain the following.
\begin{cor}\label{cor:SMT} Let \((X,D)\) be a smooth log pair, and let \(A\) be an ample line bundle on \(X\). Let \( {\pi}_{0,k}: {X}_k(D)\to X\) be the log Demailly  tower associated to the pair \((X,D)\). For any positive integers \(k,N,N'\), if the stable base locus \[\mathbf{B}\big(\oc_{{X}_k(D)}(N)\otimes {\pi}_{0,k}^*A^{-N'}\big)\subseteq {X}_k(D)^{\rm sing}\bigcup{\pi}_{0,k}^{-1}(D)\]
	then, for any entire curve \(f:\cb\to X\) not contained in \({\rm Supp} (D)\), one has
	\begin{eqnarray}
	T_f(r,A)\leqslant\frac{N}{N'}\cdot  N^{(1)}(f,D)+C\big(\log T_f(r,A)+\log r\big) \quad \lVert.
	\end{eqnarray}
\end{cor}
It now suffices to combine this result with \cref{cor:Intermediaire} to obtain a  Second Main Theorem type result for general log pairs. 
\begin{cor}\label{effective estimate2}
	Let $Y$ be a projective manifold  of dimension $n\geqslant 2$, and  let $A$ be a very ample line bundle over $Y$. If  $D\in |A^m|$ is a general smooth hypersurface with  
	\[m\geqslant (n +2)^{n+3}(n+1)^{n+3}\sim_{n\to \infty}e^3n^{2n+6},\] then for any entire curve \(f:\cb\to Y\) not contained in \({\rm Supp}(D)\), there exists \(C\in \rb^+\) such that 
\[
	T_f(r,A)\leqslant N^{(1)}(f,D)+C\big(\log T_f(r,A)+\log r\big) \quad \lVert.
\]
\end{cor}
\begin{proof}
Let us take \(k=n+1\),  \(\delta=(k+1)n+k=n^2+3n+1\) and set 
\[r_0=\delta^{k-1}k'+\delta^{k-1}(\delta+1)^{2} =\delta^{k-1}(\delta+1)(\delta+\frac{3}{2}).\]
By \eqref{basic} one can prove that any \(m\geqslant(r_0+k)\delta+2\delta\) can be written in the form
\[m=\varepsilon+(r+k)\delta\]
with \(k\leqslant \varepsilon \leqslant k+\delta-1\), and \(r>\delta^{k-1}k'+\delta^{k-1}k(\varepsilon +k\delta)\).  In particular, applying \cref{cor:Intermediaire}  for \(\alpha\gg 0\) such that \(-\tilde{\beta}-\alpha(\delta^{k-1}k(\varepsilon +k\delta)-r)>0\) and applying \cref{cor:SMT} we see that for such \(m\),  general hypersurface \(D\in |A^m|\) and for any entire curve \(f:\cb\to Y\) not contained in \({\rm Supp}(D)\), there exists \(C\in \rb^+\) such that 
\[	T_f(r,A)\leqslant\frac{\beta+\alpha\delta^{k-1}k'}{{-\tilde{\beta}-\alpha(\delta^{k-1}k(\varepsilon +k\delta)-r)}}\cdot  N^{(1)}(f,D)+C\big(\log T_f(r,A)+\log r\big) \quad \lVert.\]
However,
\[\frac{\beta+\alpha\delta^{k-1}k'}{-\tilde{\beta}-\alpha(\delta^{k-1}k(\varepsilon +k\delta)-r)}\to_{\alpha\to \infty} \frac{\delta^{k-1}k'}{r-(\delta^{k-1}k(\varepsilon +k\delta)}<1. \]
Therefore, in order to complete the proof, it now suffices to give a bound on \((r_0+k)\delta+2\delta\): 
\begin{align*}
(r_0+k)\delta+2\delta&=\big(\delta^{k-1}(\delta+1)(\delta+\frac{3}{2})+k+2\big)\delta\\
&< (n +2)^{n+3}(n+1)^{n+3}\lesssim_{n\to \infty}  e^3n^{2n+6}.   \qedhere
\end{align*}
\end{proof}

 \subsection{Orbifold hyperbolicity and hyperbolicity for the cyclic cover}
The \emph{orbifold}   introduced by Campana arises naturally in his study of the birational classification of varieties in \cite{Cam04}. One can also generalize the definition of Kobayashi hyperbolicity and the tools of jet differentials to orbifolds, which were first studied by Rousseau in \cite{Rou10}. We refer the readers to the very recent paper \cite{CDR18} for the hyperbolicity and orbifold jet differentials in the orbifold category.  In this last section, we will apply   \cref{cor:Intermediaire} to prove the orbifold hyperbolicity for general orbifolds. From \cite{CDR18} one can easily derive  the following lemma.
\begin{lem}\label{orbifold}
	Let $Y$ be an $n$-dimensional   projective manifold, and let  $D$ be a smooth hypersurface of $Y$. Then for the Campana orbifold $\big(Y,\Delta):=(Y,(1-\frac{1}{m})D\big)$ where $m\in \nb^*$,  one has natural inclusions
	\begin{align*}
	E^{\rm GG}_{k,N}\Omega_Y(\log D)\otimes \oc_Y(-\lceil \frac{N}{m}\rceil D)\hookrightarrow E^{\rm GG}_{k,N}\Omega_{Y,\Delta}\hookrightarrow E^{\rm GG}_{k,N}\Omega_Y(\log D)
	\end{align*} 
	where $E^{\rm GG}_{k,N}\Omega_{Y,\Delta}$ is the \emph{orbifold jet differential} of degree $k$ and weight $N$ defined in \cite{CDR18}.
\end{lem}
\begin{proof}
	Take any  open subset of $U\subset Y$ with local coordinates $(z_1,\ldots,z_n)$ such that $D\cap U=(z_1=0)$. By \cref{lem:new basis}, for any $j\in \mathbb{N}$,  $\frac{d^{j}z_1}{z_1}$ is  a logarithmic jet differential and moreover,   $E^{\rm GG}_{k,N}\Omega_Y(\log D)$ is the locally free  sheaf   generated in local 	coordinates by elements
	\begin{align*} 
\left\{  \frac{1}{z_1^{\alpha_{1}^1+\cdots+\alpha_{k}^1}} (d^{1}z)^{\alpha_1}(d^{2}z)^{\alpha_2}\cdots (d^{k}z)^{\alpha_k}\right\}_{|\alpha_1|+2|\alpha_2|+\cdots+k|\alpha_k|=N},
	\end{align*}
	where $\alpha_j=(\alpha_j^1,\ldots,\alpha_j^n)\in \mathbb{N}^n$. 
By \cite[\S 2.3]{CDR18}  $E^{\rm GG}_{k,N}\Omega_{Y,\Delta}$  is the locally free subsheaf of $E_{k,N}\Omega_Y(\log D)$ generated in local 	coordinates  by elements
$$
\left\{  z_1^{\big\lceil\frac{\alpha_{1}^1{\rm min}(1,m)+\cdots+k\alpha_{k}^1{\rm min}(k,m)}{m}\big\rceil}\cdot \frac{1}{z_1^{\alpha_{1}^1+\cdots+\alpha_{k}^1}} (d^{1}z)^{\alpha_1}(d^{2}z)^{\alpha_2}\cdots (d^{k}z)^{\alpha_k}\right\}_{|\alpha_1|+2|\alpha_2|+\cdots+k|\alpha_k|=N} 
$$
The lemma then follows immediately from the obvious inequality
\[
\Big\lceil\frac{\alpha_{1}^1{\rm min}(1,m)+\cdots+k\alpha_{k}^1{\rm min}(k,m)}{m}\Big\rceil\leqslant \Big\lceil \frac{N}{m}\Big\rceil. \qedhere
\]
\end{proof}

Now let us  combine \cref{orbifold}  with \cref{cor:Intermediaire} to prove the orbifold hyperbolicity.
\begin{cor} \label{orbifold Kobayashi}
	Let $Y$ be a projective manifold  of dimension $n\geqslant 2$, and  $A$ a very ample line bundle over $Y$. Then for any 
	\[m\geqslant (n +2)^{n+3}(n+1)^{n+3}\sim_{n\to \infty}e^3n^{2n+6},\]
	if $D\in |A^m|$ is a general smooth hypersurface, 
	\begin{thmlist}
		\item  the orbifold $\big(Y,\Delta):=(Y,(1-\frac{1}{m})D\big)$ is orbifold hyperbolic. 
		\item \label{cyclic}For the cyclic cover $\pi:X\to Y$ obtained by taking the $m$-th root along $D$,  
		$X$ is Kobayashi hyperbolic.
		\end{thmlist}
\end{cor}
\begin{proof}
As in the proof of \cref{effective estimate2}, we take \(k=n+1\),   \(\delta=(k+1)n+k=n^2+3n+1\) and    \[r_0  =\delta^{k-1}(\delta+1)(\delta+\frac{3}{2}).\]
 Then by the computations therein  and \cref{cor:Intermediaire}, for any	general smooth hypersurface $D\in |A^m|$ with \[m\geqslant (n +2)^{n+3}(n+1)^{n+3},\]
 we can take $\alpha\gg 0$ so that
	\[
	N:=\beta+\alpha\delta^{k-1}k'< N':={-\tilde{\beta}-\alpha(\delta^{k-1}k(\varepsilon +k\delta)-r)}
	\]
with
		\[\mathbf{B}\big(\oc_{Y_k(D)}(N)\otimes {\pi}_{0,k}^*A^{-N'}\big)\subseteq Y_k(D)^{\rm sing}\cup{\pi}_{0,k}^{-1}(D).\] 
By 		 \eqref{eq:DirectImageFormula}, for any $\ell \gg 0$ one has \emph{sufficiently many} log jet differentials in $E_{k,\ell N}\Omega_{Y}(\log D)\otimes \oc_Y(-\ell N' A)$ in the sense that,   for any germ of curve $\gamma:(\cb,0)\to (Y\setminus D, y)$ whose $k$-jet $j_k\gamma(0)\neq 0$, there always exists a logarithmic jet differential 	
	$$
	P\in H^0\big( Y,E_{k,\ell N}\Omega_{Y}(\log D)\otimes \oc_Y(-\ell N' A)\big)
	$$
	with $P(j_k\gamma)(0)\neq 0$.
By the inclusive relation in \cref{orbifold}, one also has \emph{sufficiently many  orbifold jet differentials} in 
 $$
 E_{k,\ell N}\Omega_{Y,\Delta}(\log D)\otimes \oc_Y\big(- \ell N'A+ \lceil \frac{\ell N}{m}\rceil  D  \big)=E_{k,\ell N}\Omega_{Y,\Delta}(\log D)\otimes \oc_Y\Big(\big(- \ell N' + m\lceil \frac{\ell N}{m}\rceil\big)  A  \Big). $$  Take $\ell$ divisible enough  (\emph{i.e.} $m \mid \ell$) and one thus has
	$$
 \ell N'-m\lceil \frac{\ell N}{m}\rceil  >0
	$$	
	which implies the orbifold hyperbolicity of $(Y,\Delta)$ by the \emph{fundamental vanishing theorem} in the orbifold setting (cf.  \cite[Corollary 3.11]{CDR18}).   Hence the first claim is proved. 

	To prove the second statement, since $X$ is compact, it is equivalent to show that $X$ is also Brody hyperbolic. To prove this, we assume that there exists an entire curve $f:\cb\to X$ on $X$,  and the contradiction is derived immediately by  observing that   $\pi\circ  f: \cb\to Y$ is an orbifold entire curve with respect to the orbifold $(Y,\Delta)$, whereas $(Y,\Delta)$ is orbifold hyperbolic by the first claim. This proves the second claim.
\end{proof}

Let us mention that  in \cref{effective estimate,effective estimate2,orbifold Kobayashi} we made an approximation in order to give readable bound. In all cases, as is clear from the proof, we could obtain a slightly better bound. The fact that the  same bound  appears in \cref{effective estimate} and \cref{effective estimate2} is due to this approximation. In fact, our method would provide a slightly better bound in \cref{effective estimate} than that in \cref{effective estimate2,orbifold Kobayashi}.

\nomenclature[1]{$f:(\cb,0)\to X$}{Germ of holomorphic curve}%
\nomenclature[2]{$J_kX\rightarrow X$}{Fiber bundle of $k$-jets of germs of  holomorphic curves in $X$} 
\nomenclature[3]{$j_kf\in J_kX$}{ $k$-jet of germ of curve $f$ in $J_kX$}%

\nomenclature[4]{$J_k(X,\log D)$}{Fiber bundle of logarithmic $k$-jets of germs of  holomorphic curves in $X$}
\nomenclature[5]{$E_{k,m}^{\rm GG}\Omega_X, E_{k,m}^{\rm GG}\Omega_X(\log D)$}{Green-Griffiths bundle of (logarithmic) jet differentials of order $k$ and weight  $m$}
\nomenclature[6]{$(X,V)$}{Directed manifolds}
\nomenclature[6a]{$X_k$}{Demailly-Semple $k$-jet tower of the directed manifold $(X,V)$}
\nomenclature[7]{$(X, D,V)$}{Logarithmic directed manifold}
\nomenclature[8]{$E_{k,m}\Omega_X, E_{k,m} \Omega_X(\log D)$}{Invariant (logarithmic) jet bundle of order $k$ and weight  $m$} 
\nomenclature[9]{$X_k(D)$}{logarithmic Demailly(-Semple) $k$-jet tower associated to  logarithmic directed manifold $\big(X,D,T_X(-\log D)\big)$}
\nomenclature[9a]{$\pi_{0,k}:X_k(D)\to X$}{The natural projection map}
\nomenclature[a]{$J^kL$}{Jet bundle of a line bundle $L$ \eqref{def:jet bundle}}
\nomenclature[b]{$j_L^ks \in H^0(X,J^kL)$}{ $k$-jet of the holomorphic section $s\in H^0(X,L)$}
\nomenclature[c]{$W_L(\bullet)$}{ Wronskian  \eqref{def:wronskian} associated to the line bundle $L$}
\nomenclature[d]{$\nabla_{\!\!D}(\bullet)$}{logarithmic connection \eqref{def:log connection}}
\nomenclature[e]{$\nabla^k_{\!\!D}(\bullet)$}{logarithmic connection \eqref{def:higherorder}}


\nomenclature[f]{$W_{D}(\bullet)$}{Logarithmic  Wronskian associated to log pair $(X,D)$ \eqref{eq:defLogWronskian}}
\nomenclature[g]{$\omega_{D}(\bullet)$}{$({\pi}_{0,k})_*\omega_{D}(\bullet)=W_{D}(\bullet)$ \eqref{def:omega}}
\nomenclature[h]{$\onu^j(\bullet)$}{Higher order logarithmic connection in the trivialization tower $\uk$ \eqref{eq:trivial high}}

\nomenclature[i]{$\omega_{D}'(\bullet)$}{Defined in \eqref{eq:another wronskian} or \eqref{eq:relation of two wronskian}}
\nomenclature[j]{$ {\wk}_{{X}_k(D)}$ }{  $k$-th logarithmic Wronskian ideal sheaf  of the log manifold $(X,D)$}

\nomenclature[k]{$\lbb$}{The total space of the line bundle $A^{\otimes m}$ on $Y$}
\nomenclature[l]{$W_{{\lbb},Y} (\bullet)$}{Logarithmic  Wronskian associated to the log pair $(\lbb,Y)$}
\nomenclature[l1]{$\lbb_k$}{Log  Demailly $k$-jet tower of   log  directed manifold $\big(\lbb,Y,T_{\lbb}(-\log Y)\big)$}
\nomenclature[m]{$\omega_{\log}(\bullet)$}{$({\pi}_{0,k})_*\omega_{\log} (\bullet)=W_{{\lbb},Y} (\bullet)$}
\nomenclature[m1]{${\wk}_{k,{\lbb},Y}$, ${\wk}'_{k,{\lbb},Y}$}{Ideal sheaves of $\lbb_k$ in \eqref{eq:ideal relation}}
\nomenclature[m2]{$\nu_k:\widetilde{{\lbb}}_k\to \lbb_k$}{The blow-up of ${\wk}_{k,{\lbb},Y}$ and  ${\wk}'_{k,{\lbb},Y}$}
\nomenclature[n]{$(H_\sigma,D_\sigma)$}{The sub-log manifold of the log pair $(\lbb,Y)$ induced by $\sigma\in H^0(Y,A^m)$}
\nomenclature[n1]{$ H_{\sigma,k}$}{Log  Demailly tower of   log directed manifold   $\big(H_\sigma,D_\sigma,T_{H_\sigma}(-\log D_\sigma)\big)$}
\nomenclature[n2]{$\mu_{\sigma,k}:\widetilde{H}_{\sigma,k}\to H_{\sigma,k}$}{The blow-up of ${\wk}_{H_{\sigma,k}}$}
\nomenclature[o]{$\af\in \ab$}{Family of hypersurfaces in $Y$ parametrized by certain Fermat-type hypersurfaces defined in \cref{construction} }
\nomenclature[o1]{$ (\hs,\ds)\to \ab_{\rm sm}$}{Smooth family of sub-log pairs of $(\lbb,Y)$ induced by Fermat-type hypersurfaces}
\nomenclature[o2]{$\hs_k$}{Log  Demailly $k$-jet tower of   log  directed manifold $\big(\hs,\ds,T_{\hs/\ab_{\rm sm}}(-\log \ds)\big)$}
\nomenclature[o2]{$\omega_{\log,I_1,\ldots,I_k}(\bullet)$}{Modified logarithmic Wronskians \eqref{eq:two wronskian}}
\nomenclature[p]{$E^{\rm GG}_{k,N}\Omega_{Y,\Delta}$}{Orbifold jet differentials of order $k$ and weight $N$ associated to Campana orbifold $(Y,\Delta)$.}

\printnomenclature[1.7in]

\bibliographystyle{amsalpha}
\bibliography{biblio}

\end{document}